\documentclass[11.5pt,twoside,openany,UTF8,CJK]{article}
\usepackage{amssymb,amsthm}
\usepackage[tbtags]{amsmath}
\usepackage{cite}
\usepackage{indentfirst}
\usepackage{cases}
\usepackage{graphicx}
\usepackage{subfigure}
\usepackage{lineno}
\usepackage{enumerate}
\usepackage{threeparttable}
\usepackage{multirow}
\usepackage{booktabs}
\usepackage{upgreek}  
\usepackage{bm}
\usepackage{fancyhdr}
\usepackage{ifthen}
\usepackage{lipsum}
\allowdisplaybreaks[4]
\usepackage[misc]{ifsym}
\usepackage{caption}
\usepackage{float}
\usepackage{appendix}
\usepackage[format=hang]{caption}
\usepackage{algorithm,algorithmic}

\usepackage{charter}  
\usepackage{indentfirst}  
\usepackage{soul}
\usepackage{authblk}

\newcommand{\Rmnum}[1]{\uppercase\expandafter{\romannumeral #1}} 

\usepackage[colorlinks,
linkcolor=blue,
anchorcolor=blue,
citecolor=red]{hyperref}

\numberwithin{equation}{section}

\newtheorem{Lemma}{Lemma}[section]
\newtheorem{Theorem}{Theorem}[section]
\newtheorem{Remark}{Remark}[section]

\newcounter{saveeqn}

\makeatletter
\def\@maketitle{%
	\newpage
	\null
	\vskip 2em%
	\begin{center}%
		\let \footnote \thanks
		{\LARGE \@title \par}%
		\vskip 1.5em%
		{\large
			\lineskip .5em%
			\begin{tabular}[t]{c}%
				\@author
			\end{tabular}\par}%
	\end{center}%
	\par
	\vskip 1.5em}
\makeatother

\makeatletter
\evensidemargin
\oddsidemargin
\makeatother

\title{\bf A new decoupled unconditionally stable scheme and its optimal error analysis for the Cahn-Hilliard-Navier-Stokes equations}
\author{Haijun Gao$^{1}$, Xi Li$^{2}$, and Minfu Feng\footnote{Corresponding author. e-mail:~fmf@scu.edu.cn,~lixi@cdut.edu.cn,~gaohijun@163.com.}}
\affil[1]{\small School of Mathematics, Sichuan University, Chengdu, Sichuan 610064, China}
\affil[2]{College of Mathematics and Physics, Chengdu University of Technology, Chengdu, Sichuan 610059, China}
\date{}

\begin{document}
	\maketitle
	\newcommand\blfootnote[1]{%
		\begingroup
		\renewcommand\thefootnote{}\footnote{#1}%
		\par\setlength\parindent{2em}
		\endgroup
	}
	\captionsetup[figure]{labelfont={bf},labelformat={default},labelsep=period,name={Fig.}}
	\captionsetup[table]{labelfont={bf},labelformat={default},labelsep=period,name={Tab.}}
	
	\begin{abstract}
		We construct a decoupled, first-order, fully discrete, and unconditionally energy stable scheme for the Cahn-Hilliard-Navier-Stokes equations.
		The scheme is divided into two main parts. The first part involves the calculation of the Cahn-Hilliard equations, and the other part is calculate the Navier-Stokes equations subsequently by utilizing the phase field and chemical potential values obtained from the above step.
		Specifically, the velocity in the Cahn–Hilliard equation is discretized explicitly at the discrete time level, which enables the computation of the Cahn–Hilliard equations is fully decoupled from that of Navier–Stokes equations. 
		 Furthermore, the pressure-correction projection method, in conjunction with the scalar auxiliary variable approach not only enables the discrete scheme to satisfy unconditional energy stability, but also allows the convective term in the Navier-Stokes equations to be treated explicitly. 
		We subsequently prove that the time semi-discrete scheme is unconditionally stable and analyze the optimal error estimates for the fully discrete scheme. 
		 Finally, several numerical experiments validate the theoretical results.
		\\
		
	\end{abstract}
	
	\baselineskip 15pt
	\parskip 10pt
	\setcounter{page}{1}
	\vspace{-0.5cm}
	\section{Introduction}
	We propose the time semi-discrete scheme and fully discrete finite element methods for the following Cahn-Hilliard-Navier-Stokes (CHNS) equations
	\cite{2006_FengXiaobing_FullydiscretefiniteelementapproximationsoftheNavierStokesCahnHilliarddiffuseinterfacemodelfortwophasefluidflows}, for $t\in (0,T]$:
	\begin{subequations}\label{eq_chns_equations}
		\begin{align}
			\label{eq_phi_con}
			\frac{\partial \phi }{\partial t}+(\mathbf{u}\cdot\nabla)\phi= M\Delta \mu,\quad \text{in} \quad\Omega\times (0,T],&\\
			\label{eq_mu_con}
			\mu=-\lambda\Delta \phi+\lambda F'(\phi),\quad \text{in} \quad\Omega\times (0,T],&\\
			\label{eq_ns_con}
			\frac{\partial \mathbf{u}}{\partial t}+(\mathbf{u}\cdot\nabla)\mathbf{u}-\nu\Delta\mathbf{u}+\nabla p=\mu\nabla\phi,\quad\text{in} \quad\Omega\times (0,T],&\\
			\label{eq_incompressi}
			\nabla\cdot\mathbf{u}=0,\quad\text{in} \quad\Omega\times (0,T].&
		\end{align}
	\end{subequations}
	Here $\Omega\subset\mathbb{R}^d,~(d=2,3)$ is a bounded convex polygonal or polyhedral domain, 
	and $F(\phi)=\frac{1}{4\epsilon^2}\left(\phi^2-1\right)^2$, with the following boundary and initial conditions of \eqref{eq_chns_equations}:
	\begin{subequations}
		\label{eq_boundary_initial_conditions_equations}
		\begin{align}
			\label{eq_boundary_conditions_equation}
			\frac{\partial \phi}{\partial \mathbf{n}}=\frac{\partial\mu}{\partial \mathbf{n}}=0,\quad\mathbf{u}=0,&\quad\text{on}\quad \partial\Omega\times (0,T],\\
			\label{eq_initial_conditions_equation}
			\phi(\mathbf{x},0)=\phi^0,\quad\mathbf{u}(\mathbf{x},0)=\mathbf{u}^0,&\quad\text{in}\quad \Omega.
		\end{align}
	\end{subequations}
	In the model, 
	$\phi$ represents the phase field variable, and $\mu$ denotes the chemical potential.
	$\mathbf{u}$ and $p$ are the velocity and pressure of the fluid, respectively. 
	$\epsilon$ is the interfacial width between the two phases, $M$, $\lambda$, $\nu$ represent the mobility constant, 
	the mixing coefficient and the fluid viscosity, respectively. 
	The system \eqref{eq_chns_equations}-\eqref{eq_boundary_initial_conditions_equations} 
	features a nonlinear system consisting of the incompressible Navier–Stokes (NS) equations \cite{Temam_Roger_Navier_Stokes_equations_Theory_and_numerical_analysis} coupled with
	the Cahn-Hilliard (CH) equations  \cite{1958_Cahn_Hilliard_Free_Energy_of_a_Nonuniform_System_I_Interfacial_Free_Energy}, and
	describes the interfacial dynamics of two-phase, incompressible and macroscopically immiscible Newtonian fluids 
	with density-matched\cite{2007_Kay_David_and_Welford_Richard_Efficient_numerical_solution_of_Cahn_Hilliard_Navier_Stokes_fluids_in_2D,
		2008_Kay_David_and_Welford_Richard_Finite_element_approximation_of_a_Cahn_Hilliard_Navier_Stokes_system}.
	For the CHNS system  \eqref{eq_chns_equations}-\eqref{eq_boundary_initial_conditions_equations} , the energy and mass are defined by 
	\begin{equation}
		E(\phi,\mathbf{u})=\int_{\Omega}\left(\frac{1}{2}|\mathbf{u}|^2+\frac{\lambda}{2}|\nabla\phi|^2+\lambda F(\phi)\right) dx, \quad M(\phi)=(\phi,1).
	\end{equation}
	 We can obtain the energy dissipation law and mass conservation in time $t\in [0,T]$ as follows,
	\begin{equation}
		\frac{\partial E(\phi,\mathbf{u})}{\partial t}=-M\|\nabla \mu\|^2-\nu\|\nabla\mathbf{u}\|^2,\quad \left(\phi(t),1\right)=\left(\phi^0,1\right),
	\end{equation}
	respectively.
	\par
Over the past two decades, the main numerical discretization methods for the CH equations include the convex-splitting technique  \cite{1998_Eyre_David_J_Unconditionally_gradient_stable_time_marching_the_Cahn_Hilliard_equation,2011_WangC_An_energy_stable_and_convergent_finite_difference_scheme_for_the_modified_phase_field_crystal_equation,
			2015_ShenJie_Decoupled_energy_stable_schemes_for_phase_field_models_of_two_phase_incompressible_flows} and scalar auxiliary variable (SAV) methods \cite{2018_Shenjie_Xujie_SAV,
			2018_Shenjie_Xujie_CAEAFTSAVSYGF,2019_LinLianlei_Numerical_approximation_of_incompressible_Navier_Stokes_equations_based_on_an_auxiliary_energy_variable,
			2020_LiXiaoli_On_a_SAV_MAC_scheme_for_the_Cahn_Hilliard_Navier_Stokes_phase_field_model_and_its_error_analysis_for_the_corresponding_Cahn_Hilliard_Stokes_case_,
			2022JieShen_LiXiaoliMSAVCHNStwo_phase_incompressible_flows}, etc.
		For the NS equations \cite{2022_AnRong_Temporal_error_analysis_of_a_new_Euler_semi_implicit_scheme_for_the_incompressible_Navier_Stokes_equations_with_variable_density}, the projection methods \cite{2006_Guermond_ShenJie_An_overview_of_projection_methods_for_incompressible_flows}
		and the SAV  methods \cite{2020_lixiaoli_New_SAV_MAC_NS,2024_Superconvergence_error_analysis_of_linearized_semi_implicit_bilinear_constant_SAV_finite_element_method_for_the_time_dependent_Navier_Stokes_equations} are efficient approachs for decoupling the velocity and pressure for the NS equations.
	For the CHNS model, numerical analysis in the semi-discrete and fully discrete schemes has been done extensively; see 
	\cite{2006_FengXiaobing_FullydiscretefiniteelementapproximationsoftheNavierStokesCahnHilliarddiffuseinterfacemodelfortwophasefluidflows,
		2008_Kay_David_and_Welford_Richard_Finite_element_approximation_of_a_Cahn_Hilliard_Navier_Stokes_system,
		2013_Grun_On_convergent_schemes_for_diffuse_interface_models_for_two_phase_flow_of_incompressible_fluids_with_general_mass_densities,
		2015_Diegel_Analysis_of_a_mixed_finite_element_method_for_a_Cahn_Hilliard_Darcy_Stokes_system,
		2017_Diegel_Convergence_analysis_and_error_estimates_for_a_second_order_accurate_finite_element_method_for_the_Cahn_Hilliard_Navier_Stokes_system,
		2018_CaiYongyong_Error_estimates_for_a_fully_discretized_scheme_to_a_Cahn_Hilliard_phase_field_model_for_two_phase_incompressible_flows,
		2020_YangXiaofeng_Error_Analysis_of_a_Decoupled_Linear_Stabilization_Scheme_for_the_Cahn_Hilliard_Model_of_Two_Phase_Incompressible_Flows}.
	Feng \cite{2006_FengXiaobing_FullydiscretefiniteelementapproximationsoftheNavierStokesCahnHilliarddiffuseinterfacemodelfortwophasefluidflows} proposed a fully discrete finite element method for the first time, which satisfies the discrete energy dissipation law, and analyzed the convergence of numerical results to a weak solution for the CHNS model. 
	 Kay et al. \cite{2008_Kay_David_and_Welford_Richard_Finite_element_approximation_of_a_Cahn_Hilliard_Navier_Stokes_system} obtained 
	an optimal $H^1$-norm error estimate for the semi-discrete divergence-free type linear finite element methods and demonstrated the convergence for the fully discrete scheme.
	Gr\"{u}n \cite{2013_Grun_On_convergent_schemes_for_diffuse_interface_models_for_two_phase_flow_of_incompressible_fluids_with_general_mass_densities} 
	presented a fully discrete convex-splitting finite element method for a diffuse interface model.
	Feng et al. \cite{2015_Diegel_Analysis_of_a_mixed_finite_element_method_for_a_Cahn_Hilliard_Darcy_Stokes_system} proposed an unconditionally energy stable and uniquely solvable scheme based on the convex splitting finite element method for the Cahn-Hilliard-Darcy-Stokes equations.
	Diegel et al. \cite{2017_Diegel_Convergence_analysis_and_error_estimates_for_a_second_order_accurate_finite_element_method_for_the_Cahn_Hilliard_Navier_Stokes_system} construct a second-order in time convex-splitting finite element scheme for the CHNS system. 
	Soon after, Cai et al. \cite{2018_CaiYongyong_Error_estimates_for_a_fully_discretized_scheme_to_a_Cahn_Hilliard_phase_field_model_for_two_phase_incompressible_flows} presented an error analysis of a fully discrete standard inf-sup stable finite element method based on a stabilization technique and a projection method for decoupling the velocity and pressure. Xu et al. \cite{2020_YangXiaofeng_Error_Analysis_of_a_Decoupled_Linear_Stabilization_Scheme_for_the_Cahn_Hilliard_Model_of_Two_Phase_Incompressible_Flows} proposed a first-order in time, fully decoupled scheme
	for solving the CHNS problem which is constituted three steps and a stabilization term is added to satisfy the energy stable. 
	\par
	Recently, Cai et al. \cite{2023_CaiWentao_Optimal_L2_error_estimates_of_unconditionally_stable_finite_element_schemes_for_the_Cahn_Hilliard_Navier_Stokes_system}  concerned with the numerical analysis of CHNS and provided, for the first  time,
	 the optimal 
	$L^2$ error estimates of the coupled, convex-splitting finite element fully discrete scheme. Yang et al. \cite{2024_ChenYaoyao_SAVFEM} presented a coupled fully discrete numerical scheme of the inf-sup stable and lowest-order finite element for CHNS based on the SAV approach and proved the optimal $L^2$ error estimate. Then, they employed the SAV method to achieve decoupling between CH and NS and similarly derived the optimal $L^2$ error estimate for the decoupled system \cite{2024_YiNianyu_Convergence_analysis_of_a_decoupled_pressure_correction_SAV_FEM_for_the_Cahn_Hilliard_Navier_Stokes_model}.
	For more decoupling methods specific to the CHNS equations, refer to \cite{2007_Kay_David_and_Welford_Richard_Efficient_numerical_solution_of_Cahn_Hilliard_Navier_Stokes_fluids_in_2D,
		2016_ChenYing_Efficient_adaptive_energy_stable_schemes_for_the_incompressible_Cahn_Hilliard_Navier_Stokes_phase_field_models,
		2017_CaiYongyong_Error_estimates_for_time_discretizations_of_Cahn_Hilliard_and_Allen_Cahn_phase_field_models_for_two_phase_incompressible_flows,
		2020_JiaHongen_A_novel_linear_unconditional_energy_stable_scheme_for_the_incompressible_Cahn_Hilliard_Navier_Stokes_phase_field_model,
		2020_YangXiaofeng_Error_Analysis_of_a_Decoupled_Linear_Stabilization_Scheme_for_the_Cahn_Hilliard_Model_of_Two_Phase_Incompressible_Flows,
		2021_ZhaoJia_Second_order_decoupled_energy_stable_schemes_for_Cahn_Hilliard_Navier_Stokes_equations,
		2022JieShen_LiXiaoliMSAVCHNStwo_phase_incompressible_flows}.
	However, we observe that most of these papers do not analyze the  optimal $L^2$ error estimates and especially, to the best of our konwledge, there are few literatures concerning the optimal $L^2$ error estimates of fully decoupled schemes. 
	
	In this manuscript, we propose a new fully decoupled and unconditional energy stable numerical scheme for the CHNS model \eqref{eq_chns_equations}, and derive the optimal $L^2$ error estimates.
	Firstly, the CH equations decouple from NS equations by explicitly treating the velocity in the CH equations. 
	Then, to ensure the unconditional energy stability of this decoupled scheme, we introduce the SAV technique by adding the \textit{“zero-energy-contribution”} term \cite{2021_Yang_Xiaofeng_A_new_efficient_fully_decoupled_and_second_order_time_accurate_scheme_for_Cahn_Hilliard_phase_field_model_of_three_phase_incompressible_flow,
		2023_LiYibao_Consistency_enhanced_SAV_BDF2_time_marching_method_with_relaxation_for_the_incompressible_Cahn_Hilliard_Navier_Stokes_binary_fluid_model}.
	We prove the unconditional energy stability of the temporal semi-discrete scheme and analyze the optimal $L^2$ error estimates of the fully discrete mixed finite element method. 
	Numerical examples validate our theoretical results, including the optimal $L^2$ error estimates and energy stability.
	  
	This manuscript is organized as follows. The section \ref{section_preliminaries} introduces some basic notations and the Ritz and Stokes quasi-projections. 
	In section \ref{section_the_numerical_schemes}, we present the discrete scheme for the CHNS model and prove the unconditional energy is stable and uniquely solvable in the time semi-discrete scheme. Then, we propose fully discrete mixed finite element methods that are based on the pressure-correction and SAV techniques.
	In section \ref{section_error_analysis}, we provide the corresponding error estimates and prove optimal $L^2$ error estimates for the fully discrete scheme. Numerical experiments are presented in section \ref{section_numerical_experiments} to verify the theoretical results of the proposed method.
	\section{Preliminaries}\label{section_preliminaries}
	In this section, we introduce some notations for the phase field models of two-phase incompressible flows and the decoupled numerical scheme.
	\subsection{Some basic notations}
	Firstly, we introduce some basic notations. For any integer $k\geq0$ and real number $0\leq p\leq \infty$,
	we denote by $W^{k,p}(\Omega)$ the Sobolev space.
	The norms of $L^2$ and $\mathbf{L}^2$ denote by $\|\cdot\|$.
	Let $
		H^k(\Omega)=W^{k,2}(\Omega), ~L^p(\Omega)=W^{0,p}(\Omega),
	$
	and $
		L_0^p(\Omega)=\{f\in L^p(\Omega):\int_{\Omega}f~\text{dx}=0\}.
	$
	The closure of $C_0^{\infty}(\Omega)$ in $W^{k,p}$ space is denoted by $W_0^{k,p}$. 
	Let $H_0^k(\Omega)=W_0^{k,2}(\Omega),~\mathring{H}^1(\Omega)=H^1(\Omega)\cap L_0^2(\Omega),$ 
	and $\mathbf{L}^p(\Omega)=[L^p(\Omega)]^d,~\mathbf{H}_0^1(\Omega)=[H_0^1(\Omega)]^d, ~\mathbf{W}^{k,p}(\Omega)=[W^{k,p}(\Omega)]^d,$
	and denote its the dual space by $H^{-1}(\Omega)$.
	We define the trilinear forms $b(\cdot,\cdot,\cdot)$ and $\boldsymbol{B}(\cdot,\cdot,\cdot)$ by
	\begin{equation}
		\label{eq_trilinear_forms}
		b(\boldsymbol{u},\boldsymbol{v},\boldsymbol{w})=\left(\nabla\boldsymbol{u}\cdot\boldsymbol{v},\boldsymbol{w}\right),~
		\boldsymbol{B}(\boldsymbol{u},\boldsymbol{v},\boldsymbol{w})=\frac{1}{2}\int_{\Omega}\left(\left(\boldsymbol{u}\cdot\nabla\right)\boldsymbol{v}\cdot\boldsymbol{w}-
			\left(\boldsymbol{u}\cdot\nabla\right)\boldsymbol{w}\cdot\boldsymbol{v}\right)~d\boldsymbol{x}.
	\end{equation}
	It is obvious that the trilinear form $\boldsymbol{B}(\cdot,\cdot,\cdot)$ is skew-symmetric, i.e.
	\begin{equation}
		\boldsymbol{B}(\boldsymbol{u},\boldsymbol{v},\boldsymbol{w})=-\boldsymbol{B}(\boldsymbol{u},\boldsymbol{w},\boldsymbol{v}),~
		\boldsymbol{B}(\boldsymbol{u},\boldsymbol{v},\boldsymbol{v})=0,~\forall\boldsymbol{u},~
		\boldsymbol{v},\boldsymbol{w}\in\mathbf{H}_0^1(\Omega).
	\end{equation}
	\subsection{Ritz and Stokes quasi-projections}
	According to \cite{2023_CaiWentao_Optimal_L2_error_estimates_of_unconditionally_stable_finite_element_schemes_for_the_Cahn_Hilliard_Navier_Stokes_system}, 
	we introduce the Ritz and Stokes quasi-projections. 
	Let $\Im_h$ denotes a quasi-uniform partition of $\Omega$ into triangles $\mathcal{K}_j~(j=1,\cdots,K)$, in $\mathbb{R}^2$ or tetrahedra in 
	$\mathbb{R}^3$ with mesh size $h=\max_{1\leq j\leq K}\{\text{diam}~\mathcal{K}_j\}$. We define the following finite element spaces for any integer $r \geq 1$,
	\begin{equation}
		\begin{aligned}
			& S_h^r=\left\{v_h \in C(\Omega):\left.v_h\right|_{\mathcal{K}_j} \in P_r\left(\mathcal{K}_j\right), \forall \mathcal{K}_j \in \Im_h\right\}, \\
			& \mathring{S}_h^r=S_h^r \cap L_0^2(\Omega), \\
			& \mathbf{X}_h^{r+1}=\left\{\mathbf{v}_h \in \mathbf{H}_0^1(\Omega)^d:\left.\mathbf{v}_h\right|_{\mathcal{K}_j} \in \mathbf{P}_{r+1}\left(\mathcal{K}_j\right)^d, \forall \mathcal{K}_j \in \Im_h\right\}, 
		\end{aligned}
	\end{equation}
	where $P_r\left(\mathcal{K}_j\right)$ is the space of polynomials of degree $r$ on $\mathcal{K}_j$.
	Note that the spaces $\mathbf{X}_h^{r+1} \times S_h^r$, $(r \geq 1)$ 
	is the generalized Taylor-Hood 
	elements.
	Define $R_h:H^1(\Omega)\rightarrow S_h^r$ 
	as the classic Ritz projection \cite{1973_Wheeler_Mary_Fanett_A_priori_L2_error_estimates_for_Galerkin_approximations_to_parabolic_partial_differential_equations},
	\begin{equation}\label{eq_Ritz_projection}
		\left(\nabla\left(\psi-R_h\psi\right),\nabla\varphi_h\right)=0,\quad\forall\varphi_h\in S_h^r,
	\end{equation}
	where $\int_{\Omega}\left(\psi-R_h\psi\right)dx=0$.
	Now, from the finite element approximation theory \cite{2008_Brenner_Susanne_C_The_mathematical_theory_of_finite_element_methods}, 
	it holds that
	\begin{flalign}
		\label{eq_psi_Rhpsi_Ls_norm_inequation}
		&\|\psi-R_h\psi\|_{L^s}+h\|\psi-R_h\psi\|_{W^{1,s}}\leq Ch^{r+1}\|\psi\|_{W^{r+1,s}},\\
		\label{eq_psi_Rhpsi_Hneg1_norm_inequation}
		&\|\psi-R_h\psi\|_{H^{-1}}\leq C\mathcal{E}_h\|\psi\|_{H^{r+1}},\\
		\label{eq_Dtau_psi_Rhpsi_L2_norm_inequation}
		&\|\delta_\tau\left(\psi^n-R_h\psi^n\right)\|+h\|\delta_\tau\left(\psi^n-R_h\psi^n\right)\|_{H^1}\leq Ch^{r+1}\|\delta_\tau\psi^{n}\|_{H^{r+1}},\\
		\label{eq_Dtau_psi_Rhpsi_Hneg1_norm_inequation}
		&\|\delta_\tau\left(\psi^n-R_h\psi^n\right)\|_{H^{-1}}\leq C\mathcal{E}_h\|\delta_\tau\psi^n\|_{H^{r+1}},
	\end{flalign}
	for $s\in[2,\infty]$ and $n=1,2,\cdots,N$, where $\delta_\tau$ is defined in \eqref{eq_d_t}, and $\mathcal{E}_h$ is defined by
	\begin{equation}	
		\label{eq_mathcal_E_h}
		\mathcal{E}_h=
				 \begin{cases}
				 	h^{r+1}, & \text { for } r=1,\\
				 	h^{r+2}, & \text { for } r \geq 2.
				\end{cases}
	\end{equation}
	Let $I_h:L^2(\Omega)\rightarrow S_h^r $  and $\mathbf{I}_h:\mathbf{L}^2(\Omega)\rightarrow\mathbf{X}_h^{r+1}$
	are defined the $L^2$ projection operators by
	$
		\left(v-I_hv,w_h\right)=0,~\forall w_h\in S_h^r$,
		and $\left(\mathbf{v}-\mathbf{I}_h\mathbf{v},\mathbf{w}_h\right)=0,~\forall \mathbf{w}_h\in \mathbf{X}_h^{r+1}.
	$
	It is known to all that the $L^2$ projection satisfies the following estimates \cite{2008_Brenner_Susanne_C_The_mathematical_theory_of_finite_element_methods},
	\begin{equation}
		\label{eq_L2_estimate01}
		\|\mathbf{I}_h\mathbf{u}\|_k\leq C(\Omega)\|\mathbf{u}\|_k, \quad k\leq2,
	\end{equation}
	Based on the Ritz projection, 
	Cai et al. \cite{2023_CaiWentao_Optimal_L2_error_estimates_of_unconditionally_stable_finite_element_schemes_for_the_Cahn_Hilliard_Navier_Stokes_system} 
	defined a Ritz quasi-projection, $\Pi_h: H^1(\Omega) \rightarrow S_h^r$ by
	\begin{equation}\label{eq_Ritz_quasi_projections_equation}
		\left(\nabla\left(\mu-\Pi_h \mu\right), \nabla w_h\right)+\left(\nabla\left(\phi-R_h \phi\right) \cdot \mathbf{u}, w_h\right)=0,\quad\forall w_h\in S_h^r,
	\end{equation}
	where $\int_{\Omega}\left(\mu-\Pi_h \mu\right)dx=0$, and they also proved the following convergence.
	\begin{Lemma}[\cite{2023_CaiWentao_Optimal_L2_error_estimates_of_unconditionally_stable_finite_element_schemes_for_the_Cahn_Hilliard_Navier_Stokes_system}]
		\label{Lemma_Ritz_quasi_projection}
		For the Ritz quasi-projection $\Pi_h$, it holds that 
		\begin{flalign}
			\label{eq_lemma_mu_L2_inequation}
			&\|\mu-\Pi_h\mu\|+h\|\nabla\left(\mu-\Pi_h\mu\right)\|
			\leq Ch^{r+1}\left(\|\mathbf{u}\|_{L^\infty}\|\phi\|_{H^{r+1}}+\|\mu\|_{H^{r+1}}\right),\\
			\label{eq_lemma_mu_Hneg1_inequation}
			&\|\mu-\Pi_h\mu\|_{H^{-1}}\leq C\mathcal{E}_h\left(\|\mathbf{u}\|_{W^{1,4}}\|\phi\|_{H^{r+1}}+\|\mu\|_{H^{r+1}}\right).
		\end{flalign}
	\end{Lemma}
	Moreover, they also defined a Stokes quasi-projection. That is,
	for $\left(\mathbf{v}_h,q_h\right)\in\mathbf{X}_h^{r+1}\times\mathring{S}_h^r$, the Stokes quasi-projection $P_h$ 
	or $\mathbf{P}_h:\mathbf{H}_0^1(\Omega)\times L_0^2(\Omega)\rightarrow \mathbf{X}_h^{r+1}\times\mathring{S}_h^r$ 
	are defined by
	\begin{flalign}
		\label{eq_stokes_quasi_proojection_equation0001}
		&\left(\nabla\left(\mathbf{u}-\mathbf{P}_h(\mathbf{u},p)\right),\nabla \mathbf{v}_h\right)-\left(p-P_h(\mathbf{u},p),\nabla\cdot\mathbf{v}_h\right)=\left(\mu\nabla\left(\phi-R_h\phi\right),\mathbf{v}_h\right),\\
		\label{eq_stokes_quasi_proojection_equation0002}
		&\left(\nabla\left(\mathbf{u}-\mathbf{P}_h(\mathbf{u},p)\right),q_h\right)=0.
	\end{flalign}
	To be the simplicity of notations, we denote $P_hp=P_h(\mathbf{u},p)$ and $\mathbf{P}_h\mathbf{u}=\mathbf{P}_h(\mathbf{u},p)$.
	Thus, the two related estimates follow from the lemmas as indicated below.
	\begin{Lemma}[\cite{2023_CaiWentao_Optimal_L2_error_estimates_of_unconditionally_stable_finite_element_schemes_for_the_Cahn_Hilliard_Navier_Stokes_system}]
		For the Stokes quasi-projection defined in \eqref{eq_stokes_quasi_proojection_equation0001}-\eqref{eq_stokes_quasi_proojection_equation0002}, it holds that
		\begin{equation}
			\label{eq_stokes_quasi_proojection_inequation0005}
			\begin{aligned}
				& \left\|\mathbf{u}-\mathbf{P}_h \mathbf{u}\right\| \leq \begin{cases}C h^{r+2}\left(\|\mathbf{u}\|_{H^{r+2}}+\|p\|_{H^{r+1}}+\|\phi\|_{H^{r+1}}\|\mu\|_{H^2}\right), & \text { for } r \geq 2, \\
					C h^{2}\left(\|\mathbf{u}\|_{H^{2}}+\|p\|_{H^1}+\|\phi\|_{H^{2}}\|\mu\|_{H^2}\right), & \text { for } r=1,\end{cases} \\
			\end{aligned}
		\end{equation}
		and
		\begin{equation}
			\label{eq_stokes_quasi_proojection_inequation0006}
			\begin{aligned}
				\left\|\nabla\left(\mathbf{u}-\mathbf{P}_h \mathbf{u}\right)\right\|+\left\|p-P_h p\right\| 
				\leq C h^{r+1}\left(\|\mathbf{u}\|_{H^{r+2}}+\|p\|_{H^{r+1}}+\|\phi\|_{H^{r+1}}\|\mu\|_{W^{1,4}}\right), 
			\end{aligned}
		\end{equation}
		and
		\begin{equation}
			\|\delta_\tau\left(\mathbf{u}^n-\mathbf{P}_h\mathbf{u}\right)\|\leq
			\begin{cases}
				&Ch^{r+2}\left(\|\delta_\tau\mathbf{u}^n\|_{H^{r+2}}+\|\delta_\tau p^n\|_{H^{r+1}}+\|\delta_\tau\phi^n\|_{H^{r+2}}\|\mu^n\|_{H^2}\right.\\
				&\left.\qquad+\|\phi^{n-1}\|_{H^{r+1}}\|\delta_\tau\mu^n\|_{H^2}\right), ~\text{for}~ r\geq 2,\\
				&Ch^{2}\left(\|\delta_\tau\mathbf{u}^n\|_{H^{2}}+\|\delta_\tau p^n\|_{H^{1}}+\|\delta_\tau\phi^n\|_{H^{2}}\|\mu^n\|_{H^2}\right.\\
				&\left.\qquad+\|\phi^{n-1}\|_{H^{2}}\|\delta_\tau\mu^n\|_{H^2}\right), ~\text{for}~ r=1.\\
			\end{cases}
		\end{equation}
	\end{Lemma}
	From \eqref{eq_stokes_quasi_proojection_inequation0005}, we obtain
	\begin{equation}
		\label{eq_stokes_boundedness_inequalities}
		\|\mathbf{P}_h\mathbf{u}\|_{L^\infty}+\|\mathbf{P}_h\mathbf{u}\|_{W^{1,3}}\leq C\left(\|\mathbf{u}\|_{H^2}+\|p\|_{H^2}+\|\phi\|_{H^2}\|\mu\|_{H^2}\right).
	\end{equation}
	Nextly, $\forall v_h,\xi_h\in\mathring{S}_h^r,$ the discrete Laplacian operator
	\cite{2020_Chen_Hongtao_Optimal_error_estimates_for_the_scalar_auxiliary_variable_finite_element_schemes_for_gradient_flows} $\Delta_h:\mathring{S}_h^r\rightarrow \mathring{S}_h^r$ is defined by
	\begin{flalign}\label{eq_discrete_Laplacian_operator}
		\left(-\Delta_hv_h,\xi_h \right)=\left(\nabla v_h,\nabla\xi_h\right),\\
		\label{eq_discrete_Laplacian_operator11}
		\left(\nabla(-\Delta_h^{-1})v_h,\nabla\xi_h\right)=\left(v_h,\xi_h\right),
	\end{flalign}
	which is symmetric and positive on $\mathring{S}_h^r$. If $v_h$ is a constant, 
	we denote $(-\Delta_h)^{\frac{1}{2}}v_h=0$ and $(-\Delta_h)^{-\frac{1}{2}}v_h=0$.
	\begin{Lemma}[\cite{2023_CaiWentao_Optimal_L2_error_estimates_of_unconditionally_stable_finite_element_schemes_for_the_Cahn_Hilliard_Navier_Stokes_system}]
		\label{Lemma_operators_H10203}
		For $v_h\in\mathring{S}_h^r$, and the operators $\left(-\Delta_h\right)^{1/2}:\mathring{S}_h^r\rightarrow\mathring{S}_h^r$ and
		$\left(-\Delta_h\right)^{-1/2}:\mathring{S}_h^r\rightarrow\mathring{S}_h^r$, we have the estimates 
		\begin{flalign}
			\label{eq_operators_estimates_positive_one}
			&C^{-1}\|v_h\|_{H^{1}}\leq\|\left(-\Delta_h\right)^{1/2}v_h\|\leq C\|v_h\|_{H^1},\\
			\label{eq_operators_estimates_neg_one}
			&C^{-1}\|v_h\|_{H^{-1}}\leq\|\left(-\Delta_h\right)^{-1/2}v_h\|\leq C\|v_h\|_{H^{-1}},
		\end{flalign}
		and
		\begin{equation}\label{eq_operators_estimates_neg_3}
			\|v_h\|^2\leq\varepsilon\|\nabla v_h\|^2+C\varepsilon^{-1}\|\left(-\Delta_h\right)^{-1/2}v_h\|^2,
		\end{equation}
		where $\varepsilon$ is an arbitrary positive constant independent of $h$.
	\end{Lemma}
	The proof of the above three lemmas \ref{Lemma_Ritz_quasi_projection}-\ref{Lemma_operators_H10203} has been given in the appendix of 
	\cite{2023_CaiWentao_Optimal_L2_error_estimates_of_unconditionally_stable_finite_element_schemes_for_the_Cahn_Hilliard_Navier_Stokes_system}.
	For $v_h\in S_h^r$, subtracting its mean value into \eqref{eq_operators_estimates_neg_3}, we get
	\begin{equation}
		\label{eq_operators_estimates_mean_value_inequation}
		\begin{aligned}
			\|v_h\|&\leq~ \|v_h-\frac{1}{|\Omega|}(v_h,1)\|+\|\frac{1}{|\Omega|}(v_h,1)\|\\
			&\leq ~\varepsilon\|\nabla v_h\|+C_\varepsilon\|(-\Delta_h)^{-\frac{1}{2}}\tilde{v}_h\|+|(v_h,1)|,
		\end{aligned}
	\end{equation}
	where $\tilde{v}_h:=v_h-\frac{1}{|\Omega|}(v_h,1)$.
	\begin{Lemma}[\cite{2008_Brenner_Susanne_C_The_mathematical_theory_of_finite_element_methods}]
		\label{Lemma_PW_inequalities_0203}
		Let $v\in H^1(\Omega)$, $v_h\in S_h^r,~\mathring{S}_h^r,~\mathbf{X}_h^{r+1},~1\leq l\leq 6,~0\leq n\leq m\leq 1$, and $1\leq q\leq s\leq \infty$.
		The following Poincar\'{e}-Wirtinger inequality, embedding inequality, and inverse inequality hold
		\begin{flalign}
			\label{eq_PW_inequalities_one}
			&\|v\|^2\leq C\left(\|\nabla v\|^2+|v|^2\right),\\
			\label{eq_PW_inequalities_two}
			&\|v\|_{L^{l}}\leq C\|v\|_{H^1},\\
			\label{eq_PW_inequalities_three}
			&\|v_h\|_{W^{m,s}}\leq Ch^{n-m+\frac{d}{s}-\frac{d}{q}}\|v_h\|_{W^{n,q}},
		\end{flalign}
	\end{Lemma}
	We will frequently use the following discrete version of the Gr\"{o}nwall lemma.
	\begin{Lemma}[\cite{2007_HeYinnian_SunWeiwei_Stability_and_convergence_of_the_Crank_Nicolson_Adams_Bashforth_scheme_for_the_time_dependent_Navier_Stokes_equations}]
		\label{lemma_discrete_Gronwall_inequation}
		For all $0 \leq n \leq m$, let $a_n, b_n, c_n, d_n, \tau, c \geq 0$ such that
		\begin{equation}
			a_m+\tau\sum_{n=0}^{m}b_n\leq \tau\sum_{n=0}^{m}d_na_n+\tau\sum_{n=0}^{m}c_n+c,
		\end{equation}
		then
		\begin{equation}
			a_{m}+\tau\sum_{n=0}^{m} b_n \leq \exp \left(\tau\sum_{n=0}^m d_n\right)\left\{a_0+\left(b_0+c_0\right) \tau+\tau\sum_{n=1}^{m} c_n \right\} .
		\end{equation}
		
	\end{Lemma}
	Throughout the manuscript we use $C$ or $c$, with or without subscript, to denote a positive constant independent of discretization parameters that could have uncertain values in different places. 
	
	\section{The numerical scheme}\label{section_the_numerical_schemes}
	In the section, we first introduce the CHNS equations into an equivalent system with scalar auxiliary variables. Then, we construct first-order fully decoupled scheme,
	and prove that the numerical scheme satisfies unconditional energy stabilization.
	
	For $r\geq 1$, we assume that the solution to the CHNS model \eqref{eq_chns_equations} in this manuscript
	exists and satisfies the following regularities:
	\begin{equation}
		\label{eq_varibles_satisfied_regularities}
		\begin{aligned}
			&\phi\in H^2\left(0,T;L^2(\Omega)\right)\cap H^1(0,T;H^{r+1}(\Omega))\cap C\left([0,T];W^{2,4}(\Omega)\right),~\mu\in H^1\left(0,T;H^{r+1}(\Omega)\right),\\
			&\mathbf{u}\in H^2\left(0,T;\mathbf{L}^2(\Omega)\right)\cap H^1\left(0,T;\mathbf{H}^{r+1}(\Omega)\right),~p\in L^2\left(0,T;H^{r+1}(\Omega)\cap L_0^2(\Omega)\right).
		\end{aligned}
	\end{equation}
	\subsection{SAV method}
	In this subsection, we introduce an auxiliary variable $\rho$ in order to unconditionally stabilize the energy of the CHNS system which is defined by the following:
	\begin{equation}
		\rho(t)=\sqrt{E_1(\phi(t))+C}, ~\forall t\in [0,T],
	\end{equation}
 	where $E_1(\phi):=\int_{\Omega}F(\phi)~d\text{x}$ and  $C$ is a positive lower bound  of $E_1$.
	For $\epsilon \leq 1$ and $\left(x,t\right)\in \Omega\times (0,T]$, we reformulate the model \eqref{eq_chns_equations} as 
	\begin{flalign}
		\label{eq_CHNS_reformulate_phi_con}
		\frac{\partial \phi }{\partial t}+(\mathbf{u}\cdot\nabla)\phi- M\Delta \mu&=0,~ \text{in} ~\Omega\times (0,T],\\
		\label{eq_CHNS_reformulate_mu_con}
		\mu+\lambda\Delta \phi-\lambda F'(\phi)&=0,~ \text{in} ~\Omega\times (0,T],\\
		\label{eq_CHNS_reformulate_ns_con}
		\frac{\partial \mathbf{u}}{\partial t}+\frac{\rho(t)}{\sqrt{E_1(\phi)+C}}(\mathbf{u}\cdot\nabla)\mathbf{u}-\nu\Delta\mathbf{u}+\nabla p-\mu\nabla\phi&=0,~\text{in} ~\Omega\times (0,T],\\
		\label{eq_CHNS_reformulate_incompressi}
		\nabla\cdot\mathbf{u}&=0,~\text{in} ~\Omega\times (0,T],\\
		\frac{d\rho}{dt}-\frac{1}{2\rho(t)}\left(\int_{\Omega}F'(\phi)\cdot\phi_t~d\mathbf{x}
			+\frac{\rho(t)}{\lambda\sqrt{E_1(\phi)+C}}\left((\mathbf{u}\cdot\nabla)\mathbf{u},\mathbf{u}\right)\right.&\notag\\
			\left.+\frac{1}{\lambda}\left(\left((\mathbf{u}\cdot\nabla)\phi,\mu\right)-\left(\mu\nabla\phi,\mathbf{u}\right)\right)\right)&=0,~\text{in} ~\Omega\times (0,T].
	\end{flalign} 
	It is esay to see that $\rho(0)=1$ does not change the original system. 
	The additional terms $\left((\mathbf{u}\cdot\nabla)\mathbf{u},\mathbf{u}\right)$ 
	and $\left(\left((\mathbf{u}\cdot\nabla)\phi,\mu\right)-\left(\mu\nabla\phi,\mathbf{u}\right)\right)$ are so-called
	 \textit{“zero-energy-contribution”} feature, see also
	  \cite{2021_Yang_Xiaofeng_A_new_efficient_fully_decoupled_and_second_order_time_accurate_scheme_for_Cahn_Hilliard_phase_field_model_of_three_phase_incompressible_flow,
	  2023_LiYibao_Consistency_enhanced_SAV_BDF2_time_marching_method_with_relaxation_for_the_incompressible_Cahn_Hilliard_Navier_Stokes_binary_fluid_model}.
	\subsection{The first order semi-discretization scheme}\label{subsection_semi_discrete}
	For simplicity of description, we perform a time semi-discretization to prove that our numerical scheme satisfies unconditional energy stabilization.
	Let $N$ be a positive integer and $\left\{t^n=n \tau\right\}_{n=0}^N$ denote a uniform partition of the time interval $[0, T]$ with a step size $\tau=T / N$.
	With the above notations, we introduce the fully decoupled, unconditionally energy stable scheme for the CHNS system given by 
	the initial conditions $\left(\phi,\mu,\mathbf{u},p,\rho\right)^0$ and $\left(\phi,\mu,\mathbf{u},p,\rho\right)^n$, and update  $\left(\phi,\mu,\tilde{\mathbf{u}},\mathbf{u},p,\rho\right)^{n+1}$ for $n\geq 0$ from:
	
	$\mathbf{Step 1.}$ Given $\phi^n$, $\mu^n$ and $\mathbf{u}^n$, find $\left(\phi^{n+1},~\mu^{n+1}\right)$ such that
	\begin{flalign}
		\label{eq_semi_discrete_CHNS_scheme_phi}
		\frac{\phi^{n+1}-\phi^{n}}{\tau}+\mathbf{u}^{n}\cdot\nabla\phi^{n+1}-M\Delta\mu^{n+1}=0,&\\
		\label{eq_semi_discrete_CHNS_scheme_mu}
		\mu^{n+1}+\lambda\Delta\phi^{n+1}-\lambda F'(\phi^n)=0,&\\
		\label{eq_semi_discrete_CHNS_scheme_boundary_conditions}
		\partial_n\phi^{n+1}|_{\partial\Omega}=0,~	\partial_n\mu^{n+1}|_{\partial\Omega}=0.
	\end{flalign}

	$\mathbf{Step 2.}$ Given $\mathbf{u}^n$, $\rho^{n}$, $p^n$ and by using the already calculated $\phi^{n+1}$, $\mu^{n+1}$, find  $\tilde{\mathbf{u}}^{n+1}$ such that
	\begin{flalign}
		\label{eq_semi_discrete_CHNS_scheme_tilde_u}
		\frac{\tilde{\mathbf{u}}^{n+1}-\mathbf{u}^n}{\tau}+\frac{\rho^{n}}{\sqrt{E_1^{n+1}}}\mathbf{u}^n\cdot\nabla\mathbf{u}^{n}-\nu\Delta\tilde{\mathbf{u}}^{n+1}
			+\nabla p^n-\mu^{n+1}\nabla\phi^{n+1}&=0,\\
			\tilde{\mathbf{u}}^{n+1}|_{\partial\Omega}&=0,
	\end{flalign}
	where $\sqrt{E_1^{n+1}}$ replace by $\sqrt{E_1(\phi^{n+1})+C}$.

	 $\mathbf{Step 3.}$ Given $\mathbf{u}^n$, $\rho^n$ and by using the already calculated $\phi^{n+1}$, $\mu^{n+1}$, $\tilde{\mathbf{u}}^{n+1}$, 
	 find $\rho^{n+1}$ such that
	  \begin{flalign}
		 	\label{eq_semi_discrete_CHNS_scheme_tilde_u_add_term}
		 \frac{\rho^{n+1}-\rho^n}{\tau}-\frac{1}{2\rho^{n+1}}\left(\left(F'(\phi^n),\frac{\phi^{n+1}-\phi^n}{\tau}\right)
		 +\frac{\rho^{n}}{\lambda \sqrt{E_1^{n+1}}}\left(\mathbf{u}^n\cdot\nabla\mathbf{u}^{n},\tilde{\mathbf{u}}^{n+1}\right)\right.\notag&\\\left.
		 +\frac{1}{\lambda}\left(\left(\mathbf{u}^n\cdot\nabla\phi^{n+1},\mu^{n+1}\right)-\left(\mu^{n+1}\nabla\phi^{n+1},\tilde{\mathbf{u}}^{n+1}\right)\right)\right)&=0.
		 \end{flalign}
	Note that we added the terms $\left(\mathbf{u}^n\cdot\nabla\phi^{n+1},\mu^{n+1}\right)-\left(\mu^{n+1}\nabla\phi^{n+1},\tilde{\mathbf{u}}^{n+1}\right)$
	in \eqref{eq_semi_discrete_CHNS_scheme_tilde_u_add_term} which is a first-order approximation to $\left((\mathbf{u}\cdot\nabla)\phi,\mu\right)-\left(\mu\nabla\phi,\mathbf{u}\right)=0$.

	 $\mathbf{Step 4.}$  Given $p^n$ and by using the already calculated $\tilde{\mathbf{u}}^{n+1}$,
	 find $\left(p^{n+1},\mathbf{u}^{n+1}\right)$ such that
	 \begin{flalign}
	 	\label{eq_semi_discrete_CHNS_scheme_tilde_u_u1_p_p1}
	 	\frac{\mathbf{u}^{n+1}-\tilde{\mathbf{u}}^{n+1}}{\tau}+\nabla\left(p^{n+1}-p^n\right)=0,\\
	 	\label{eq_semi_discrete_CHNS_scheme_tilde_incompressible_condition}
	 	\nabla\cdot\mathbf{u}^{n+1}=0,~\mathbf{u}^{n+1}\cdot\mathbf{n}|_{\partial\Omega}=0.
	 \end{flalign}
	 \begin{Remark}
	 	In order to obtain the unconditional energy stabilization and decouple the CH and NS equations in the whole discrete system, 
	 	we construct \eqref{eq_semi_discrete_CHNS_scheme_tilde_u_add_term}, by the SAV approach and  \textit{“zero-energy-contribution”} technique, which is nonlinear.
	    Nevertheless, the complexity of nonlinearity has little negative impacts on the computational cost of our scheme.
	 \end{Remark}
 	We note that the \eqref{eq_semi_discrete_CHNS_scheme_tilde_u_add_term} is a nonlinear quadratic equation for $\rho^{n+1}$, and rewrite it in	the following form
 	\begin{equation}
 		\begin{aligned}
 			2(\rho^{n+1})^2-2\rho^n\rho^{n+1}-\left(\left(F'(\phi^n),\phi^{n+1}-\phi^n\right)
 			+\frac{\tau\rho^{n}}{\lambda \sqrt{E_1^{n+1}}}\left(\mathbf{u}^n\cdot\nabla\mathbf{u}^{n},\tilde{\mathbf{u}}^{n+1}\right)\right.\notag&\\\left.
 			+\frac{\tau}{\lambda}\left(\left(\mathbf{u}^n\cdot\nabla\phi^{n+1},\mu^{n+1}\right)-\left(\mu^{n+1}\nabla\phi^{n+1},\tilde{\mathbf{u}}^{n+1}\right)\right)\right)&=0.
 		\end{aligned}
 	\end{equation}
 The above equation can be simplified as 
 \begin{equation}
 	ax^2+bx+c=0,
 \end{equation}
 where the coefficients are
 \begin{equation}
 	\begin{aligned}
 		a=&2,~b=-2\rho^n,\\
 		c=&-\left(\left(F'(\phi^n),\phi^{n+1}-\phi^n\right)
 		+\frac{\tau\rho^{n}}{\lambda \sqrt{E_1^{n+1}}}\left(\mathbf{u}^n\cdot\nabla\mathbf{u}^{n},\tilde{\mathbf{u}}^{n+1}\right)\right.\\
 		&\left.
 		+\frac{\tau}{\lambda}\left(\left(\mathbf{u}^n\cdot\nabla\phi^{n+1},\mu^{n+1}\right)-\left(\mu^{n+1}\nabla\phi^{n+1},\tilde{\mathbf{u}}^{n+1}\right)\right)\right).
 	\end{aligned}
 \end{equation}
 According to $-2ab\geq -\left(a^2+b^2\right)$, we can obtain $b^2-4ac\geq 0$, then the nonlinear quadratic equation \eqref{eq_semi_discrete_CHNS_scheme_tilde_u_add_term} has real number solutions.
 \begin{Remark}
 	We observe that the solution of the quadratic SAV approach \eqref{eq_semi_discrete_CHNS_scheme_tilde_u_add_term} may have different two real roots. Therefore, we need to give how to choose a more desirable root. A similar approach has been given in \cite{2020_lixiaoli_New_SAV_MAC_NS,2019_LinLianlei_Numerical_approximation_of_incompressible_Navier_Stokes_equations_based_on_an_auxiliary_energy_variable} and we omit the detailed process here.
 	Since the exact solution of $\frac{\rho^{n+1}}{\sqrt{E_1^{n+1}}}$ is 1, we choose the root $\rho^{n+1}$ such that
 	$\frac{\rho^{n+1}}{\sqrt{E_1^{n+1}}}$ is closer to 1.
 \end{Remark}
 	\subsection{Unconditional energy stability}
 	For the above scheme, we can establish the unconditional energy stabilization as follows:
 	\begin{Theorem}
 		\label{theorem_unconditionally_energy_stable}
 		The decoupled scheme \eqref{eq_semi_discrete_CHNS_scheme_phi}-\eqref{eq_semi_discrete_CHNS_scheme_tilde_incompressible_condition} is uniquely solvable,
 		unconditionally energy stable, and satisfies discrete energy law as follows:
 		\begin{equation}
 			\label{eq_semi_discrete_scheme_energy_law}
 			\begin{aligned}
 			\tilde{E}^{n+1}-\tilde{E}^{n}\leq -2M\tau\|\nabla\mu^{n+1}\|^2-2\nu\tau\|\nabla\tilde{\mathbf{u}}^{n+1}\|^2,
 			\end{aligned}
 		\end{equation}
 	where $\tilde{E}^{n+1}$ is defined by
 	\begin{equation}
 		\tilde{E}^{n+1}=\lambda\|\nabla\phi^{n+1}\|^2+\|\mathbf{u}^{n+1}\|^2+\tau^2\|\nabla p^{n+1}\|^2+2\lambda|\rho^{n+1}|^2.
 	\end{equation}
 	\end{Theorem}
 	\begin{proof}
 		Taking the inner product of the equation \eqref{eq_semi_discrete_CHNS_scheme_phi} with $2\tau\mu^{n+1}$, we can obtain
 		\begin{equation}
 			\label{eq_semi_discrete_CHNS_scheme_phi_inner_2tau_mu}
 			2\left(\phi^{n+1}-\phi^n,\mu^{n+1}\right)+2\tau\left(\mathbf{u}^n\cdot\nabla\phi^{n+1},\mu^{n+1}\right)+2M\tau\|\nabla\mu^{n+1}\|^2=0.
 		\end{equation}
 		Taking the inner product of the equation \eqref{eq_semi_discrete_CHNS_scheme_mu} with $-2\left(\phi^{n+1}-\phi^n\right)$, we can obtain
 		\begin{equation}
 			\label{eq_inner_product_eq_semi_discrete_CHNS_scheme_mu}
 			\begin{aligned}
 				-2\left(\mu^{n+1},\phi^{n+1}-\phi^n\right)+\lambda\left(\|\nabla\phi^{n+1}\|^2-\|\nabla\phi^{n}\|^2+\|\nabla\phi^{n+1}-\nabla\phi^n\|^2\right)\\
 				+\lambda\left(F'(\phi^n),\frac{\phi^{n+1}-\phi^n}{\tau}\right)&=0.
 			\end{aligned}
 		\end{equation}
 	Taking the inner product of \eqref{eq_semi_discrete_CHNS_scheme_tilde_u} with $2\tau\tilde{\mathbf{u}}^{n+1}$, we 
 	derive that
 	\begin{equation}
 		\label{eq_semi_discrete_CHNS_scheme_tilde_u_inner_product_2delta_t}
 		\begin{aligned}
 			\|\tilde{\mathbf{u}}^{n+1}\|^2-\|\mathbf{u}^n\|^2+\|\tilde{\mathbf{u}}^{n+1}-\mathbf{u}^n\|^2+2\nu\tau
 			\|\nabla\tilde{\mathbf{u}}^{n+1}\|^2+\frac{2\tau\rho^n}{\sqrt{E_1^{n+1}}}\left(\mathbf{u}^n\cdot\nabla\mathbf{u}^n,\tilde{\mathbf{u}}^{n+1}\right)&\\
 			+2\tau\left(\nabla p^n,\tilde{\mathbf{u}}^{n+1}\right)-2\tau\left(\mu^{n+1}\nabla\phi^{n+1},\tilde{\mathbf{u}}^{n+1}\right)&=0.
 		\end{aligned}
 	\end{equation}
 	Next, we rewrite \eqref{eq_semi_discrete_CHNS_scheme_tilde_u_u1_p_p1} as
 	\begin{equation}
 		\label{eq_semi_discrete_CHNS_scheme_tilde_u_u1_p_p1_rewrite}
 		\mathbf{u}^{n+1}+\tau\nabla p^{n+1}=\tilde{\mathbf{u}}^{n+1}+\tau\nabla p^n.
 	\end{equation}
 	Then we derive from \eqref{eq_semi_discrete_CHNS_scheme_tilde_u_u1_p_p1_rewrite} that
 	\begin{equation}
 		\label{eq_semi_discrete_CHNS_scheme_tilde_u_u1_p_p1_rewrite_d}
 		\|\mathbf{u}^{n+1}\|^2-\|\tilde{\mathbf{u}}^{n+1}\|^2-2\tau\left(\nabla p^n,\tilde{\mathbf{u}}^{n+1}\right)+\tau^2\left(\|\nabla p^{n+1}\|^2-\|\nabla p^n\|^2\right)=0.
 	\end{equation}
 	Multiplying \eqref{eq_semi_discrete_CHNS_scheme_tilde_u_add_term} by $4\tau\lambda\rho^{n+1}$, we have
 	\begin{equation}
 		\label{eq_r_mut_rn1}
 		\begin{aligned}
 			2&\lambda\left(|\rho^{n+1}|^2-|\rho^n|^2+|\rho^{n+1}-\rho^n|^2\right)
 			=\lambda\left(F'(\phi^n),\frac{\phi^{n+1}-\phi^n}{\tau}\right)\\
 				&+\frac{2\tau\rho^n}{\sqrt{E_1^{n+1}}}\left(\mathbf{u}^n\cdot\nabla\mathbf{u}^n,\tilde{\mathbf{u}}^{n+1}\right)+2\tau\left(\mathbf{u}^n\cdot\nabla\phi^{n+1},\mu^{n+1}\right)-2\tau\left(\mu^{n+1}\nabla\phi^{n+1},\tilde{\mathbf{u}}^{n+1}\right).
 		\end{aligned}
 	\end{equation}
 	Adding the \eqref{eq_semi_discrete_CHNS_scheme_tilde_u_u1_p_p1_rewrite_d}, \eqref{eq_r_mut_rn1} and \eqref{eq_semi_discrete_CHNS_scheme_tilde_u_inner_product_2delta_t}, we can obtain
 	\begin{equation}
 		\label{eq_combing_two_equations}
 		\begin{aligned}	
 			\|\mathbf{u}^{n+1}\|^2&-\|\mathbf{u}^n\|^2+\|\tilde{\mathbf{u}}^{n+1}-\mathbf{u}^n\|^2+2\nu\tau
 			\|\nabla\tilde{\mathbf{u}}^{n+1}\|^2+\tau^2\left(\|\nabla p^{n+1}\|^2-\|\nabla p^n\|^2\right)\\
 			&+2\lambda\left(|\rho^{n+1}|^2-|\rho^n|^2+|\rho^{n+1}-\rho^n|^2\right)=\lambda\left(F'(\phi^n),\frac{\phi^{n+1}-\phi^n}{\tau}\right)
 			+
 			2\tau\left(\mathbf{u}^n\cdot\nabla\phi^{n+1},\mu^{n+1}\right).
 		\end{aligned}
 	\end{equation}
 	Combining the \eqref{eq_semi_discrete_CHNS_scheme_phi_inner_2tau_mu},  \eqref{eq_inner_product_eq_semi_discrete_CHNS_scheme_mu} with \eqref{eq_combing_two_equations}, we get
 	\begin{equation}
 		\label{eq_unconditional_energy_energy}
 		\begin{aligned}
 			\lambda&\left(\|\nabla\phi^{n+1}\|^2-\|\nabla\phi^{n}\|^2+\|\nabla\phi^{n+1}-\nabla\phi^n\|^2\right)\\
 			&+\|\mathbf{u}^{n+1}\|^2-\|\mathbf{u}^n\|^2+\|\tilde{\mathbf{u}}^{n+1}-\mathbf{u}^n\|^2+\tau^2\left(\|\nabla p^{n+1}\|^2-\|\nabla p^n\|^2\right)\\
 			&+2\lambda\left(|\rho^{n+1}|^2-|\rho^n|^2+|\rho^{n+1}-\rho^n|^2\right)
 			=-2M\tau\|\nabla\mu^{n+1}\|^2-2\nu\tau
 			\|\nabla\tilde{\mathbf{u}}^{n+1}\|^2.
 		\end{aligned}
 	\end{equation}
  Thus, we can obtain the desired result \eqref{eq_semi_discrete_scheme_energy_law}. 
 	\end{proof}
 	\begin{Remark}\label{remark_existence_solution}
 		The existence of a solution for the scheme \eqref{eq_semi_discrete_CHNS_scheme_phi}-\eqref{eq_semi_discrete_CHNS_scheme_tilde_incompressible_condition} 
 		can be proved using the standard argument of the Leray-Schauder fixed point theorem (see \cite{2015_ShenJie_Decoupled_energy_stable_schemes_for_phase_field_models_of_two_phase_incompressible_flows} for details of similar problems).
 	\end{Remark}
 	
 	\subsection{Fully discrete scheme}
 	In the subsection, we construct a decoupled, first-order and fully discrete scheme. For any test functions $\left(w,\varphi,\mathbf{v},q\right)\in\left(H^1(\Omega),H^1(\Omega),\mathbf{H}_0^1(\Omega),L_0^2(\Omega)\right)$,
 	the weak solutions $\left(\phi,\mu,\mathbf{u},p,\rho\right)$ of \eqref{eq_chns_equations} satisfies the following variational 
 	forms:
 	\begin{flalign}
 		\label{eq_variational_forms_phi}
 		\left(\partial_t\phi,w\right)+M\left(\nabla\mu,\nabla w\right)+b\left(\phi,\mathbf{u},w\right)&=0,\\
 		\label{eq_variational_forms_mu}
 		\left(\mu,\varphi\right)-\lambda\left(\nabla\phi,\nabla\varphi\right)-\lambda\left(F'(\phi),\varphi\right)&=0,\\
 		\label{eq_variational_forms_ns}
 		\left(\partial_t\mathbf{u},\mathbf{v}\right)+\nu\left(\nabla\mathbf{u},\nabla\mathbf{v}\right)
 			+\frac{\rho(t)}{\sqrt{E_1(\phi)+C}}\boldsymbol{B}\left(\mathbf{u},\mathbf{u},\mathbf{v}\right)+\left(\nabla p,\mathbf{v}\right)-b\left(\phi,\mathbf{v},\mu\right)&=0,\\
 		\label{eq_variational_forms_incompressible_condition}
 		\left(\nabla\cdot\mathbf{u},q\right)&=0,\\
 		\label{eq_variational_forms_additional_term}
 		\frac{d\rho}{dt}-\frac{1}{2\rho(t)}\left(\int_{\Omega}F'(\phi)\cdot\phi_t~d\mathbf{x}
 		+\frac{\rho(t)}{\lambda\sqrt{E_1(\phi)+C}}\boldsymbol{B}\left(\mathbf{u},\mathbf{u},\mathbf{u}\right)
 		+\frac{1}{\lambda}\left(\left(\mathbf{u}\cdot\nabla\phi,\mu\right)-\left(\mu\cdot\nabla\phi,\mathbf{u}\right)\right)\right)&=0.
 	\end{flalign}
 	It is well known that both the Taylor-Hood elements satisfy the discrete inf-sup condition \cite{1986_Girault_Vivette_Finite_element_methods_for_Navier_Stokes_equations}, i.e.,
 	\begin{equation}
 		\left\|q_h\right\|_{L^2} \leq C \sup _{\mathbf{0} \neq \mathbf{v}_h \in \mathbf{X}_h^{r+1} }
 		\frac{\left(q_h, \nabla \cdot \mathbf{v}_h\right)}{\left\|\nabla\mathbf{v}_h\right\|}, \quad \forall q_h \in \mathring{S}_h^r,
 	\end{equation}
 	for some constant $C>0$. 
 	For the simplicity of notations, we denote
 	$
 		\mathbf{\mathcal{X}}_h^r= S_h^r \times S_h^r \times \mathbf{X}_h^{r+1}  \times \mathring{S}_h^r\times \mathbb{R},
 	$
 	$\boldsymbol{v}^{n+1}=\boldsymbol{v}\left(x, t^{n+1}\right)$, 
 	and
 	\begin{equation}\label{eq_d_t}
 		\delta_\tau\boldsymbol{v}^{n+1}=\frac{\boldsymbol{v}^{n+1}-\boldsymbol{v}^n}{\tau}.
 	\end{equation}
	 \begin{Remark}
	 	In an abuse of notation, we use  $\boldsymbol{v}^{n+1}$  hereafter to denote the value of the exact solution $\boldsymbol{v}$ at $t^{n+1}$. 
	 	To simplify the notation, we replace $E_{1,h}^{n+1}$ with
	 	$E_1(\phi_h^{n+1})+C$, and $E_{1}^{n+1}$ with
	 	$E_1(\phi(t^{n+1}))+C$.
	 \end{Remark}
 	With the above definitions, we construct the following fully discrete decoupled first order scheme:
 	
	 $\mathbf{Step 1.}$ Given $\phi_h^n$, $\mu_h^n$ and $\mathbf{u}_h^n$,
	 find $\left(\phi_h^{n+1},\mu_h^{n+1}\right)\in \left(S_h^r\times S_h^r\right)$ such that
	 \begin{flalign}
	 	\label{eq_fully_discrete_CHNS_scheme_phi}
	 	\left(\delta_\tau\phi_h^{n+1},w_h\right)+b\left(\phi_h^{n+1},\mathbf{u}_h^{n},w_h\right)+M\left(\nabla\mu_h^{n+1},\nabla w_h\right)=0,&\\
	 	\label{eq_fully_discrete_CHNS_scheme_mu}
	 	\left(\mu_h^{n+1},\varphi_h\right)-\lambda\left(\nabla\phi_h^{n+1},\nabla\varphi_h\right)-\lambda\left(F'(\phi_h^n),\varphi_h\right)=0,&\\
	 	\label{eq_fully_discrete_CHNS_scheme_boundary_conditions}
	 	\partial_n\phi_h^{n+1}|_{\partial\Omega}=0,~	\partial_n\mu_h^{n+1}|_{\partial\Omega}=0.
	 \end{flalign}
	 
	 $\mathbf{Step 2.}$ Given $\mathbf{u}_h^n$, $\rho_h^n$, $p_h^n$ and by using the already calculated $\phi_h^{n+1}$, 
	 $\mu_h^{n+1}$, find $\tilde{\mathbf{u}}_h^{n+1}\in \mathbf{X}_h^{r+1}$ such that
	 \begin{flalign}
	 	\label{eq_fully_discrete_CHNS_scheme_tilde_u}
	 	\left(\frac{\tilde{\mathbf{u}}_h^{n+1}-\mathbf{u}_h^n}{ \tau},\mathbf{v}_h\right)
	 		+\frac{\rho_h^{n}}{ \sqrt{E_{1,h}^{n+1}}}\boldsymbol{B}\left(\mathbf{u}_h^n,\mathbf{u}_h^n,\mathbf{v}_h\right)
	 		+\nu\left(\nabla\tilde{\mathbf{u}}_h^{n+1},\nabla\mathbf{v}_h\right)&\notag\\
	 	+\left(\nabla p_h^n,\mathbf{v}_h\right)-b\left(\phi_h^{n+1},\mu_h^{n+1},\mathbf{v}_h\right)&=0,\\
	 	\tilde{\mathbf{u}}_h^{n+1}|_{\partial\Omega}&=0.
	 \end{flalign}
 	\par
  $\mathbf{Step 3.}$ Given $\mathbf{u}_h^n$, $\rho_h^n$ and by using the already calculated $\phi_h^{n+1}$, $\mu_h^{n+1}$, $\tilde{\mathbf{u}}_h^{n+1}$, 
  find $\rho_h^{n+1}$ such that
  \begin{flalign}
  	\label{eq_fully_discrete_CHNS_scheme_tilde_u_add_term}
  	\frac{\rho_h^{n+1}-\rho_h^n}{\tau}-\frac{1}{2\rho_h^{n+1}}\left(\left(F'(\phi_h^n),\frac{\phi_h^{n+1}-\phi_h^n}{\tau}\right)
  	+\frac{\rho_h^{n}}{\lambda \sqrt{E_{1,h}^{n+1}}}\boldsymbol{B}\left(\mathbf{u}_h^n,\mathbf{u}_h^{n},\tilde{\mathbf{u}}_h^{n+1}\right)\right.\notag&\\\left.
  	+\frac{1}{\lambda}\left(\left(\mathbf{u}_h^n\cdot\nabla\phi_h^{n+1},\mu_h^{n+1}\right)-\left(\mu_h^{n+1}\nabla\phi_h^{n+1},\tilde{\mathbf{u}}_h^{n+1}\right)\right)\right)&=0.
  \end{flalign}

 	 $\mathbf{Step 4.}$
 	Find $\left(p_h^{n+1},\mathbf{u}_h^{n+1}\right)\in \left(\mathring{S}_h^r,\mathbf{X}_h^{r+1}\right)$ such that
 	\begin{flalign}
 		\label{eq_fully_discrete_CHNS_scheme_tilde_u_u1_p_p1}
 		\frac{\mathbf{u}_h^{n+1}-\tilde{\mathbf{u}}_h^{n+1}}{\tau}+\nabla\left(p_h^{n+1}-p_h^n\right)=0,\\
 		\label{eq_fully_discrete_CHNS_scheme_tilde_incompressible_condition}
 		\left(\nabla\cdot\mathbf{u}_h^{n+1},q_h\right)=0,~\mathbf{u}_h^{n+1}\cdot\mathbf{n}|_{\partial\Omega}=0,
 	\end{flalign}
 	hold for all $\left(w_h,\varphi_h,\mathbf{v}_h,q_h\right)\in\mathbf{\mathcal{X}}_h^r$, and $n=0,1,2,\cdots,N-1$, where $\phi_h^0=R_h\phi^0,~\mathbf{u}_h^0=\mathbf{I}_h\mathbf{u}^0.$
 	\begin{Remark}
 		The fully discrete scheme \eqref{eq_fully_discrete_CHNS_scheme_phi}-\eqref{eq_fully_discrete_CHNS_scheme_tilde_incompressible_condition}
 		is unconditionally energy stable and unconditionally uniquely solvable 
 		(see Theorem \ref{theorem_unconditionally_energy_stable} and Remark \ref{remark_existence_solution}).
 	\end{Remark} 
 	To simplify the following theoretical analysis, we set the parameters 
 	$M=\lambda=\nu=1$. Then we can get the following lemmas which is boundedness of numerical solutions of scheme \eqref{eq_fully_discrete_CHNS_scheme_phi}-\eqref{eq_fully_discrete_CHNS_scheme_tilde_incompressible_condition}.
 	\begin{Lemma}
 		\label{Lemma_0301}
 		Let $\left(\phi_h^{n+1},\mu_h^{n+1},\tilde{\mathbf{u}}_h^{n+1},\mathbf{u}_h^{n+1},p_h^{n+1},\rho_h^{n+1}\right)\in \mathbf{\mathcal{X}}_h^r $ be the unique solution of  \eqref{eq_fully_discrete_CHNS_scheme_phi}-\eqref{eq_fully_discrete_CHNS_scheme_tilde_incompressible_condition}. Suppose 
 		that $E(\phi_h^0,\mathbf{u}_h^0)<C$, 
 		 for any $h,\tau>0$, the following estimates hold that
 		\begin{flalign}
 			\label{eq_boundedness_estimates_phi1}
 			\max_{0\leq n\leq m}\left(\|\mathbf{u}_h^{n+1}\|^2+\|\nabla\phi_h^{n+1}\|^2
 				+\tau^2\|\nabla p_h^{n+1}\|^2+2|\rho_h^{n+1}|^2\right)&\leq C,\\
 				\label{eq_boundedness_estimates_phi2}
 			\tau\sum_{n=0}^{m}\left(\|\nabla\mu_h^{n+1}\|^2+\|\nabla\tilde{\mathbf{u}}_h^{n+1}\|^2
 			\right)&\leq C,\\
			\sum_{n=0}^{N-1}\left(\|\nabla\phi_h^{n+1}-\nabla\phi_h^n\|^2
			+\|\tilde{\mathbf{u}}_h^{n+1}-\mathbf{u}_h^n\|^2+2|\rho_h^{n+1}-\rho_h^n|^2\right)&\leq C,\\
 			\|\phi_h^{m+1}\|_1^2+\tau\sum_{n=0}^{m}\|\mu_h^{n+1}\|^2&\leq C,
 		\end{flalign}
 	for some $C>0$ that is positive constant independent of $\tau$, $h$ and $T$.
 	\end{Lemma}
 	\begin{proof}
 		Summing from $n=0$ to $m$ in \eqref{eq_unconditional_energy_energy}, we can obtain
 		\eqref{eq_boundedness_estimates_phi1} and \eqref{eq_boundedness_estimates_phi2}.
 		Taking inner product with $w_h=2\tau\phi_h^{n+1}$ in \eqref{eq_fully_discrete_CHNS_scheme_phi}, we can obtain
 		\begin{equation}
 			\label{eq_taking_inner_product_2tau_phi}
 			\|\phi_h^{n+1}\|^2-\|\phi_h^n\|^2+\|\phi_h^{n+1}-\phi_h^n\|^2=-2\tau\left(\mathbf{u}_h^{n}\cdot\nabla\phi_h^{n+1},\phi_h^{n+1}\right)
 				-2\tau\left(\nabla\mu_h^{n+1},\nabla\phi_h^{n+1}\right).
 		\end{equation}
 		Then using the Cauchy-Schwarz inequality and Young inequality, the right-hand side of \eqref{eq_taking_inner_product_2tau_phi} 
 		can be estimated by
 		\begin{flalign}
 			\label{eq_taking_inner_product_2tau_phi_right_hand_side01}
 			\big|2\tau\left(\mathbf{u}_h^{n}\cdot\nabla\phi_h^{n+1},\phi_h^{n+1}\right)\big|
 					\leq&~2\tau\|\mathbf{u}_h^{n}\|_{L^4}\|\phi_h^{n+1}\|_{L^4}\|\nabla\phi_h^{n+1}\|\notag\\
 					\leq&~\tau\|\nabla\phi_h^{n+1}\|^2
 					+C\tau\|\nabla\mathbf{u}_h^n\|^2\left(\|\phi_h^{n+1}\|^2+\|\nabla\phi_h^{n+1}\|^2\right),\\
 			\label{eq_taking_inner_product_2tau_phi_right_hand_side02}
 			|2\tau\left(\nabla\mu_h^{n+1},\nabla\phi_h^{n+1}\right)|\leq &~
 				 \tau\left(\|\nabla\mu_h^{n+1}\|^2+\|\nabla\phi_h^{n+1}\|^2\right).
 		\end{flalign}
 		Combining the above inequalities \eqref{eq_taking_inner_product_2tau_phi_right_hand_side01} and \eqref{eq_taking_inner_product_2tau_phi_right_hand_side02} with \eqref{eq_taking_inner_product_2tau_phi},
 		we have
 		\begin{equation}
 			\label{eq_combining_the_above_inequalities}
 			\begin{aligned}
 				\|\phi_h^{n+1}\|^2-\|\phi_h^n\|^2+\|\phi_h^{n+1}-\phi_h^n\|^2\leq &~\tau\|\nabla\phi_h^{n+1}\|^2
 					+C\tau\|\nabla\mathbf{u}_h^n\|^2\left(\|\phi_h^{n+1}\|^2+\|\nabla\phi_h^{n+1}\|^2\right)\\
 					&+\tau\left(\|\nabla\mu_h^{n+1}\|^2+\|\nabla\phi_h^{n+1}\|^2\right).
 			\end{aligned}
 		\end{equation}
 	Summing from $n=0$ to $m$ in \eqref{eq_combining_the_above_inequalities}, we can obtain
 	\begin{equation}
 		\label{eq_summing_from_comb_above_ineq}
 		\|\phi_h^{m+1}\|^2\leq\|\phi_h^0\|^2+\tau\sum_{n=0}^{m}\left(\|\nabla\mu_h^{n+1}\|^2+\|\nabla\phi_h^{n+1}\|^2+C\|\nabla\mathbf{u}_h^n\|^2\|\phi_h^{n+1}\|^2+C\|\nabla\mathbf{u}_h^{n+1}\|^2\|\nabla\phi_h^{n+1}\|^2\right).
 	\end{equation}
 	Nextly, we observe from \eqref{eq_fully_discrete_CHNS_scheme_tilde_u_u1_p_p1} that
 	 $\mathbf{u}_h^{n+1}=\mathbf{I}_h\tilde{\mathbf{u}}_h^{n+1}$. 
 	 Hence,  \eqref{eq_L2_estimate01} implies that $\|\mathbf{u}_h^{n+1}\|_1\leq C(\Omega)\|\tilde{\mathbf{u}}_h^{n+1}\|_1$, using the Gr\"{o}nwall's Lemma \ref{lemma_discrete_Gronwall_inequation}, we can get
 	 \begin{equation}
 	 	\label{eq_phi_norm1to_norm2}
 	 	\|\phi_h^{m+1}\|_1^2\leq\|\phi_h^{m+1}\|^2+\|\nabla\phi_h^{m+1}\|^2\leq C_1.
 	 \end{equation}
  	Taking inner product with $2\tau\mu_h^{n+1}$ in \eqref{eq_fully_discrete_CHNS_scheme_mu}, using the \eqref{eq_boundedness_estimates_phi1} and \eqref{eq_boundedness_estimates_phi2}, we can obtain
  	\begin{equation}
  		\label{eq_taking_inner_product_2tau_mu}
  		\begin{aligned}
  			2\tau\|\mu_h^{n+1}\|^2=&~2\tau\left(\nabla\phi_h^{n+1},\nabla\mu_h^{n+1}\right)+\lambda\left(F'(\phi_h^n),\mu_h^{n+1}\right)\\
  				\leq&~2\tau\left(\|\nabla\phi_h^{n+1}\|^2+\|\nabla\mu_h^{n+1}\|^2+\|\mu_h^{n+1}\|^2+\|F'(\phi_h^n)\|^2\right).
  		\end{aligned}
  	\end{equation}
  	From (2.37) in \cite{2024_YiNianyu_Convergence_analysis_of_a_decoupled_pressure_correction_SAV_FEM_for_the_Cahn_Hilliard_Navier_Stokes_model} and $\|\phi\|_{L^2}\leq C\|\phi\|_1$, we can have
  	\begin{equation}
  		\label{eq_nonlinear_item_inequalities}
  		\begin{aligned}
  			\|F'(\phi_h^n)\|^2=&~\frac{1}{\epsilon^4}\|(\phi_h^n)^3-\phi_h^n\|^2\leq C\left(\|\phi_h^n\|_{L^6}^6+\|\phi_h^n\|^2\right)\\
  			\leq &~C\left(\|\phi_h^n\|_1^6+\phi_h^n\|^2\right)\leq C_2.
  		\end{aligned}
  	\end{equation} 
  	According to \eqref{eq_taking_inner_product_2tau_mu} and \eqref{eq_nonlinear_item_inequalities}, it follows that 
  	\begin{equation}
  		\label{eq_muhn_inequalities}
  		\tau\sum_{n=0}^{m}\|\mu_h^{n+1}\|^2\leq C.
  	\end{equation} 
  	Combining the \eqref{eq_summing_from_comb_above_ineq} and \eqref{eq_muhn_inequalities}, 
  	we have thus proved the Lemma \autoref{Lemma_0301}.
 	\end{proof}
	\section{Error analysis}\label{section_error_analysis}
	In this section, we focus on the error estimates for the scheme \eqref{eq_fully_discrete_CHNS_scheme_phi}-\eqref{eq_fully_discrete_CHNS_scheme_tilde_incompressible_condition}. We denote
	\begin{equation}
		e_\phi^{n+1}=R_h\phi^{n+1}-\phi_h^{n+1}, e_\mu^{n+1}=\Pi_h\mu^{n+1}-\mu_h^{n+1}, e_\rho^{n+1}=\rho^{n+1}-\rho_h^{n+1}
	\end{equation}
	and
	\begin{equation}
		e_{\tilde{\mathbf{u}}}^{n+1}=\mathbf{P}_h\mathbf{u}^{n+1}-\tilde{\mathbf{u}}_h^{n+1},
		e_\mathbf{u}^{n+1}=\mathbf{P}_h\mathbf{u}^{n+1}-\mathbf{u}_h^{n+1}, e_p^{n+1}=P_hp^{n+1}-p_h^{n+1}.
	\end{equation}
	For $\left(e_\phi^{n+1},e_\mu^{n+1},e_{\tilde{\mathbf{u}}}^{n+1},e_\mathbf{u}^{n+1},e_p^{n+1},e_\rho^{n+1}\right)$ 
	and the Ritz quasi-projection \eqref{eq_Ritz_quasi_projections_equation} and 
	Stokes quasi-projection \eqref{eq_stokes_quasi_proojection_equation0001}-\eqref{eq_stokes_quasi_proojection_equation0002}, 
	we subtract \eqref{eq_variational_forms_phi}-\eqref{eq_variational_forms_incompressible_condition} 
	from \eqref{eq_fully_discrete_CHNS_scheme_phi}-\eqref{eq_fully_discrete_CHNS_scheme_tilde_incompressible_condition} to
	obtain the following error equations,
	\begin{flalign}
		\label{eq_error_equations_phi}
		\left(\delta_\tau e_\phi^{n+1},w_h\right)+M\left(\nabla e_\mu^{n+1},\nabla w_h\right)+b\left(R_h\phi^{n+1},\mathbf{u}^n,w_h\right)-b\left(\phi_h^{n+1},\mathbf{u}_h^n,w_h\right)
			\notag\\
			+\left(\nabla\left(\phi^{n+1}-R_h\phi^{n+1}\right)\cdot\mathbf{u}^n,w_h\right)-\left(\delta_\tau\left(R_h\phi^{n+1}-\phi^{n+1}\right),w_h\right)
			-\left(R_1^{n+1},w_h\right)&=0,\\
		\label{eq_error_equations_mu}
		\left(e_\mu^{n+1},\varphi_h\right)-\lambda\left(\nabla e_\phi^{n+1},\nabla\varphi_h\right)
			-\lambda\left(\nabla\left(\phi^{n+1}-R_h\phi^{n+1}\right),\nabla\varphi_h\right)-\left(\mu^{n+1}-\Pi_h\mu^{n+1},\varphi_h\right)\notag\\
			+\frac{\lambda}{\epsilon^2}\left((\phi_h^{n})^3-(\phi^{n})^3+e_\phi^{n}+\phi^n-R_h\phi^n,\varphi_h\right)&=0,\\
		\label{eq_error_equations_ns}
		\left( \frac{e_{\tilde{\mathbf{u}}}^{n+1}-e_\mathbf{u}^n}{\tau},\mathbf{v}_h\right)+\nu\left(\nabla e_{\tilde{\mathbf{u}}}^{n+1},\nabla\mathbf{v}_h\right)
		+\left(\nabla e_p^{n},\mathbf{v}_h\right)&\notag\\
		+\left(\frac{\rho^{n+1}}{\sqrt{E_1^{n+1}}}\left(\mathbf{u}^{n+1}\cdot\nabla\mathbf{u}^{n+1},\mathbf{v}_h\right)
		-\frac{\rho_h^n}{\sqrt{E_{1,h}^{n+1}}}\left(\mathbf{u}_h^{n}\cdot\nabla{\mathbf{u}}_h^{n},\mathbf{v}_h\right)\right)\notag\\
		+\left(\delta_\tau\left(\mathbf{u}^{n+1}-\mathbf{P}_h\mathbf{u}^{n+1}\right),\mathbf{v}_h\right)+\left(b\left(\phi_h^{n+1},\mu_h^{n+1},\mathbf{v}_h\right)-b\left(R_h\phi^{n+1},\mu^{n+1},\mathbf{v}_h\right)\right)-
		\left(R_2^{n+1},\mathbf{v}_h\right)&=0,\\
		\label{eq_error_equations_u_p_cor}
		\frac{e_\mathbf{u}^{n+1}-e_{\tilde{\mathbf{u}}}^{n+1}}{\tau}+\nabla\left(e_p^{n+1}-e_p^{n}\right)-R_3^{n+1}&=0,\\
		\label{eq_error_equations_incompressible}
		\left(\nabla\cdot e_\mathbf{u}^{n+1},q_h\right)&=0,\\
		\label{eq_error_equations_additional_terms}
		\delta_\tau e_\rho^{n+1}-\left(\frac{1}{2\rho^{n+1}}
		\left(F'(\phi^n),\delta_\tau\phi^{n+1}\right)-\frac{1}{2\rho_h^{n+1}}\left(F'(\phi_h^n),\delta_\tau\phi_h^{n+1}\right)\right)&\notag\\
		-\left(\frac{1}{2\lambda \sqrt{E_{1}^{n+1}}}\left(\mathbf{u}^n\cdot\nabla\mathbf{u}^{n},\mathbf{u}^{n}\right)-\frac{\rho_h^{n}}{2\lambda\rho_h^{n+1} \sqrt{E_{1,h}^{n+1}}}\left(\mathbf{u}_h^n\cdot\nabla\mathbf{u}_h^{n},\tilde{\mathbf{u}}_h^{n+1}\right)\right)\notag&\\
		-\left(\frac{1}{\lambda}\left(\left(\mathbf{u}^{n+1}\cdot\nabla\phi^{n+1},\mu^{n+1}\right)\right)-\left(\mu^{n+1}\nabla\phi^{n+1},\mathbf{u}^{n+1}\right)\right)&\notag\\
		+\left(\frac{1}{\lambda}\left(\left(\mathbf{u}_h^n\cdot\nabla\phi_h^{n+1},\mu_h^{n+1}\right)\right)-\left(\mu_h^{n+1}\nabla\phi_h^{n+1},\tilde{\mathbf{u}}_h^{n+1}\right)\right)
		+\frac{1}{2\rho^{n+1}}\left(F'(\phi^n),R_1^{n+1}\right)-R_4^{n+1}&=0,
	\end{flalign}
	where the truncation errors $R_1^{n+1}$, $R_2^{n+1}$, $R_3^{n+1}$ and $R_4^{n+1}$ are denoted by
	\begin{equation}
		\begin{aligned}
			&R_1^{n+1}=\delta_\tau R_h\phi^{n+1}-\partial_t\phi^{n+1},
			&R_2^{n+1}=\delta_\tau\mathbf{P}_h\mathbf{u}^{n+1}-\partial_t\mathbf{u}^{n+1},\\
			&R_3^{n+1}=\nabla\left(P_h p^{n+1}-P_h p^n\right),
			&R_4^{n+1}=\delta_\tau \rho^{n+1}-\partial_t \rho^{n+1}.
		\end{aligned}
	\end{equation}
	With the above error equations, we can obtain the following optimal $L^2$ error estimate which  is the main result of this draft:
	\begin{Theorem}\label{theorem_error_estimates_u_phi_mu_r}
		Suppose that the model \eqref{eq_chns_equations}-\eqref{eq_boundary_initial_conditions_equations} has a
		  unique solution $\left(\phi,\mu,\mathbf{u},p\right)$ 
		satisfying the regularities \eqref{eq_varibles_satisfied_regularities}. Then $\left(\phi_h^{n+1},\mu_h^{n+1},\mathbf{u}_h^{n+1},p_h^{n+1},\rho_h^{n+1}\right)\in \mathbf{\mathcal{X}}_h^r $ is 
		unique solution in the fully discrete scheme 
		\eqref{eq_fully_discrete_CHNS_scheme_phi}-\eqref{eq_fully_discrete_CHNS_scheme_tilde_u_add_term}.
		for a positive constant $\tau_0$ such that $0<\tau\leq \tau_0$, the numerical solution satisfies the error estimates as follows:
		\begin{flalign}
			\max_{0\leq n\leq N-1}\|\phi^{n+1}-\phi_h^{n+1}\|^2+\tau\sum_{n=0}^{N-1}\|\mu^{n+1}-\mu_h^{n+1}\|^2\leq C\left(\tau^2+h^{2(r+1)}\right),\\
			\max_{0\leq n\leq N-1}\|\mathbf{u}^{n+1}-\mathbf{u}_h^{n+1}\|^2\leq C\left(\tau^2+\mathcal{E}_h^2\right),\\
			\tau\sum_{n=0}^{N-1}\|\nabla\left(\mathbf{u}^{n+1}-\mathbf{u}_h^{n+1}\right)\|^2\leq C\left(\tau^2+h^{2(r+1)}\right),\\
			\tau\sum_{n=0}^{N-1}\|p^{n+1}-p_h^{n+1}\|^2\leq C\left(\tau^2+h^{2(r+1)}\right),\\
			\max_{0\leq n\leq N-1} |\rho^{n+1}-\rho_h^{n+1}|^2\leq C\left(\tau^2+h^{2(r+1)}\right).
		\end{flalign}
	where $C$ is a positive constant independent of  $\tau$ and $h$, and $\mathcal{E}_h$ is defined by \eqref{eq_mathcal_E_h}.
	\end{Theorem}
	\begin{Remark}
		For $r = 1$, the error estimate for the velocity is one order lower than the interpolation, while the numerical experiments validate that this estimation is indeed optimal \cite{2023_CaiWentao_Optimal_L2_error_estimates_of_unconditionally_stable_finite_element_schemes_for_the_Cahn_Hilliard_Navier_Stokes_system}.
	\end{Remark}
	Next, we need several lemmas in order to prove the Theorem \ref{theorem_error_estimates_u_phi_mu_r}.
	\begin{Lemma}\label{Lemma_0401}
		For all $m\geq 0$ and for any $h,\tau>0$, there exists $C>0$, independent of $h$ and $\tau$, such that the truncation errors satisfy
		\begin{equation}
			\label{eq_truncation_errors_inequation}
			\tau\sum_{n=0}^{m}\left(\|R_1^{n+1}\|^2+\|R_2^{n+1}\|^2+\|R_3^{n+1}\|^2+\|R_4^{n+1}\|^2\right)\leq C\tau^2.
		\end{equation}
	\end{Lemma}
	The proof of this lemma can be completed by the methods analogous to that used in Lemma 3.1 of \cite{2015_Diegel_Analysis_of_a_mixed_finite_element_method_for_a_Cahn_Hilliard_Darcy_Stokes_system} 
	and \cite{2024_YiNianyu_Convergence_analysis_of_a_decoupled_pressure_correction_SAV_FEM_for_the_Cahn_Hilliard_Navier_Stokes_model}. Here, we omit the details.
	\begin{Lemma} \label{Lemma_0402}
		Suppose $g\in H^1(\Omega)$ and $v\in \mathring{S}_h^r$. Then
		\begin{equation}
			|(g,v)|\leq C\|\nabla g\|\|v\|_{H^{-1}},
		\end{equation}
		where $C$ is positive constant independent of $h$.
	\end{Lemma}
	The proof of the Lemma \ref{Lemma_0402} is similar to Lemma 3.2 of  
	\cite{2015_Diegel_Analysis_of_a_mixed_finite_element_method_for_a_Cahn_Hilliard_Darcy_Stokes_system}, here we omit it. 
	\begin{Lemma}\label{Lemma_0403}
		With the regularities \eqref{eq_varibles_satisfied_regularities}, for any $h,\tau >0$, suppose that $\left(\phi,\mu,\mathbf{u},p\right)$ is weak solution to \eqref{eq_fully_discrete_CHNS_scheme_phi}-\eqref{eq_fully_discrete_CHNS_scheme_tilde_incompressible_condition}, we have 
		\begin{equation}
			\|\nabla\left((\phi^{n})^3-(\phi_h^{n})^3\right)\|\leq C\|\nabla \left(\phi^{n}-\phi_h^{n}\right)\|.
		\end{equation}
		where $n=0,1,2,\cdots, N-1.$
	\end{Lemma}
	\begin{proof}
		For $C$ is independent of $\tau$ and $h$.
		\begin{equation}
			\begin{aligned}
				\|\nabla\left((\phi^{n})^3-(\phi_h^{n})^3\right)\| \leq &~3\|(\phi_h^{n})^2\nabla \left(\phi^{n}-\phi_h^{n}\right)\|+3\|\nabla\phi^{n}\left(\phi^{n}+\phi_h^{n}\right)\left(\phi^{n}-\phi_h^{n}\right)\|\\
				\leq &~3\|\phi_h^{n}\|_{L^{\infty}}^2\|\nabla \left(\phi^{n}-\phi_h^{n}\right)\|\notag\\
				&+ 	3\|\nabla \phi^{n}\|_{L^6}\|\phi^{n}+\phi_h^{n}\|_{L^6}\|\phi^{n}-\phi_h^{n}\|_{L^6}\\
				\leq &~3\left(\|\phi_h^{n}\|_{L^{\infty}}^2+C\|\nabla \phi^{n}\|_{L^6}\|\phi^{n}+\phi_h^{n}\|_{H^1}\right)\|\nabla \left(\phi^{n}-\phi_h^{n}\right)\|\\
				\leq &~C\|\nabla \left(\phi^{n}-\phi_h^{n}\right)\|.
			\end{aligned}
		\end{equation}
		Then, using the Lemma \ref{Lemma_0301} and the regularities \eqref{eq_varibles_satisfied_regularities}, we can obtain the result.
	\end{proof}
	\subsection{Estimates for $e_\phi^{n+1}$ and $e_\mu^{n+1}$}
	The aim of subsection is to estimate the errors of $\phi$ and $\mu$. 
	To obtain the estimates for $e_\phi^{n+1}$ and $e_\mu^{n+1}$, testing $w_h=e_\mu^{n+1}$ and $\varphi_h=\delta_\tau e_{\phi}^{n+1}$ in \eqref{eq_error_equations_phi} and \eqref{eq_error_equations_mu}, we can obtain
	\begin{equation}
		\label{eq_error_equations_phi_test_e_mu}
		\begin{aligned}
			&\left(\delta_\tau e_\phi^{n+1},e_\mu^{n+1}\right)+M\left(\nabla e_\mu^{n+1},\nabla e_\mu^{n+1}\right)\\
				=&~b\left(\phi_h^{n+1},\mathbf{u}_h^n,e_\mu^{n+1}\right)-b\left(R_h\phi^{n+1},\mathbf{u}^n,e_\mu^{n+1}\right)
				+\left(\delta_\tau\left(R_h\phi^{n+1}-\phi^{n+1}\right),e_\mu^{n+1}\right)+\left(R_1^{n+1},e_\mu^{n+1}\right),
		\end{aligned}
	\end{equation}
	and
	\begin{equation}
		\label{eq_error_equations_mu_test_delta_e_mu}
		\begin{aligned}
			&-\left(e_\mu^{n+1},\delta_\tau e_\phi^{n+1}\right)+\lambda\left(\nabla e_\phi^{n+1},\nabla\delta_\tau e_\phi^{n+1}\right)+\lambda\left(\nabla\left(\phi^{n+1}-R_h\phi^{n+1}\right),\nabla\delta_\tau e_\phi^{n+1}\right)\\
			=&\left(\mu^{n+1}-\Pi_h\mu^{n+1},\delta_\tau e_\phi^{n+1}\right)
			+\frac{\lambda}{\epsilon^2}\left((\phi_h^{n})^3-(\phi^{n})^3+e_\phi^{n}+\phi^n-R_h\phi^n,\delta_\tau e_\phi^{n+1}\right).
		\end{aligned}
	\end{equation}
	Adding the equation \eqref{eq_error_equations_phi_test_e_mu} to \eqref{eq_error_equations_mu_test_delta_e_mu}, we get
	\begin{equation}
		\label{eq_error_phi_mu}
		\begin{aligned}
			&\frac{\lambda}{2}\delta_\tau\|\nabla e_\phi^{n+1}\|^2+\frac{\lambda}{2\tau}\bigg\|\nabla\left(e_\phi^{n+1}-e_\phi^n\right)\bigg\|^2+M\|\nabla e_\mu^{n+1}\|^2\\
			=&\left(\delta_\tau\left(R_h\phi^{n+1}-\phi^{n+1}\right),e_\mu^{n+1}\right)
				+\left(b\left(\phi_h^{n+1},\mathbf{u}_h^n,e_\mu^{n+1}\right)-b\left(R_h\phi^{n+1},\mathbf{u}^n,e_\mu^{n+1}\right)\right)+\left(R_1^{n+1},e_\mu^{n+1}\right)\\
				&+\frac{\lambda}{\epsilon^2}\left(e_\phi^{n},\delta_\tau e_\phi^{n+1}\right)
				+\frac{\lambda}{\epsilon^2}\left((\phi_h^{n})^3-(\phi^{n})^3,\delta_\tau e_\phi^{n+1}\right)
				+\frac{\lambda}{\epsilon^2}\left(\phi^n-R_h\phi^n,\delta_\tau e_\phi^{n+1}\right)\\
				&+\left(\mu^{n+1}-\Pi_h\mu^{n+1},\delta_\tau e_\phi^{n+1}\right)
					+\lambda\left(\nabla\left(R_h\phi^{n+1}-\phi^{n+1}\right),\nabla\delta_\tau e_\phi^{n+1}\right).
		\end{aligned}
	\end{equation}
	Now, we estimate the right-hand side terms of \eqref{eq_error_phi_mu}, respectively. 
	By \eqref{eq_Dtau_psi_Rhpsi_Hneg1_norm_inequation}, it is easy for us to see that
	\begin{equation}
		\left(\delta_\tau\left(R_h\phi^{n+1}-\phi^{n+1}\right),e_\mu^{n+1}\right)\leq 
			C\|\delta_\tau\left(R_h\phi^{n+1}-\phi^{n+1}\right)\|_{H^{-1}}\|e_\mu^{n+1}\|_{H^1}
			\leq C\varepsilon^{-1}\mathcal{E}_h^{2}+\varepsilon\|e_\mu^{n+1}\|_{H^1}^2.
	\end{equation}
	By \eqref{eq_trilinear_forms} and error estimates \eqref{eq_stokes_quasi_proojection_inequation0005}, we can obtain
	\begin{equation}
		\begin{aligned}
			&b\left(\phi_h^{n+1},\mathbf{u}_h^n,e_\mu^{n+1}\right)-b\left(R_h\phi^{n+1},\mathbf{u}^n,e_\mu^{n+1}\right)\\
			=&~\left(\nabla\phi_h^{n+1}\cdot\mathbf{u}_h^n,e_\mu^{n+1}\right)-\left(\nabla R_h\phi^{n+1}\cdot\mathbf{u}^n,e_\mu^{n+1}\right)\\
			=&~\left(\nabla\phi_h^{n+1}\cdot(\mathbf{u}_h^n-\mathbf{u}^n),e_\mu^{n+1}\right)
				+\left(\nabla (\phi_h^{n+1}-R_h\phi^{n+1})\cdot\mathbf{u}^n,e_\mu^{n+1}\right)\\
			\leq&~C\left(\|e_\mathbf{u}^n+\mathbf{u}^n-\mathbf{P}_h\mathbf{u}^n\|+\|\nabla e_\phi^{n+1}\|\right)\|e_\mu^{n+1}\|_{L^6}\\
			\leq&~C\varepsilon^{-1}\left(\|e_\mathbf{u}^n\|^2+\|\nabla e_\phi^{n+1}\|^2+\mathcal{E}_h^2+\tau^2\right)+\varepsilon\|e_\mu^{n+1}\|_{H^1}^2.
		\end{aligned}
	\end{equation}
	We now proceed to estimate the third term at the right-hand side of \eqref{eq_error_phi_mu}.
	In order to estimate this term, we need to the following lemma.
	\begin{Lemma}
		For any $h,\tau>0$, there exists $C>0$, independent of $h$ and $\tau$, such that
		\begin{flalign}
			\|R_1^{n+1}\|^2\leq C\frac{h^{2r+2}}{\tau}\int_{t^n}^{t^{n+1}}\|\partial_t\phi(s)\|_{H^{r+1}}^2~ds
				+\frac{\tau}{3}\int_{t^n}^{t^{n+1}}\|\partial_{tt}\phi(s)\|^2~d s\leq C\tau^2,
		\end{flalign}
	\end{Lemma}
	\begin{proof}
		 We can rewrite $R_1^{n+1}=\delta_\tau R_h\phi^{n+1}-\delta_\tau\phi^{n+1}+\delta_\tau\phi^{n+1}-\partial_t\phi^{n+1}$. Then, it follows that 
		 \begin{equation}
		 	\|\delta_\tau R_h\phi^{n+1}-\delta_\tau\phi^{n+1}\|^2\leq C\frac{h^{2r+2}}{\tau}\int_{t^n}^{t^{n+1}}\|\partial_t\phi(s)\|_{H^{r+1}}^2~ds.
		 \end{equation}
	 By Taylor's theorem, we have
	 \begin{equation}
	 	\|\delta_\tau\phi^{n+1}-\partial_t\phi^{n+1}\|^2\leq\frac{\tau^3}{3\tau^2}\int_{t^n}^{t^{n+1}}\|\partial_{tt}\phi(s)\|^2~d s=\frac{\tau}{3}\int_{t^n}^{t^{n+1}}\|\partial_{tt}\phi(s)\|^2~d s.
	 \end{equation}
	Using the triangle inequality and Taylor expansion, the result for $\|R_1^{n+1}\|^2$ follows.
	\end{proof}
 	Thus, using Lemma \ref{Lemma_0401}, the Cauchy-Schwarz inequality, and the fact that $\left(R_1^{n+1},1\right)=0$, we can obtain the following estimates:
 	\begin{equation}
 		\begin{aligned}
 			\bigg|\left(R_1^{n+1},e_\mu^{n+1}\right)\bigg|&\leq C\|R_1^{n+1}\|\|e_\mu^{n+1}\|\leq C\varepsilon^{-1}\|R_1^{n+1}\|^2+\varepsilon\| e_\mu^{n+1}\|^2\\
 			&\leq C\left(\mathcal{E}_h^2+\tau^2\right)+\varepsilon\| e_\mu^{n+1}\|^2.\\
 		\end{aligned}
 	\end{equation}
	Nextly, we have
	\begin{equation}
		\begin{aligned}
			\frac{\lambda}{\epsilon^2}\left(e_\phi^{n},\delta_\tau e_\phi^{n+1}\right)
			&=\frac{\lambda}{\tau\epsilon^2}\left(e_\phi^{n},e_\phi^{n+1}-e_\phi^{n}\right)\\
			&=-\frac{\lambda}{\tau\epsilon^2}\left(\|e_\phi^{n}\|^2-\|e_\phi^{n+1}\|^2+\|e_\phi^n-e_\phi^{n+1}\|^2\right)\leq \frac{C}{2}\delta_\tau\|e_\phi^{n+1}\|^2.
		\end{aligned}
	\end{equation}
	For the error estimation of the term $\frac{\lambda}{\epsilon^2}\left((\phi_h^{n})^3-(\phi^{n})^3,\delta_\tau e_\phi^{n+1}\right)$, we need to the above lemmas.
	
	Now, using Lemmas \ref{Lemma_0402} and \ref{Lemma_0403}, we can obtain
	\begin{equation}
		\begin{aligned}
			\bigg|\frac{\lambda}{\epsilon^2}\left((\phi_h^{n})^3-(\phi^{n})^3,\delta_\tau e_\phi^{n+1}\right)\bigg|&\leq C\|\nabla\left((\phi_h^{n})^3-(\phi^{n})^3\right)\|\|\delta_\tau e_\phi^{n+1}\|_{H^{-1}}\\
			&\leq C\|\nabla\left((\phi_h^{n})^3-(\phi^{n})^3\right)\|^2+\frac{\varepsilon}{2}\|\delta_\tau e_\phi^{n+1}\|_{H^{-1}}^2\\
			&\leq C\|\nabla \left(\phi_h^{n}-\phi^{n}\right)\|^2+\frac{\varepsilon}{2}\|\delta_\tau e_\phi^{n+1}\|_{H^{-1}}^2\\
			&\leq C\|\nabla\left(\phi^{n+1}-R_h\phi^{n+1}\right)\|^2+C\|\nabla e_\phi^{n+1}\|^2+\frac{\varepsilon}{2}\|\delta_\tau e_\phi^{n+1}\|_{H^{-1}}^2\\
			&\leq Ch^{2r}\|\phi^{n+1}\|_{H^{r+1}}^2+C\|\nabla e_\phi^{n+1}\|^2+\frac{\varepsilon}{2}\|\delta_\tau e_\phi^{n+1}\|_{H^{-1}}^2.
		\end{aligned}
	\end{equation}
	For the two sequences $\left\{f^{i}\right\}_{i=1}^{n}$ and $\left\{g^{i}\right\}_{i=1}^{n}$, we have the following result:
	\begin{equation}\label{eq_two_sequences_error_estimate}
		\tau\sum_{i=1}^{n}\left(f^{i},\delta_\tau g^{i}\right)=-\tau\sum_{i=2}^{n}\left(\delta_\tau f^{i},g^{i-1}\right)+\left(f^n,g^n\right)-\left(f^1,g^0\right).
	\end{equation}
	 By \eqref{eq_two_sequences_error_estimate} and \eqref{eq_Dtau_psi_Rhpsi_Hneg1_norm_inequation}, we can obatin
	 \begin{equation}
	 	\begin{aligned}
	 		\tau\sum_{i=1}^{n}\bigg|\frac{\lambda}{\epsilon^2}\left(\phi^i-R_h\phi^i,\delta_\tau e_\phi^{i+1}\right)\bigg|
	 		=&~-\frac{\lambda\tau}{\epsilon^2}\sum_{i=1}^{n}\left(\delta_\tau\left(\phi^i-R_h\phi^i\right), e_\phi^{i}\right)+\left(\phi^n-R_h\phi^n,e_\phi^{n+1}\right)\\
	 		\leq&~C\tau\sum_{i=1}^{n}\|\delta_\tau\left(\phi^i-R_h\phi^i\right)\|_{H^{-1}}\|e_\phi^{i}\|_{H^{1}}+\|\phi^n-R_h\phi^n\|_{H^{-1}}\|e_\phi^{n+1}\|_{H^{1}}\\
	 		\leq&~C\tau\sum_{i=1}^{n}\|e_\phi^{i}\|_{H^1}^2+\varepsilon\|e_\phi^{n+1}\|_{H^1}^2+C\mathcal{E}_h^2,
	 	\end{aligned}
	 \end{equation}
 	where $e_\phi^0=0$.
	For the term $\left(\mu^{n+1}-\Pi_h\mu^{n+1},\delta_\tau e_\phi^{n+1}\right)$, we have
	\begin{equation}
		\begin{aligned}
			\bigg|\left(\mu^{n+1}-\Pi_h\mu^{n+1},\delta_\tau e_\phi^{n+1}\right)\bigg|&\leq C\|\nabla(\mu^{n+1}-\Pi_h\mu^{n+1})\|\|\delta_\tau e_\phi^{n+1}\|_{H^{-1}}\\
			&\leq Ch^r\|\mu^{n+1}\|_{H^{r+1}}\|\delta_\tau e_\phi^{n+1}\|_{H^{-1}}\\
			&\leq Ch^r\|\mu^{n+1}\|_{H^{r+1}}^2+\frac{\varepsilon}{2}\|\delta_\tau e_\phi^{n+1}\|_{H^{-1}}^2.
		\end{aligned}
	\end{equation}
	Using the classic Ritz projection \eqref{eq_Ritz_projection}, we can obtain
	\begin{equation}
			\lambda\left(\nabla\left(R_h\phi^{n+1}-\phi^{n+1}\right),\nabla\delta_\tau e_\phi^{n+1}\right)=0.
	\end{equation}
	Thus, combining the above results and summing up for $n$ from $0$ to $m$, we can obtain
	\begin{equation}
		\label{eq_error_estimate_phi_mu_H1_L2}
		\begin{aligned}
			\|\nabla e_\phi^{m}\|^2&+\tau\sum_{n=0}^{m}\|\nabla e_\mu^{n+1}\|^2\leq C\tau\sum_{n=0}^{m-1}\left(\|e_\mathbf{u}^{n+1}\|^2+\|\nabla e_\phi^{n+1}\|^2+\|e_\phi^n\|_{H^1}^2\right)+C\left(\mathcal{E}_h^2+\tau^2\right)\\
			&+\varepsilon\tau\sum_{n=0}^{m}\left(\|e_\mu^{n}\|_{H^1}^2+\|e_\mu^{n}\|^2+\|\delta_\tau e_\phi^{n+1}\|_{H^{-1}}^2\right)+\|e_\phi^{N}\|^2+2\varepsilon\|e_\phi^{N}\|_{H^1}^2\\
			\leq &~C_\varepsilon\tau\sum_{n=0}^{m}\left(\|e_\mathbf{u}^{n+1}\|^2+\|\nabla e_\phi^{n+1}\|^2+\|e_\phi^n\|_{H^1}^2\right)+C\left(\mathcal{E}_h^2+\tau^2\right)\\
			&+\varepsilon\tau\sum_{n=0}^{m}\left(\|e_\mu^{n}\|_{H^1}^2+\|e_\mu^{n}\|^2+\|\delta_\tau e_\phi^{n+1}\|_{H^{-1}}^2\right)+C_{\varepsilon}\|e_\phi^{m}\|^2.\\
		\end{aligned}
	\end{equation} 
	Next, we will analyze the error estimates for $e_\mathbf{u}^{n+1}$.
	\subsection{Estimate for $e_\mathbf{u}^{n+1}$}
	Testing $\mathbf{v}_h=e_{\tilde{\mathbf{u}}}^{n+1}$ in \eqref{eq_error_equations_ns}, we can obtain
	\begin{equation}
		\label{eq_error_equatio_u}
		\begin{aligned}
			&\frac{1}{2\tau}\left(\|e_{\tilde{\mathbf{u}}}^{n+1}\|^2-\|e_\mathbf{u}^n\|^2+\|e_{\tilde{\mathbf{u}}}^{n+1}-e_\mathbf{u}^n\|^2\right)
				+\nu\|\nabla e_{\tilde{\mathbf{u}}}^{n+1}\|^2+\left(\nabla e_p^{n},e_{\tilde{\mathbf{u}}}^{n+1}\right)\\
		=&~\frac{\rho_h^n}{\sqrt{E_{1,h}^{n+1}}}\left(\mathbf{u}_h^{n}\cdot\nabla\mathbf{u}_h^{n},e_{\tilde{\mathbf{u}}}^{n+1}\right)-\frac{\rho^n}{\sqrt{E_{1}^{n+1}}}\left(\mathbf{u}^{n+1}\cdot\nabla\mathbf{u}^{n+1},e_{\tilde{\mathbf{u}}}^{n+1}\right)
				+\left(\delta_\tau\left(\mathbf{P}_h\mathbf{u}^{n+1}-\mathbf{u}^{n+1}\right), e_{\tilde{\mathbf{u}}}^{n+1}\right)\\
			&-\left(b\left(\phi_h^{n+1},\mu_h^{n+1},e_{\tilde{\mathbf{u}}}^{n+1}\right)-b\left(R_h\phi^{n+1},\mu^{n+1},e_{\tilde{\mathbf{u}}}^{n+1}\right)\right)+\left(R_2^{n+1},e_{\tilde{\mathbf{u}}}^{n+1}\right).
		\end{aligned}
	\end{equation}
	By \eqref{eq_error_equations_u_p_cor} and \eqref{eq_error_equations_incompressible},
	we can derive that
	\begin{equation}
		e_{\tilde{\mathbf{u}}}^{n+1}=e_{\mathbf{u}}^{n+1}+\tau\left(\nabla e_p^{n+1}-\nabla e_p^{n}\right)-\tau R_3^{n+1}
	\end{equation}
	and
	\begin{equation}
		\label{eq_nabla_e_p_e_tileun1}
		\begin{aligned}
			\left(\nabla e_p^n,e_{\tilde{\mathbf{u}}}^{n+1}\right)&=\left(\nabla e_p^n,e_{\mathbf{u}}^{n+1}+\tau\left(\nabla e_p^{n+1}-\nabla e_p^{n}\right)-\tau R_3^{n+1}\right)\\
			&=\frac{\tau}{2}\left(\|\nabla e_p^{n+1}\|^2-\|\nabla e_p^n\|^2-\|\nabla e_p^{n+1}-\nabla e_p^n\|^2\right)-\tau\left(R_3^{n+1},\nabla e_p^n\right).
		\end{aligned}
	\end{equation}
	 The equation \eqref{eq_error_equations_u_p_cor} is rewritten as  
	\begin{equation}
		\label{eq_error_equations_u_p_cor_rewrite}
		\left(e_p^{n+1}-e_p^n\right)+\frac{1}{\tau}e_\mathbf{u}^{n+1}=R_3^{n+1}+\frac{1}{\tau}e_{\tilde{\mathbf{u}}}^{n+1}.
	\end{equation}
	For using the \eqref{eq_error_equations_incompressible} and taking the inner product of \eqref{eq_error_equations_u_p_cor_rewrite} with itself on both sides, we have
	\begin{equation}
		\label{eq_inner_product_0436_0000}
		\frac{\tau}{2}\|e_p^{n+1}-e_p^n\|^2=\frac{\tau}{2}\|R_3^{n+1}\|^2+\frac{1}{2\tau}\left(\|e_{\tilde{\mathbf{u}}}^{n+1}\|^2-\|e_\mathbf{u}^{n+1}\|^2\right)+\left(R_3^{n+1},\nabla e_p^n\right).
	\end{equation}
	Combining \eqref{eq_error_equatio_u} with \eqref{eq_nabla_e_p_e_tileun1} and \eqref{eq_inner_product_0436_0000}, we obtain
	\begin{equation}
		\label{eq_error_equatio_u_com}
		\begin{aligned}
			\frac{1}{2\tau}&\left(\|e_\mathbf{u}^{n+1}\|^2-\|e_\mathbf{u}^n\|^2+\|e_{\tilde{\mathbf{u}}}^{n+1}-e_\mathbf{u}^n\|^2\right)+\nu\|\nabla e_{\tilde{\mathbf{u}}}^{n+1}\|^2
				+\frac{\tau}{2}\left(\|\nabla e_p^{n+1}\|^2-\|\nabla e_p^n\|^2\right)\\
				=&~\frac{\rho_h^n}{\sqrt{E_{1,h}^{n+1}}}\left(\mathbf{u}_h^n\cdot\nabla\mathbf{u}_h^{n},e_{\tilde{\mathbf{u}}}^{n+1}\right)-\frac{\rho^{n+1}}{\sqrt{E_1^{n+1}}}\left(\mathbf{u}^{n+1}\cdot\nabla\mathbf{u}^{n+1},e_{\tilde{\mathbf{u}}}^{n+1}\right)
				+\left(\delta_\tau\left(\mathbf{P}_h\mathbf{u}^{n+1}-\mathbf{u}^{n+1}\right), e_{\tilde{\mathbf{u}}}^{n+1}\right)\\
				&-\left(b\left(\phi_h^{n+1},\mu_h^{n+1},e_{\tilde{\mathbf{u}}}^{n+1}\right)-b\left(R_h\phi^{n+1},\mu^{n+1},e_{\tilde{\mathbf{u}}}^{n+1}\right)\right)+\left(R_2^{n+1},e_{\tilde{\mathbf{u}}}^{n+1}\right)\\
				&+\frac{\tau}{2}\|R_3^{n+1}\|^2+\left(R_3^{n+1},e_{\tilde{\mathbf{u}}}^{n+1}\right)+\tau\left(R_3^{n+1},\nabla e_p^n\right).
		\end{aligned}
	\end{equation}
	Then, we estimate the above error equation \eqref{eq_error_equatio_u_com}, 
	\begin{equation}
		\label{eq_error_estimate_b_stars}
		\begin{aligned}
			&\bigg|\frac{\rho_h^n}{\sqrt{E_{1,h}^{n+1}}}\left(\mathbf{u}_h^n\cdot\nabla\mathbf{u}_h^{n},e_{\tilde{\mathbf{u}}}^{n+1}\right)
			-\frac{\rho^{n+1}}{\sqrt{E_1^{n+1}}}\left(\mathbf{u}^{n+1}\cdot\nabla\mathbf{u}^{n+1},e_{\tilde{\mathbf{u}}}^{n+1}\right)\bigg|\\
				=~&\bigg|\left(\mathbf{u}^{n+1}\cdot\nabla\mathbf{u}^{n+1},e_{\tilde{\mathbf{u}}}^{n+1}\right)
					-\left(\mathbf{u}^n\cdot\nabla\mathbf{u}^n,e_{\tilde{\mathbf{u}}}^{n+1}\right)
					+\left(\mathbf{u}^n\cdot\nabla\mathbf{u}^{n},e_{\tilde{\mathbf{u}}}^{n+1}\right)-\left(\mathbf{u}_h^n\cdot\nabla\mathbf{u}_h^{n},e_{\tilde{\mathbf{u}}}^{n+1}\right)\\
					&+\frac{e_\rho^{n}}{\sqrt{E_1^{n+1}}}\left(\mathbf{u}_h^n\cdot\nabla\mathbf{u}_h^n,e_{\tilde{\mathbf{u}}}^{n+1}\right)
						+\left(\frac{\rho_h^n}{\sqrt{E_1^{n+1}}}-\frac{\rho_h^n}{\sqrt{E_{1,h}^{n+1}}}\right)\left(\mathbf{u}_h^n\cdot\nabla\mathbf{u}_h^n,e_{\tilde{\mathbf{u}}}^{n+1}\right)\bigg|
				=:\sum_{i=1}^{4}A_{i}.
		\end{aligned}
	\end{equation}
	According to \cite{2018_Shenjie_Xujie_CAEAFTSAVSYGF}, we have
	\begin{equation}
		\label{eq_sav_inequalities}
		\frac{1}{\sqrt{E_1^{n+1}}}-\frac{1}{\sqrt{E_{1,h}^{n+1}}}=
		\frac{|E_{1,h}^{n+1}-E_1^{n+1}|}{\sqrt{E_1^{n+1}E_{1,h}^{n+1}}\left(\sqrt{E_1^{n+1}}+\sqrt{E_{1,h}^{n+1}}\right)}
		\leq C\|e_{\phi}^{n+1}\|.
	\end{equation}
	From the Taylor expansion and \eqref{eq_sav_inequalities}, we can obtain estimates as follows.
	\begin{equation}
		\begin{aligned}
			|A_1|=~&\bigg|\left(\mathbf{u}^{n+1}\cdot\nabla\mathbf{u}^{n+1},e_{\tilde{\mathbf{u}}}^{n+1}\right)
			-\left(\mathbf{u}^n\cdot\nabla\mathbf{u}^n,e_{\tilde{\mathbf{u}}}^{n+1}\right)\bigg|\\
			=~&\bigg|\left(\left(\mathbf{u}^{n+1}-\mathbf{u}^n\right)\cdot\nabla\mathbf{u}^{n+1}
				+\mathbf{u}^n\cdot\nabla\left(\mathbf{u}^{n+1}-\mathbf{u}^n\right),e_{\tilde{\mathbf{u}}}^{n+1}
		\right)\bigg|\\
		\leq~&\frac{C}{4}\|\nabla e_{\tilde{\mathbf{u}}}^{n+1}\|^2+C\left(\mathcal{E}_h^2+\tau^2\right),
		\end{aligned}
	\end{equation}
	\begin{equation}
		\begin{aligned}
			|A_2|=~&\bigg|\left(\mathbf{u}^n\cdot\nabla\mathbf{u}^{n},e_{\tilde{\mathbf{u}}}^{n+1}\right)-\left(\mathbf{u}_h^n\cdot\nabla\mathbf{u}_h^{n},e_{\tilde{\mathbf{u}}}^{n+1}\right)\bigg| \\
			=~&\bigg|\left(\left(\mathbf{u}^n-\mathbf{P}_h\mathbf{u}^n\right)\cdot\nabla\mathbf{u}^n,e_{\tilde{\mathbf{u}}}^{n+1}\right)
				+\left(\mathbf{P}_h\mathbf{u}^n\cdot\nabla\left(\mathbf{u}^n-\mathbf{u}_h^n\right),e_{\tilde{\mathbf{u}}}^{n+1}\right)\\
				&-\left(e_\mathbf{u}^n\cdot\nabla e_\mathbf{u}^n,e_{\tilde{\mathbf{u}}}^{n+1}\right)
				-\left(e_\mathbf{u}^n\cdot\nabla\mathbf{P}_h\mathbf{u}^n,e_{\tilde{\mathbf{u}}}^{n+1}\right)
				+\left(\mathbf{P}_h\mathbf{u}^n\cdot\nabla e_\mathbf{u}^n,e_{\tilde{\mathbf{u}}}^{n+1}\right)
				-\left(\mathbf{P}_h\mathbf{u}^n\cdot\nabla \mathbf{P}_h\mathbf{u}^n,e_{\tilde{\mathbf{u}}}^{n+1}\right)\bigg|\\
			=~&\bigg|\left(\left(\mathbf{u}^n-\mathbf{P}_h\mathbf{u}^n\right)\cdot\nabla\mathbf{u}^n,e_{\tilde{\mathbf{u}}}^{n+1}\right)
				+\left(\mathbf{P}_h\mathbf{u}^n\cdot\nabla\left(\mathbf{u}^n-\mathbf{P}_h\mathbf{u}^n\right),e_{\tilde{\mathbf{u}}}^{n+1}\right)
				+\left(\mathbf{P}_h\mathbf{u}^n\cdot\nabla e_\mathbf{u}^n,e_{\tilde{\mathbf{u}}}^{n+1}\right)
				\\
				&-\left(e_\mathbf{u}^n\cdot\nabla e_\mathbf{u}^n,e_{\tilde{\mathbf{u}}}^{n+1}\right)
				-\left(e_\mathbf{u}^n\cdot\nabla\mathbf{P}_h\mathbf{u}^n,e_{\tilde{\mathbf{u}}}^{n+1}\right)
				+\left(\mathbf{P}_h\mathbf{u}^n\cdot\nabla e_\mathbf{u}^n,e_{\tilde{\mathbf{u}}}^{n+1}\right)
				-\left(\mathbf{P}_h\mathbf{u}^n\cdot\nabla \mathbf{P}_h\mathbf{u}^n,e_{\tilde{\mathbf{u}}}^{n+1}\right)\bigg|\\
			\leq~&C\|\mathbf{u}^n-\mathbf{P}_h\mathbf{u}^n\|\|\mathbf{u}^n\|_{L^\infty}\|\nabla e_{\tilde{\mathbf{u}}}^{n+1}\|
				+C\|\mathbf{u}^n-\mathbf{P}_h\mathbf{u}^n\|\|\mathbf{P}_h\mathbf{u}^n\|_{L^\infty}\|\nabla e_{\tilde{\mathbf{u}}}^{n+1}\|\\
				&+C\|e_\mathbf{u}^n\|\|\mathbf{P}_h\mathbf{u}^n\|_{L^\infty}\|\nabla e_{\tilde{\mathbf{u}}}^{n+1}\|
				+C\|e_\mathbf{u}^n\|\|e_\mathbf{u}^n\|_{L^\infty}\|\nabla e_{\tilde{\mathbf{u}}}^{n+1}\|
				+\|\mathbf{P}_h\mathbf{u}^n\|\|\mathbf{P}_h\mathbf{u}^n\|_{L^\infty}\|\nabla e_{\tilde{\mathbf{u}}}^{n+1}\|\\
			\leq ~&\frac{C}{4}\|\nabla e_{\tilde{\mathbf{u}}}^{n+1}\|^2+\frac{C}{3}\|e_\mathbf{u}^n\|^2+C\left(\mathcal{E}_h^2+\tau^2\right),
		\end{aligned}
	\end{equation}
	\begin{equation}
		\begin{aligned}
			|A_3|=~&\bigg|\frac{e_\rho^{n}}{\sqrt{E_1^{n+1}}}\left(\mathbf{u}_h^n\cdot\nabla\mathbf{u}_h^n,
							e_{\tilde{\mathbf{u}}}^{n+1}\right)\bigg|\\
				 = ~&\bigg|\frac{e_\rho^{n}}{\sqrt{E_1^{n+1}}}\left(e_\mathbf{u}^n\cdot\nabla e_\mathbf{u}^n,e_{\tilde{\mathbf{u}}}^{n+1}\right)
				 	-\frac{e_\rho^{n}}{\sqrt{E_1^{n+1}}}\left(e_\mathbf{u}^n\cdot\nabla \mathbf{P}_h\mathbf{u}^n,e_{\tilde{\mathbf{u}}}^{n+1}\right)\\
				 	&+\frac{e_\rho^{n}}{\sqrt{E_1^{n+1}}}\left(\mathbf{P}_h\mathbf{u}^n\cdot\nabla e_\mathbf{u}^n,e_{\tilde{\mathbf{u}}}^{n+1}\right)
				 	-\frac{e_\rho^{n}}{\sqrt{E_1^{n+1}}}\left(\mathbf{P}_h\mathbf{u}^n\cdot\nabla\mathbf{P}_h\mathbf{u}^n,e_{\tilde{\mathbf{u}}}^{n+1}\right)\bigg|\\
				\leq~&\frac{C|e_\rho^n|}{\sqrt{C_0}}\left(\|e_\mathbf{u}^n\|\|e_\mathbf{u}^n\|_{L^\infty}\|\nabla e_{\tilde{\mathbf{u}}}^{n+1}\|
						+\|e_\mathbf{u}^n\|\|\mathbf{P}_h\mathbf{u}^n\|_{L^\infty}\|\nabla e_{\tilde{\mathbf{u}}}^{n+1}\|\right.\\
						&\left.
						+\|e_\mathbf{u}^n\|\|\mathbf{P}_h\mathbf{u}^n\|_{L^\infty}\|\nabla e_{\tilde{\mathbf{u}}}^{n+1}\|
						+\|\mathbf{P}_h\mathbf{u}^n\|\|\mathbf{P}_h\mathbf{u}^n\|_{L^\infty}\|\nabla e_{\tilde{\mathbf{u}}}^{n+1}\|
						\right)\\
				\leq~&\frac{C}{4}\|\nabla e_{\tilde{\mathbf{u}}}^{n+1}\|^2+\frac{C}{3}\|e_\mathbf{u}^n\|^2+C\left(\tau^2+\mathcal{E}_h^2\right),
		\end{aligned}
	\end{equation}
	and
	\begin{equation}
		\begin{aligned}
			|A_4|=~&\bigg|\left(\frac{\rho_h^n}{\sqrt{E_1^{n+1}}}-\frac{\rho_h^n}{\sqrt{E_{1,h}^{n+1}}}\right)
							\left(\mathbf{u}_h^n\cdot\nabla\mathbf{u}_h^n,e_{\tilde{\mathbf{u}}}^{n+1}\right)
						\bigg|\\
					\leq ~&|\rho_h^n|\frac{|E_{1,h}^{n+1}-E_1^{n+1}|}{\sqrt{E_1^{n+1}E_{1,h}^{n+1}}
							\left(\sqrt{E_1^{n+1}}+\sqrt{E_{1,h}^{n+1}}\right)}
							\left(\mathbf{u}_h^n\cdot\nabla\mathbf{u}_h^n,e_{\tilde{\mathbf{u}}}^{n+1}\right)\\
					\leq~&\frac{C}{4}\|\nabla e_{\tilde{\mathbf{u}}}^{n+1}\|^2+\frac{C}{3}\|e_\mathbf{u}^n\|^2
								+C\|e_\phi^{n+1}\|^2+|e_\rho^n|^2+C\left(\tau^2+\mathcal{E}_h^2\right).
		\end{aligned}
	\end{equation}
	Then, combining the above inequalities $A_1$, $A_2$, $A_3$ and $A_4$, we can obtain
	\begin{equation}
		\bigg|\sum_{i=1}^{4}A_i\bigg|\leq C\|\nabla e_{\tilde{\mathbf{u}}}^{n+1}\|^2+C\|e_\mathbf{u}^n\|^2
		+C\|e_\phi^{n+1}\|^2+|e_\rho^n|^2+C\left(\tau^2+\mathcal{E}_h^2\right).
	\end{equation}
	Using the Lemma \ref{Lemma_0401}, we have 
	\begin{equation}
		\left(R_2^{n+1},e_{\tilde{\mathbf{u}}}^{n+1}\right)\leq C\|R_2^{n+1}\|\|e_{\tilde{\mathbf{u}}}^{n+1}\|\leq C\|R_2^{n+1}\|\|\nabla e_{\tilde{\mathbf{u}}}^{n+1}\|\leq C\tau^2+\varepsilon\|\nabla e_{\tilde{\mathbf{u}}}^{n+1}\|^2,
	\end{equation}
	and
	\begin{equation}
		\left(R_3^{n+1},e_{\tilde{\mathbf{u}}}^{n+1}\right)\leq C\|R_3^{n+1}\|\|e_{\tilde{\mathbf{u}}}^{n+1}\|\leq C\|R_3^{n+1}\|\|\nabla e_{\tilde{\mathbf{u}}}^{n+1}\|\leq C\tau^2+\varepsilon\|\nabla e_{\tilde{\mathbf{u}}}^{n+1}\|^2
	\end{equation}
 	and
 	\begin{equation}
 		\tau\left(R_3^{n+1},\nabla e_p^{n+1}\right)\leq C\tau^2+\varepsilon\tau^2 \|\nabla e_{p}^{n}\|^2.
 	\end{equation}
  	Next, we deal the term $\left(\delta_\tau\left(\mathbf{P}_h\mathbf{u}^{n+1}-\mathbf{u}^{n+1}\right), e_{\tilde{\mathbf{u}}}^{n+1}\right)$ with
  	\begin{equation}
  		\left(\delta_\tau\left(\mathbf{P}_h\mathbf{u}^{n+1}-\mathbf{u}^{n+1}\right), e_{\tilde{\mathbf{u}}}^{n+1}\right)\leq C\|\delta_\tau\left(\mathbf{P}_h\mathbf{u}^{n+1}-\mathbf{u}^{n+1}\right)\|\|e_{\tilde{\mathbf{u}}}^{n+1}\|\leq C\left(\mathcal{E}_h^2+\|e_{\tilde{\mathbf{u}}}^{n+1}\|^2\right).
  	\end{equation}
  	Finally, by \eqref{eq_trilinear_forms},  the following term can be estimated by
  	\begin{equation}
  		\begin{aligned}
  			&b\left(R_h\phi^{n+1},\mu^{n+1},e_{\tilde{\mathbf{u}}}^{n+1}\right)-b\left(\phi_h^{n+1},\mu_h^{n+1},e_{\tilde{\mathbf{u}}}^{n+1}\right)\\
  			=&\left(\nabla R_h\phi^{n+1}\cdot e_{\tilde{\mathbf{u}}}^{n+1},\mu^{n+1}\right)-\left(\nabla\phi_h^{n+1}\cdot e_{\tilde{\mathbf{u}}}^{n+1},\mu_h^{n+1}\right)\\
  			=&\left(\nabla R_h\phi^{n+1}\cdot e_{\tilde{\mathbf{u}}}^{n+1},e_\mu^{n+1}\right)+\left(\nabla R_h\phi^{n+1}\cdot e_{\tilde{\mathbf{u}}}^{n+1},\mu^{n+1}-\Pi_h\mu^{n+1}\right)+\left(\nabla e_\phi^{n+1}\cdot e_{\tilde{\mathbf{u}}}^{n+1},\mu_h^{n+1}\right)\\
  			=&:\sum_{i=1}^{3}Tb_i.
  		\end{aligned}
  	\end{equation} 
  	We first estimate the term $Tb_1$ by obtain 
  	\begin{equation}
  		\begin{aligned}
  			\big|Tb_1\big|&=\left(\nabla\left(R_h\phi^{n+1}-\phi_h^{n+1}\right)\cdot e_{\tilde{\mathbf{u}}}^{n+1},e_\mu^{n+1}\right)+\left(\nabla \phi_h^{n+1}\cdot e_{\tilde{\mathbf{u}}}^{n+1},e_\mu^{n+1}\right)\\
  			&=\left(\nabla e_\phi^{n+1}\cdot e_{\tilde{\mathbf{u}}}^{n+1},e_\mu^{n+1}\right)+\left(\nabla \phi_h^{n+1}\cdot e_{\tilde{\mathbf{u}}}^{n+1},e_\mu^{n+1}\right)\\
  			&\leq~C\|\nabla e_\phi^{n+1}\|\|e_{\tilde{\mathbf{u}}}^{n+1}\|_{L^4}\|e_\mu^{n+1}\|+C\|\nabla \phi_h^{n+1}\|_{L^4}\|e_{\tilde{\mathbf{u}}}^{n+1}\|_{L^4}\|e_\mu^{n+1}\|\\
  			&\leq C\|\nabla e_\phi^{n+1}\|\|\nabla e_{\tilde{\mathbf{u}}}^{n+1}\|\|e_\mu^{n+1}\|+C\|\nabla \phi_h^{n+1}\|_{L^4}\|e_{\tilde{\mathbf{u}}}^{n+1}\|^{1/4}\|\nabla e_{\tilde{\mathbf{u}}}^{n+1}\|^{3/4}\|e_\mu^{n+1}\|\\
  			&\leq C_\varepsilon\left(\|\nabla e_\phi^{n+1}\|^2+\|e_{\tilde{\mathbf{u}}}^{n+1}\|^2\right)+\frac{\varepsilon}{2}\|\nabla e_{\tilde{\mathbf{u}}}^{n+1}\|^2+\varepsilon\|e_\mu^{n+1}\|^2.
  		\end{aligned}
  	\end{equation}
	  Using the error estimates \eqref{eq_psi_Rhpsi_Ls_norm_inequation}, \eqref{eq_psi_Rhpsi_Hneg1_norm_inequation}, \eqref{eq_lemma_mu_L2_inequation} and \eqref{eq_PW_inequalities_three},  we have 
	  \begin{equation}
	  	\begin{aligned}
	  		\big| Tb_2\big| =&~\left(\nabla \left(R_h\phi^{n+1}-\phi^{n+1}\right)\cdot e_{\tilde{\mathbf{u}}}^{n+1},\mu^{n+1}-\Pi_h\mu^{n+1}\right)
						+\left(\nabla \phi^{n+1}\cdot e_{\tilde{\mathbf{u}}}^{n+1},\mu^{n+1}-\Pi_h\mu^{n+1}\right)\\
			\leq& ~C\|\nabla \left(R_h\phi^{n+1}-\phi^{n+1}\right)\|_{L^3}\|e_{\tilde{\mathbf{u}}}^{n+1}\|_{L^6}\|\mu^{n+1}-\Pi_h\mu^{n+1}\|\\
				&+\|\nabla \phi^{n+1}\cdot e_{\tilde{\mathbf{u}}}^{n+1}\|_{H^1}\|\mu^{n+1}-\Pi_h\mu^{n+1}\|_{H^{-1}}\\
			\leq&~ Ch\|e_{\tilde{\mathbf{u}}}^{n+1}\|_{L^6}\|\mu^{n+1}-\Pi_h\mu^{n+1}\|\\
				&+C\left(\|\phi^{n+1}\|_{W^{2,3}}\|e_{\tilde{\mathbf{u}}}^{n+1}\|_{L^6}+\|\nabla \phi^{n+1}\|_{L^{\infty}}\|\nabla e_{\tilde{\mathbf{u}}}^{n+1}\|\right)\|\mu^{n+1}-\Pi_h\mu^{n+1}\|_{H^{-1}}\\
			\leq&~Ch\|\nabla e_{\tilde{\mathbf{u}}}^{n+1}\|\|\mu^{n+1}-\Pi_h\mu^{n+1}\|\\
					&+C\left(\|\phi^{n+1}\|_{W^{2,3}}\|\nabla e_{\tilde{\mathbf{u}}}^{n+1}\|+\|\nabla \phi^{n+1}\|_{L^{\infty}}\|\nabla e_{\tilde{\mathbf{u}}}^{n+1}\|\right)\|\mu^{n+1}-\Pi_h\mu^{n+1}\|_{H^{-1}}\\
			\leq&~\frac{\varepsilon}{2}\|\nabla e_{\tilde{\mathbf{u}}}^{n+1}\|^2+C_\varepsilon\left(\mathcal{E}_h^2+\tau^2\right).
  		\end{aligned}
	  \end{equation} 
  	Nextly, the  term $Tb_3$ can be estimated by
  	\begin{equation}
  		\begin{aligned}
  			\big|Tb_3\big|\leq C\|\nabla e_\phi^{n+1}\|\|e_{\tilde{\mathbf{u}}}^{n+1}\|\|\mu_h^{n+1}\|_{L^4}\leq C_\varepsilon\left(\|\nabla e_\phi^{n+1}\|^2+\|e_{\tilde{\mathbf{u}}}^{n+1}\|^2+\mathcal{E}_h^2+\tau^2\right).
  		\end{aligned}
  	\end{equation}
	Since, we can obtain
	\begin{equation}
		\sum_{i=1}^{3}\big|Tb_i\big|\leq C_\varepsilon\left(\|\nabla e_\phi^{n+1}\|^2+\|e_{\tilde{\mathbf{u}}}^{n+1}\|^2+\mathcal{E}_h^2+\tau^2\right)+\varepsilon\left(\|e_\mu^{n+1}\|^2+\|\nabla e_{\tilde{\mathbf{u}}}^{n+1}\|^2\right).
	\end{equation}
	The error estimate \eqref{eq_error_equatio_u_com} reduces to
	\begin{equation}
		\begin{aligned}
			\frac{1}{2\tau}&\left(\|e_\mathbf{u}^{n+1}\|^2-\|e_\mathbf{u}^n\|^2+\|e_{\tilde{\mathbf{u}}}^{n+1}-e_\mathbf{u}^n\|^2\right)+\nu\|\nabla e_{\tilde{\mathbf{u}}}^{n+1}\|^2
			+\frac{\tau}{2}\left(\|\nabla e_p^{n+1}\|^2-\|\nabla e_p^n\|^2\right)\\
			&\leq C_\varepsilon\left(\|\nabla e_\phi^{n+1}\|^2+\|e_{\tilde{\mathbf{u}}}^{n+1}\|^2+\|e_\mathbf{u}^n\|^2+\mathcal{E}_h^2+\tau^2\right)+\varepsilon\left(\|e_\mu^{n+1}\|^2+\|\nabla e_{\tilde{\mathbf{u}}}^{n+1}\|^2\right)+\varepsilon\tau^2\|\nabla e_p^n\|^2.
		\end{aligned}
	\end{equation}
	With the above error estimaties and a sufficiently small $\varepsilon$, summing from $n=0$ to $m$, $(m\leq N-1)$, we have
	\begin{equation}
		\label{eq_error_inequation_eu}
		\begin{aligned}
			&\|e_\mathbf{u}^{m}\|^2+2\nu\tau\sum_{n=0}^{m}\|\nabla e_{\tilde{\mathbf{u}}}^{n+1}\|^2+\sum_{n=0}^{m}\|e_{\tilde{\mathbf{u}}}^{n+1}-e_\mathbf{u}^n\|^2+\tau^2\|\nabla e_p^{m}\|^2\\
			\leq&~C_\varepsilon\tau\sum_{n=0}^{m}\left(\|\nabla e_{\phi}^{n+1}\|^2+\|e_\mathbf{u}^{n}\|^2+\|e_{\tilde{\mathbf{u}}}^{n+1}\|^2\right)+C_\varepsilon \left(\mathcal{E}_h^2+\tau^2\right)+\varepsilon\tau\sum_{n=0}^{m}\|e_\mu^{n+1}\|^2,
		\end{aligned}
	\end{equation}
	where we have noted $\|\nabla e_p^0\|^2\leq C\mathcal{E}_h^2$ and $\|e_{\tilde{\mathbf{u}}}^0\|=\|e_{\mathbf{u}}^0\|\leq C\mathcal{E}_h^2$.
	Finally, we obtain
	\begin{equation}
		\label{eq_error_inequation_eu1}
		\begin{aligned}
			\|e_\mathbf{u}^{m}\|^2+&2\nu\tau\sum_{n=0}^{m}\|\nabla e_{\tilde{\mathbf{u}}}^{n+1}\|^2+\tau^2\|\nabla e_p^{m}\|^2\\
			\leq&~C_\varepsilon\tau\sum_{n=0}^{m}\left(\|\nabla e_{\phi}^{n+1}\|^2+\|e_\mathbf{u}^{n}\|^2+\|e_{\tilde{\mathbf{u}}}^{n+1}\|^2\right)+C_\varepsilon \left(\mathcal{E}_h^2+\tau^2\right)+\varepsilon\tau\sum_{n=0}^{m}\|e_\mu^{n+1}\|^2.
		\end{aligned}
	\end{equation}
	\subsection{$L^2$ norm estimates for $e_\phi^{n+1}$ and $e_\mu^{n+1}$}
	According to \eqref{eq_error_estimate_phi_mu_H1_L2}, we need to obtain the optimal error estimates in $L^2$ norm. Therefore, we prove the optimal error estimates for $\|e_\phi^{n+1}\|$ and $\|e_\mu^{n+1}\|$ using the lemmas \ref{Lemma_operators_H10203} and \ref{Lemma_PW_inequalities_0203} in the subsection. In order to obtain the results, we present the following lemmas which are estimated for $|\left(e_\phi^{n+1},1\right)|^2$, $|\left(e_\mu^{n+1},1\right)|^2$ and $\|(-\Delta_h)^{-1/2}e_\phi^{n+1}\|^2$.
	\begin{Lemma}
		For the discrete mass conservation and estimates as follows:
		\begin{flalign}
			\label{eq_discrete_mass_conservation}
			&\left(\phi_h^{n+1},1\right)=\left(\phi_h^0,1\right),\\
			&\bigg|\left(e_\phi^{n+1},1\right)\bigg|\leq C\mathcal{E}_h,\\
			&\bigg|\left(e_\mu^{n+1},1\right)\bigg|\leq C\left(\|\nabla e_\phi^{n+1}\|+\|\nabla e_\phi^n\|+\mathcal{E}_h+\tau\right),
		\end{flalign} 
	where $C$ is a positive constant which is independent of $n,~h $ and $\tau$, for any $n=0,1,2,\cdots,N-1,$ and $h,\tau >0.$
	\end{Lemma}
	\begin{proof}
		Testing $w =1$ in \eqref{eq_variational_forms_phi} yields $\left(\partial_t\phi,1\right)=0$, we have 
		\begin{equation}
			\label{eq_discrete_mass_conservation_phi}
			\left(\phi^{n+1},1\right)=\left(\phi^0,1\right),
		\end{equation}
	which represents the mass conservation for the phase field model. We choose $q_h=1$ in \eqref{eq_fully_discrete_CHNS_scheme_tilde_incompressible_condition}, and get $\left(\nabla\cdot\mathbf{u}_h^{n+1},1\right)=\int_{\partial\Omega}\mathbf{u}_h^{n+1}\cdot\mathbf{n}=0$. Then 
	 $\forall q_h\in S_h$, the equation \eqref{eq_fully_discrete_CHNS_scheme_tilde_incompressible_condition} implies that $\left(\nabla\cdot\mathbf{u}_h^{n+1},q_h\right)=0$.\\
	Testing $w_h=1$ in \eqref{eq_fully_discrete_CHNS_scheme_phi}, we get
	\begin{equation}
		\left(\delta_\tau\phi_h^{n+1},1\right)=-\left(\nabla\phi_h^n\cdot\mathbf{u}_h^n,1\right)=\left(\phi_h^n,\nabla\cdot\mathbf{u}_h^n\right)=0,
	\end{equation}
 	which shows that $\phi_h^{n+1}$ satisfies the fully discrete mass conservation \eqref{eq_discrete_mass_conservation}.
 	By using \eqref{eq_discrete_mass_conservation} and \eqref{eq_discrete_mass_conservation_phi}, we can obtain
 	\begin{equation}
 		\begin{aligned}
 			\big|\left(e_\phi^{n+1},1\right)\big|=&~\big|\left(R_h\phi^{n+1}-\phi_h^{n+1},1\right)\big|=\big|\left(R_h\phi^{n+1}-\phi^{n+1}+\phi^0-\phi_h^0,1\right)\big|\\
 			\leq&~C\|R_h\phi^{n+1}-\phi^{n+1}\|_{H^{-1}}+C\|\phi^0-\phi_h^0\|_{H^{-1}}\leq C\mathcal{E}_h.
 		\end{aligned}
 	\end{equation}
 	Nextly, we choose $\varphi = 1$ and $\varphi_h=1 $  in \eqref{eq_variational_forms_mu} and \eqref{eq_fully_discrete_CHNS_scheme_mu}, respectively, to obtain
 	\begin{flalign}
 		\left(\mu^{n+1},1\right)=\left((\phi^{n+1})^3-\phi^{n+1},1\right),\\
 		\left(\mu_h^{n+1},1\right)=\left((\phi_h^{n})^3-\phi_h^{n},1\right).
 	\end{flalign}
 	Thus,
 	\begin{equation}
 		\begin{aligned}
 			\big|\left(e_\mu^{n+1},1\right)\big|=&~\big|\left(\Pi_h\mu^{n+1}-\mu^{n+1}+\mu^{n+1}-\mu_h^{n+1},1\right)\big|\\
 				=&~\big|\left(\Pi_h\mu^{n+1}-\mu^{n+1},1\right)+\left((\phi^{n+1})^3-(\phi_h^{n})^3-(\phi^n-\phi_h^n)-(\phi^{n+1}-\phi^n),1\right)\big|\\
 			\leq&~C\|\Pi_h\mu^{n+1}-\mu^{n+1}\|_{H^{-1}}+C\|\phi^{n+1}-R_h\phi^{n+1}\|_{H^{-1}}+C\|e_\phi^{n+1}\|\\
 			&+C\|\phi^n-R_h\phi^n\|_{H^{-1}}+C\|e_\phi^n\|+C\tau\\
 			\leq&~C\left(\tau+\mathcal{E}_h+\|e_\phi^{n+1}\|+\|e_\phi^n\|\right)\\
 			\leq& ~C\left(\|\nabla e_\phi^{n+1}\|+\|\nabla e_\phi^n\|+\mathcal{E}_h+\tau\right)(\text{using the inequality \eqref{eq_PW_inequalities_one}}),
 		\end{aligned}
 	\end{equation}
 	which leads to the desired result.
	\end{proof}
	\begin{Lemma}\label{Lemma_H_neg_1norm}
		According to the Lemma \ref{Lemma_PW_inequalities_0203}, there exists a positive constant $\tau_1$ such that when $\tau\leq \tau_1$, the estimate holds  as follows:
		\begin{equation}
			\label{eq_Lemma_H_neg_1}
			\|e_\phi^{n+1}\|_{H^{-1}}^2\leq \|(-\Delta_h)^{-\frac{1}{2}}e_\phi^{n+1}\|^2\leq C\left(\tau\sum_{m=0}^{n}\|e_\mathbf{u}^{m+1}\|^2+\mathcal{E}_h^2+\tau^2\right)
		\end{equation} 
	and
	\begin{equation}
		\label{eq_Lemma_H_neg_1_delta_tau}
		\|\delta_\tau e_\phi^{n+1}\|_{H^{-1}}^2\leq \|(-\Delta_h)^{-\frac{1}{2}}\delta_\tau e_\phi^{n+1}\|^2\leq C\left(\|\nabla e_\mu^{n+1}\|^2+\|\nabla e_\phi^{n+1}\|^2+\|\nabla e_\mathbf{u}^{n+1}\|^2+\mathcal{E}_h^2+\tau^2\right),
	\end{equation}
	where $C $ is positive constant independent of $\tau$, $h$ and $n~(n=0,1,2,\cdots,N-1) $.
	\end{Lemma}
	\begin{proof}
		We define by the operator $A_h^{-1}:S_h^r\rightarrow S_h^r$ \cite{2023_CaiWentao_Optimal_L2_error_estimates_of_unconditionally_stable_finite_element_schemes_for_the_Cahn_Hilliard_Navier_Stokes_system} with
		\begin{equation}
			A_h^{-1}v_h=\left(-\Delta_h\right)^{-1}\left(v_h-\frac{1}{|\Omega|}\left(v_h,1\right)\right) ,
		\end{equation}
	which satisfies $A_h^{-1}v_h=(-\Delta_h)^{-1}v_h$ for $v_h\in \mathring{S}_h^r$. Testing $w_h=A_h^{-1}e_\phi^{n+1}$ in \eqref{eq_fully_discrete_CHNS_scheme_phi} and $\varphi_h=\tilde{e}_\phi^{n+1}:=e_\phi^{n+1}-\frac{1}{|\Omega|}(e_\phi^{n+1},1)$ in \eqref{eq_fully_discrete_CHNS_scheme_mu},
	and combining the resulting equations, we can obtain
	\begin{equation}
		\label{eq_conbining_result_equation}
		\begin{aligned}
			&\frac{1}{2}\delta_\tau\|A_h^{-\frac{1}{2}}e_\phi^{n+1}\|^2+\frac{1}{2\tau}\|A_h^{-\frac{1}{2}}\left(e_\phi^{n+1}-e_\phi^n\right)\|^2+M\|\nabla e_\phi^{n+1}\|^2\\
			=&\left(A_h^{-\frac{1}{2}}\delta_\tau\left(R_h\phi^{n+1}-\phi^{n+1}\right),A_h^{-\frac{1}{2}}e_\phi^{n+1}\right)
				+b\left(\phi_h^{n+1},\mathbf{u}_h^n,A_h^{-1}e_\phi^{n+1}\right)-b\left(R_h\phi^{n+1},\mathbf{u}^n,A_h^{-1}e_\phi^{n+1}\right)\\
			&+\left(A_h^{-\frac{1}{2}}R_1^{n+1},A_h^{-\frac{1}{2}}e_\phi^{n+1}\right)+\frac{\lambda}{\epsilon^2}\left(e_\phi^n,\tilde{e}_\phi^{n+1}\right)+\frac{\lambda}{\epsilon^2}\left((\phi_h^{n})^3-(\phi^{n})^3,\tilde{e}_\phi^{n+1}\right)\\
			&+\frac{\lambda}{\epsilon^2}\left(\phi^n-R_h\phi^n,\tilde{e}_\phi^{n+1}\right)
			+\left(\mu^{n+1}-\Pi_h\mu^{n+1},\tilde{e}_\phi^{n+1}\right)\\
			\leq&~\|A_h^{-\frac{1}{2}}e_\phi^{n+1}\|^2+C_\varepsilon\left(\|e_\phi^{n+1}\|^2+\|e_\phi^{n}\|^2+\mathcal{E}_h^2+\tau^2\right)+\varepsilon\|\tilde{e}_\phi^{n+1}\|_{H^1}^2\\
			&+b\left(\phi_h^{n+1},\mathbf{u}_h^n,A_h^{-1}e_\phi^{n+1}\right)-b\left(R_h\phi^{n+1},\mathbf{u}^n,A_h^{-1}e_\phi^{n+1}\right),
		\end{aligned}
	\end{equation}
	where we used \eqref{eq_psi_Rhpsi_Ls_norm_inequation} and $\left(\nabla e_\mu^{n+1},\nabla A_h^{-1}e_\phi^{n+1}\right)=\left(e_\mu^{n+1},\tilde{e}_\phi^{n+1}\right)$.
	From the defintion \eqref{eq_trilinear_forms} of $b(\cdot,\cdot,\cdot)$, we know that
	\begin{equation}
		\begin{aligned}
			&b\left(\phi_h^{n+1},\mathbf{u}_h^n,A_h^{-1}e_\phi^{n+1}\right)-b\left(R_h\phi^{n+1},\mathbf{u}^n,A_h^{-1}e_\phi^{n+1}\right)\\
			=&\left(\nabla\phi_h^{n+1}\cdot\mathbf{u}_h^n,A_h^{-1}e_\phi^{n+1}\right)-\left(\nabla R_h\phi^{n+1}\cdot\mathbf{u}^n,A_h^{-1}e_\phi^{n+1}\right)\\
			=&-\left(\nabla\phi_h^{n+1}\cdot e_\mathbf{u}^{n},A_h^{-1}e_\phi^{n+1}\right)-\left(\nabla\phi_h^{n+1}\cdot\left(\mathbf{u}^n-\mathbf{P}_h\mathbf{u}^n\right),A_h^{-1}e_\phi^{n+1}\right)-\left(\nabla e_\phi^{n+1}\cdot\mathbf{u}^n,A_h^{-1}e_\phi^{n+1}\right)\\
			=&:\sum_{i=1}^{3}K_i.
		\end{aligned}
	\end{equation}
	Then using the inequalities \eqref{eq_operators_estimates_neg_one}, we can obtain the bound of $K_1$ as follows,
	\begin{equation}
		\begin{aligned}
			\big|K_1\big|&=\big|-\left(A_h^{-\frac{1}{2}}\left(\nabla \phi_h^{n+1}\cdot e_\mathbf{u}^n\right),A_h^{-\frac{1}{2}}e_\phi^{n+1}\right)\big|
			\leq C\|\nabla\phi_h^{n+1}\cdot e_\mathbf{u}^n\|_{H^{-1}}\|A_h^{-\frac{1}{2}}e_\phi^{n+1}\|\\
			&\leq\|\nabla\phi_h^{n+1}\cdot e_\mathbf{u}^n\|_{L^{\frac{6}{5}}}\|A_h^{-\frac{1}{2}}e_\phi^{n+1}\|
			\leq C\left(\|e_\mathbf{u}^n\|^2+\|A_h^{-\frac{1}{2}}e_\phi^{n+1}\|^2\right).
		\end{aligned}
	\end{equation}
	and
	\begin{equation}
		\begin{aligned}
			\big|K_2\big|=&\big|-\left(A_h^{-\frac{1}{2}}\left(\nabla\phi_h^{n+1}\cdot\left(\mathbf{u}^n-\mathbf{P}_h\mathbf{u}^n\right)\right),A_h^{-\frac{1}{2}}e_{\phi}^{n+1}\right)\big|
			\leq C\left(\|\mathbf{u}^n-\mathbf{P}_h\mathbf{u}^n\|^2+\|A_h^{-\frac{1}{2}}e_{\phi}^{n+1}\|^2\right)\\
			\leq &~C\left(\mathcal{E}_h^2+\|A_h^{-\frac{1}{2}}e_{\phi}^{n+1}\|^2\right).
		\end{aligned}
	\end{equation}
	And using the integration by parts and $\nabla\cdot\mathbf{u}^n=0$, we have
	\begin{equation}
		\begin{aligned}
			\big|K_3\big|=&\big|\left({e}_\phi^{n+1},\mathbf{u}^n\cdot\nabla A_h^{-1}e_{\phi}^{n+1}\right)\big|
			\leq C\|{e}_\phi^{n+1}\|\|\mathbf{u}^n\|_{L^{\infty}}\|A_h^{-\frac{1}{2}}e_{\phi}^{n+1}\|\\
			\leq & \left(\varepsilon\|\nabla e_\phi^{n+1}\|+C_\varepsilon\|A_h^{-\frac{1}{2}}e_{\phi}^{n+1}\|\right)\|A_h^{-\frac{1}{2}}e_{\phi}^{n+1}\|\\
			\leq &~\varepsilon\|\nabla e_\phi^{n+1}\|^2+C_\varepsilon\left(\|A_h^{-\frac{1}{2}}e_{\phi}^{n+1}\|^2+\|A_h^{-\frac{1}{2}}e_{\phi}^{n+1}\|^2\right)
		\end{aligned}
	\end{equation}
	Combining the above equations $K_1$, $K_2$ and $K_3$, we get
	\begin{equation}
		\label{eq_combining_Ki_results}
		\sum_{i=1}^{3}K_i\leq C\left(\|e_\mathbf{u}^n\|^2+\|A_h^{-\frac{1}{2}}e_{\phi}^{n+1}\|^2+\mathcal{E}_h^2+\tau^2\right)+\varepsilon\|\nabla e_\phi^{n+1}\|^2.
	\end{equation}
	Thus, combining the \eqref{eq_combining_Ki_results} with \eqref{eq_conbining_result_equation} and then using the equations \eqref{eq_operators_estimates_neg_3} and \eqref{eq_operators_estimates_mean_value_inequation}, we can obtain 
	\begin{equation}
		\label{eq_error_estimate_L2_phi_mu}
		\begin{aligned}
			&\frac{1}{2}\delta_\tau\|A_h^{-\frac{1}{2}}e_\phi^{n+1}\|^2+M\|\nabla e_\phi^{n+1}\|^2\\
			\leq&~C_\varepsilon\left(\|A_h^{-\frac{1}{2}}e_\phi^{n+1}\|^2+\|A_h^{-\frac{1}{2}}e_\phi^{n}\|^2+\|e_\phi^{n+1}\|^2+\|e_\phi^{n}\|^2+\|e_\mathbf{u}^{n+1}\|^2+\mathcal{E}_h^2+\tau^2\right).
		\end{aligned}
	\end{equation}
	Summing up the estimate \eqref{eq_error_estimate_L2_phi_mu} yield
	\begin{equation}
		\|A_h^{-\frac{1}{2}}e_\phi^N\|^2+\tau\sum_{n=0}^{N-1}\|
		\nabla e_\phi^{n+1}\|^2\leq C\tau\sum_{n=0}^{N-1}\left(\|A_h^{-\frac{1}{2}}e_\phi^{n+1}\|^2+\|e_\mathbf{u}^{n+1}\|^2\right)+C(\mathcal{E}_h^2+\tau^2).
	\end{equation}
	Using Gr\"{o}nwall lemma \ref{lemma_discrete_Gronwall_inequation}, there exists a positive constant $\tau_1$ such that \eqref{eq_Lemma_H_neg_1} holds when $\tau\leq\tau_1$. 
	\\For the \eqref{eq_Lemma_H_neg_1_delta_tau}, we choose $w_h=A_h^{-1}\delta_\tau e_\phi^{n+1}$ in \eqref{eq_error_equations_phi}  to get
	\begin{equation}
		\label{eq_delta_tau_A_h_e_phi_n1}
		\begin{aligned}
			&\left(\delta_\tau e_\phi^{n+1},A_h^{-1}\delta_\tau e_\phi^{n+1}\right)+M\left(\nabla e_\mu^{n+1},\nabla A_h^{-1}\delta_\tau e_\phi^{n+1}\right)\\=
			&\left(\delta_\tau(R_h\phi^{n+1}-\phi^{n+1}),A_h^{-1}\delta_\tau e_\phi^{n+1}\right)
			+\left(R_1^{n+1},A_h^{-1}\delta_\tau e_\phi^{n+1}\right)\\
			&+b\left(\phi_h^{n+1},\mathbf{u}_h^n,A_h^{-1}\delta_\tau e_\phi^{n+1}\right)-b\left(R_h\phi^{n+1},\mathbf{u}^n,A_h^{-1}\delta_\tau e_\phi^{n+1}\right)\\
			\leq &~C\|\nabla e_\mu^{n+1}\|^2+\frac{1}{10}\|\delta_\tau e_\phi^{n+1}\|_{H^{-1}}+\|\delta_\tau(R_h\phi^{n+1}-\phi^{n+1})\|^2+\frac{1}{10}\|\delta_\tau e_\phi^{n+1}\|_{H^{-1}}+\|\nabla e_\mathbf{u}^n\|^2\\
			&+b\left(\phi_h^{n+1},\mathbf{u}_h^n,A_h^{-1}\delta_\tau e_\phi^{n+1}\right)-b\left(R_h\phi^{n+1},\mathbf{u}^n,A_h^{-1}\delta_\tau e_\phi^{n+1}\right).
		\end{aligned}
	\end{equation}
	Similarly,  we can obtain the error estimates of the last two terms of \eqref{eq_delta_tau_A_h_e_phi_n1} from the defintion \eqref{eq_trilinear_forms} that
	\begin{equation}
		\begin{aligned}
			&b\left(\phi_h^{n+1},\mathbf{u}_h^n,A_h^{-1}\delta_\tau e_\phi^{n+1}\right)-b\left(R_h\phi^{n+1},\mathbf{u}^n,A_h^{-1}\delta_\tau e_\phi^{n+1}\right)\\
			=&\left(\nabla\phi_h^{n+1}\cdot\mathbf{u}_h^n,A_h^{-1}\delta_\tau e_\phi^{n+1}\right)-\left(\nabla R_h\phi^{n
			+1}\cdot\mathbf{u}^n,A_h^{-1}\delta_\tau e_\phi^{n+1}\right)\\
			=&-\left(\nabla e_\phi^{n+1}\cdot \mathbf{u}_h^n,A_h^{-1}\delta_\tau e_\phi^{n+1}\right)
				-\left( \nabla R_h\phi^{n+1}\cdot e_\mathbf{u}^n,A_h^{-1}\delta_\tau e_\phi^{n+1}\right)\\
				&-\left(\nabla R_h\phi^{n+1}\cdot\left(\mathbf{u}^n-\mathbf{P}_h\mathbf{u}^n\right),A_h^{-1}\delta_\tau e_\phi^{n+1}\right)\\
			\leq&~\|\nabla e_\phi^{n+1}\|\|\mathbf{u}_h^n\|\|A_h^{-1}\delta_\tau e_\phi^{n+1}\|_{L^4}+\|\nabla R_h\phi^{n+1}\|\|e_\mathbf{u}^n\|_{L^4}\|A_h^{-1}\delta_\tau e_\phi^{n+1}\|_{L^4}\\&+\|\nabla R_h\phi^{n+1}\|\|\mathbf{u}^n-\mathbf{P}_h\mathbf{u}^n\|_{L^4}\|A_h^{-1}\delta_\tau e_\phi^{n+1}\|_{L^4}\\
			\leq&~ C\|\nabla e_\phi^{n+1}\|^2+\frac{1}{10}\|\delta_\tau e_\phi^{n+1}\|_{H^{-1}}+C\tau^2+\frac{1}{10}\|\delta_\tau e_\phi^{n+1}\|_{H^{-1}}+C\mathcal{E}_h^2+\frac{1}{10}\|\delta_\tau e_\phi^{n+1}\|_{H^{-1}}.
		\end{aligned}
	\end{equation}
	For \eqref{eq_delta_tau_A_h_e_phi_n1}, we have
	\begin{equation}
		\begin{aligned}
			&\left(\delta_\tau e_\phi^{n+1},A_h^{-1}\delta_\tau e_\phi^{n+1}\right)=\|A_h^{-\frac{1}{2}}\delta_\tau e_\phi^{n+1}\|_{H^{-1}}^2\\
			\leq&~ C\|\nabla e_\mu^{n+1}\|^2+\frac{1}{10}\|\delta_\tau e_\phi^{n+1}\|_{H^{-1}}+\|\delta_\tau(R_h\phi^{n+1}-\phi^{n+1})\|^2+\frac{1}{10}\|\delta_\tau e_\phi^{n+1}\|_{H^{-1}}+\|\nabla e_\mathbf{u}^n\|^2\\
			&+C\|\nabla e_\phi^n\|^2+\frac{1}{10}\|\delta_\tau e_\phi^{n+1}\|_{H^{-1}}+C\tau^2+\frac{1}{10}\|\delta_\tau e_\phi^{n+1}\|_{H^{-1}}+C\mathcal{E}_h^2+\frac{1}{10}\|\delta_\tau e_\phi^{n+1}\|_{H^{-1}}\\
			\leq&~C\left(\|\nabla e_\mu^{n+1}\|^2+\|\nabla e_\phi^{n+1}\|^2+\|\nabla e_\mathbf{u}^n\|^2+\mathcal{E}_h^2+\tau^2\right)+\frac{1}{2}\|\delta_\tau e_\phi^{n+1}\|_{H^{-1}}.
		\end{aligned}
	\end{equation}
	 Thus, for the above error estimates, we have
	 \begin{equation}
	 	\|\delta_\tau e_\phi^{n+1}\|_{H^{-1}}^2\leq \|(-\Delta_h)^{-\frac{1}{2}}\delta_\tau e_\phi^{n+1}\|^2\leq C\left(\|\nabla e_\mu^{n+1}\|^2+\|\nabla e_\phi^{n+1}\|^2+\|\nabla e_\mathbf{u}^{n+1}\|^2+\mathcal{E}_h^2+\tau^2\right).
	 \end{equation}
	 The proof of Lemma \ref{Lemma_H_neg_1norm} result has been completed.
	\end{proof}
	Next, we also need to estimate the optimal $L^2$ error of $e_\rho^{n+1}$ and $e_p^{n+1}$.
	\begin{Lemma}\label{Lemma_boundness_e_rho}
		Under the assumption of \eqref{eq_varibles_satisfied_regularities}, for all $m\leq0$ and $e_\rho^0=0$, we have 
		\begin{equation}
			|e_\rho^{m+1}|^2\leq C\tau\sum_{n=0}^{m}\left(\|e_\mu^{n+1}\|_{H^1}^2+\|e_{\tilde{\mathbf{u}}}^{n+1}\|^2+\|\nabla e_\phi^{n+1}\|^2\right)+C(\tau^2+h^{2(r+1)}),
		\end{equation}
	where C is positive constant independent of $\tau$.
	\end{Lemma}
	\begin{proof}
		Multiplying both side of \eqref{eq_error_equations_additional_terms} by $4\tau e_\rho^{n+1}$ and using the Lemma \ref{Lemma_0401} yield
		\begin{equation}
			\label{eq_error_multiplying_4tau_e_rho}
			\begin{aligned}
				|e_\rho^{n+1}|^2-~&|e_\rho^n|^2+|e_\rho^{n+1}-e_\rho^n|^2=2\tau e_\rho^{n+1}
					\left(\left(\frac{F'(\phi^n)}{\rho^{n+1}}-\frac{F'(\phi_h^n)}{\rho_h^{n+1}},\delta_\tau\phi^{n+1}\right)
					+\left(\frac{F'(\phi_h^n)}{\rho_h^{n+1}},\delta_\tau e_\phi^{n+1}\right)\right)\\
				&+2\tau e_\rho^{n+1}\left(\frac{1}{\lambda \sqrt{E_1^{n+1}}}\left(\mathbf{u}^{n}\cdot\nabla\mathbf{u}^{n},\mathbf{u}^{n}\right)-
					\frac{\rho_h^{n}}{\lambda\rho_h^{n+1} \sqrt{E_{1,h}^{n+1}}}
					\left(\mathbf{u}_h^n\cdot\nabla\mathbf{u}_h^{n},\tilde{\mathbf{u}}_h^{n+1}\right)\right)\\
				&+\left(\frac{2\tau e_\rho^{n+1}}{\lambda}\left(\left(\mathbf{u}^n\cdot\nabla\phi^{n+1},\mu^{n+1}\right)\right)
					-\left(\mu^{n+1}\nabla\phi^{n+1},\mathbf{u}^{n+1}\right)\right)\\
				&-\left(\frac{2\tau e_\rho^{n+1}}{\lambda}\left(\left(\mathbf{u}_h^n\cdot\nabla\phi_h^{n+1},\mu_h^{n+1}\right)\right)
					-\left(\mu_h^{n+1}\nabla\phi_h^{n+1},\tilde{\mathbf{u}}_h^{n+1}\right)\right)\\
				&-\frac{2\tau e_\rho^{n+1}}{\rho^{n+1}}\left(F'(\phi^n),R_1^{n+1}\right)+4\tau e_\rho^{n+1}R_4^{n+1}\\
				:=~&\sum_{i=4}^{8}K_i.
			\end{aligned}
		\end{equation}
		Now we need to estimate the error of the right-hand side terms of \eqref{eq_error_multiplying_4tau_e_rho}.
		From the definitions of $\rho(\phi^{n+1})$, $\rho_h^{n+1}$, \eqref{eq_sav_inequalities} and the lemma \ref{Lemma_H_neg_1norm},
		we can derive
		\begin{equation}
			\begin{aligned}
				\bigg|\frac{1}{\rho^{n+1}}-\frac{1}{\rho_h^{n+1}}\bigg|
					=\bigg|\frac{E_{1,h}^{n+1}-E_1^{n+1}}{\rho_h^{n+1}\rho^{n+1}\left(\rho_h^{n+1}+\rho^{n+1}\right)}\bigg|
					\leq C\left(\|e_\phi^{n+1}\|+\tau\right).
			\end{aligned}
		\end{equation}
	Due to the fact that $\rho_h^{n+1}=\sqrt{E_{1,h}^{n+1}}>\sqrt{C_0}$ and $F'(\phi_h^n)$ has an uniform upper bound, it follows that
	\begin{equation}
		\begin{aligned}
			\big|K_4\big|=~&\bigg|2\tau e_\rho^{n+1}\left(\frac{F'(\phi^n)}{\rho^{n+1}}-\frac{F'(\phi_h^n)}{\rho_h^{n+1}},\delta_\tau\phi_h^{n+1}\right)\bigg|\\
			\leq ~&2\tau|e_\rho^{n+1}|\bigg|\left(\frac{1}{\rho^{n+1}}-\frac{1}{\rho_h^{n+1}}\right)\left(F'(\phi^{n+1}),\delta_\tau\phi^{n+1}\right)\\
				&+\frac{1}{\rho_h^{n+1}}\left(F'(\phi^n)-F'(\phi_h^n),\delta_\tau\phi^{n+1}\right)
					\bigg|\\
			\leq~&C\tau\left(|e_\rho^{n+1}|^2+\|e_\phi^{n+1}\|^2+\|e_\phi^n\|^2+\tau^2+h^{2(r+1)}\right),
		\end{aligned}
	\end{equation}
	and
	\begin{equation}
		\begin{aligned}
			\big|K_5\big|=~&\bigg|2\tau e_\rho^{n+1}\left(\frac{F'(\phi_h^n)}{\rho_h^{n+1}},\delta_\tau e_\phi^{n+1}\right)\bigg|
				\leq C\tau|e_\rho^{n+1}|\|\nabla F'(\phi_h^n)\|\|\delta_\tau e_{\phi}^{n+1}\|_{H^{-1}}\\
				\leq~&C\tau|e_\rho^{n+1}|\|\nabla\left((\phi_h^n)^3-\phi_h^n\right)\|\|\delta_\tau e_{\phi}^{n+1}\|_{H^{-1}}\\
				\leq~&C\tau\left(\|\nabla e_{\mu}^{n+1}\|^2+\|\nabla e_\phi^{n+1}\|^2
							+\|\nabla e_\mathbf{u}^{n+1}\|^2+\|e_\mathbf{u}^{n}\|^2+\|e_\phi^n\|_{H^1}^2
							+|\rho^{n+1}|^2+|e_\rho^n|^2+\tau^2+h^{2(r+1)}\right).
		\end{aligned}
	\end{equation}
	For the term $K_6$, we have the following  error estimates
	\begin{equation}
		\begin{aligned}
			\big|K_6\big|=~&\bigg|2\tau e_\rho^{n+1}\left(\frac{1}{\lambda \sqrt{E_{1}(\phi^{n+1})}}
					\left(\mathbf{u}^{n}\cdot\nabla\mathbf{u}^{n},\mathbf{u}^{n}\right)-
					\frac{\rho_h^{n}}{\lambda\rho_h^{n+1} \sqrt{E_{1,h}^{n+1}}}
					\left(\mathbf{u}_h^n\cdot\nabla\mathbf{u}_h^{n},\tilde{\mathbf{u}}_h^{n+1}\right)\right)\bigg|\\
			\leq~&\frac{2\tau|e_\rho^{n+1}|}{\lambda\sqrt{E_1^{n+1}}}\bigg|
					\left(\left(\mathbf{u}^{n}\cdot\nabla\mathbf{u}^{n},\mathbf{u}^{n}\right)
						-\frac{\rho_h^n\sqrt{E_1^{n+1}}}{\rho_h^{n+1}\sqrt{E_{1,h}^{n+1}}}
						\left(\mathbf{u}_h^n\cdot\mathbf{u}_h^n,\tilde{\mathbf{u}}_h^{n+1}\right)\right)
					\bigg|\\
			\leq ~&\frac{2\tau|e_\rho^{n+1}|}{\sqrt{E_1^{n+1}}}\bigg|
						\left(\left(\mathbf{u}^n-\mathbf{P}_h\mathbf{u}^n\right)\cdot\nabla\mathbf{u}^n,\mathbf{u}^n\right)
						+\left(e_\mathbf{u}^n\cdot\nabla\mathbf{u}^n,\mathbf{u}^n\right)\\
						&\qquad\qquad+\left(\mathbf{u}_h^{n}\cdot\nabla\left(\mathbf{u}^n-\mathbf{P}_h\mathbf{u}^n\right),\mathbf{u}^n\right)
						+\left(\mathbf{u}_h^n\cdot\nabla e_{\mathbf{u}}^n,\mathbf{u}^n\right)\\
						&\qquad\qquad+\left(\mathbf{u}_h^n\cdot\mathbf{u}_h^n,\mathbf{u}^n-\mathbf{u}^{n+1}\right)
						+\left(\mathbf{u}_h^n\cdot\nabla\mathbf{u}_h^n,e_{\tilde{\mathbf{u}}}^{n+1}\right)\\
						&+\frac{2\tau|e_\rho^{n+2}|}{\rho_h^{n+1}}\left(\frac{\rho_h^{n+1}}{\sqrt{E_1^{n+1}}}-\frac{\rho_h^n}{\sqrt{E_{1,h}^{n+1}}}\right)
						\left(\mathbf{u}_h^n\cdot\nabla\mathbf{u}_h^n,\tilde{\mathbf{u}}_h^{n+1}\right)
					\bigg|.
		\end{aligned}
	\end{equation}
	By \eqref{eq_sav_inequalities}, we can derive
	\begin{equation}
		\begin{aligned}
			\bigg|\frac{\rho_h^{n+1}}{\sqrt{E_1^{n+1}}}-\frac{\rho_h^n}{\sqrt{E_{1,h}^{n+1}}}\bigg|=~&
			\bigg|\frac{\rho_h^{n+1}}{\sqrt{E_1^{n+1}}}-\frac{\rho_h^{n}}{\sqrt{E_1^{n+1}}}
					+\frac{\rho_h^{n}}{\sqrt{E_1^{n+1}}}-\frac{\rho_h^n}{\sqrt{E_{1,h}^{n+1}}}\bigg|\\
			\leq~&\bigg|\frac{1}{\sqrt{E_1^{n+1}}}\left(\rho_h^{n+1}-\rho^{n+1}+\rho^{n+1}-\rho^n+\rho^n-\rho_h^{n}\right)\\
					&+\rho_h^n\left(\frac{1}{\sqrt{E_1^{n+1}}}-\frac{1}{\sqrt{E_{1,h}^{n+1}}}\right)	
						\bigg|\\
			\leq~&C\left(|e_\rho^{n+1}|+|e_\rho^n|+\|e_\phi^{n+1}\|+\tau\right).
		\end{aligned}
	\end{equation}
	It can easily be shown that
	\begin{equation}
		\begin{aligned}
			|K_6|\leq~&\frac{C\tau|e_\rho^{n+1}|}{\sqrt{C_0}}\left(\|e_\mathbf{u}^n\|\|\mathbf{u}^n
				-\mathbf{P}_h\mathbf{u}^n\|\|\mathbf{u}^n\|_{L^\infty}\right.\\
				&\left.+\|\mathbf{u}_h^n\|\|\nabla\mathbf{u}_h^n\|\left(\tau+\|\nabla e_{\tilde{\mathbf{u}}}^{n+1}\|
				+|e_\rho^{n+1}|+|e_\rho^n|+\|e_\phi^{n+1}\|\right)
			\right)\\
			\leq~&C\tau\left(\|\nabla e_{\tilde{\mathbf{u}}}^{n+1}\|^2+\tau\|e_\mathbf{u}^n\|^2
					+\tau\|e_\phi^n\|^2+\|e_\phi^{n+1}\|^2+|e_\rho^n|^2+|e_\rho^{n+1}|^2
					+h^{2(r+1)}+\tau^2\right).
		\end{aligned}
	\end{equation}
	The term $K_7$ can be estimated by
	\begin{equation}
		\label{eq_inequalities_K7}
		\begin{aligned}
			|K_7|=~&\bigg|\left(\frac{2\tau e_\rho^{n+1}}{\lambda}\left(\left(\mathbf{u}^n\cdot\nabla\phi^{n+1},\mu^{n+1}\right)\right)
			-\left(\mu^{n+1}\nabla\phi^{n+1},\mathbf{u}^{n+1}\right)\right)\\
			&-\left(\frac{2\tau e_\rho^{n+1}}{\lambda}\left(\left(\mathbf{u}_h^n\cdot\nabla\phi_h^{n+1},\mu_h^{n+1}\right)\right)
			-\left(\mu_h^{n+1}\nabla\phi_h^{n+1},\tilde{\mathbf{u}}_h^{n+1}\right)\right)\bigg|\\
			\leq~&C\tau|e_\rho^{n+1}|\left(\phi_h^{n+1},\left(\tilde{\mathbf{u}}_h^{n+1}-\mathbf{u}_h^n\right)\cdot\nabla\mu_h^{n+1}\right)\\
			\leq~&C\tau|e_\rho^{n+1}|\|\phi_h^{n+1}\|\|\tilde{\mathbf{u}}_h^{n+1}-\mathbf{u}_h^n\|_{L^4}\|\nabla\mu_h^{n+1}\|_{L^4}\\
			\leq~&C\tau\left(\|e_\mathbf{u}^{n+1}\|^2+\|e_{\mathbf{u}}^n\|^2
					+\|e_\phi^{n+1}\|^2+\|e_\phi^{n}\|^2+\|\nabla e_\mu^{n+1}\|^2+|e_\rho^{n+1}|^2+h^{2(r+1)}+\tau\right).
		\end{aligned}
	\end{equation}
	For the term $K_8$, using the Young inequality and the Cauchy-Schwarz inequality, we can obtain
	\begin{equation}
		\label{eq_inquality_K8}
		\begin{aligned}
			|K_8|=\bigg|-\frac{2\tau e_\rho^{n+1}}{\rho^{n+1}}\left(F'(\phi^n),R_1^{n+1}\right)+4\tau e_\rho^{n+1}R_4^{n+1}\bigg|
				\leq C\tau|e_\rho^{n+1}|^2+C\tau^3.
		\end{aligned}
	\end{equation}
	Summing from $n=0$ to $m$, we have
	\begin{equation}
		\label{eq_the_above_inequalities_summing_up}
		\begin{aligned}
			|e_\rho^{m+1}|^2\leq &~|e_\rho^0|^2+C(\tau^2+h^{2(r+1)})\\
			&+C\tau\sum_{n=0}^{m}\left(\|e_\phi^{n+1}\|^2
			+\|e_\phi^{n}\|^2+\|\nabla e_\phi^{n+1}\|^2+\|\nabla e_\mu^{n+1}\|^2
				+\|e_\mathbf{u}^{n+1}\|^2+\|e_\mathbf{u}^{n}\|^2\right).
		\end{aligned}
	\end{equation}
	\end{proof}
	Using the above estimates and discrete Gr\"onwall's inequality in Lemma \ref{lemma_discrete_Gronwall_inequation}, there exitsts a positive constant $\tau_2$ such that, when $\tau\leq \tau_2$, 
	\begin{equation}
		\|\nabla e_\phi^{m+1}\|^2+\|e_\mathbf{u}^{m+1}\|^2+|e_\rho^{m+1}|^2+\tau\sum_{n=0}^{m}\left(\|\nabla e_\mu^{n+1}\|^2+\|\nabla e_\mathbf{u}^{n+1}\|^2\right)\leq C\left(\mathcal{E}_h^2+\tau^2\right),
	\end{equation}
	where $C$ is positive constant independent of $\tau$ and $h$.
	\subsection{Estimates for $e_p^{n+1}$}
	We analyze the pressure error estimates through  the inf-sup condition
	\begin{equation}
		\|p\|\leq \sup_{\mathbf{v} \in \mathbf{H}_0^1(\Omega)}
		\frac{\left(p,\nabla\cdot \mathbf{v}\right)}{\|\nabla\mathbf{v}\|}=\sup_{\mathbf{v} \in \mathbf{H}_0^1(\Omega)}	\frac{-\left(\nabla p, \mathbf{v}\right)}{\|\nabla\mathbf{v}\|}.
	\end{equation}
	Therefore, we need to estimate the term $\left(\nabla e_p^{n+1},\mathbf{v}_h\right)$ in the discrete case as the following theorem.
	\begin{Theorem}
		Assuming that the solution to the CHNS model \eqref{eq_chns_equations} satisfies the regularities \eqref{eq_varibles_satisfied_regularities},
	then for the fully disctere scheme \eqref{eq_fully_discrete_CHNS_scheme_phi}-\eqref{eq_fully_discrete_CHNS_scheme_tilde_u_add_term}, we can obtain
	\begin{equation}
		\tau\sum_{n=0}^{m}\|e_p^{n+1}\|^2\leq C(\tau^2+h^{2(r+1)}),\quad \forall 0\leq m\leq N-1,
	\end{equation}
	where $C$ is positive constant independent of $\tau$ and $h$.
	\end{Theorem}
	\begin{proof}
		We first present the error estimate of the term $|\delta_\tau e_\rho^{n+1}|$. 
		Multiplying both side of \eqref{eq_error_equations_additional_terms} by $\delta_\tau e_\rho^{n+1}$ using the Lemma \ref{Lemma_0401} 
		and \ref{Lemma_boundness_e_rho} yields
		\begin{equation}
			\label{eq_delta_t_r}
			\begin{aligned}
				\big|\delta_\tau e_\rho^{n+1}\big|^2 =&\bigg|\delta_\tau e_\rho^{n+1}
				\left(\frac{1}{2\rho^{n+1}}\left(F'(\phi^n),\delta_\tau\phi^{n+1}\right)
				-\frac{1}{2\rho_h^{n+1}}\left(F'(\phi_h^n),\delta_\tau\phi_h^{n+1}\right)\right)\\
				&+\delta_\tau e_\rho^{n+1}\left(\frac{1}{2\lambda\sqrt{E_1^{n+1}}}
					\left(\mathbf{u}^n\cdot\nabla\mathbf{u}^n,\mathbf{u}^n\right)
					-\frac{\rho_h^n}{2\lambda\rho_h^{n+1}\sqrt{E_{1,h}^{n+1}}}
					\left(\mathbf{u}_h^n\cdot\nabla\mathbf{u}_h^n,\tilde{\mathbf{u}}_h^{n+1}\right)\right)\\
				&+\frac{\delta_\tau e_\rho^{n+1}}{\lambda}\left(\left(\mathbf{u}_h^n\cdot\nabla\phi_h^{n+1},\mu_h^{n+1}\right)
					-\left(\mu_h^{n+1}\cdot\nabla\phi_h^{n+1},\tilde{\mathbf{u}}_h^{n+1}\right)\right)\\
				&+\frac{\delta_\tau e_\rho^{n+1}}{2\rho^{n+1}}\left(R_4^{n+1}-\left(F'(\phi^n),R_1^{n+1}\right)\right)
				\\
				\leq ~&\frac{1}{2}\bigg|\delta_\tau e_\rho^{n+1}\bigg|^2+C\left(\|e_\phi^{n+1}\|^2+\|e_\phi^n\|^2+\|\nabla e_\phi^{n+1}\|^2+\|\nabla e_\phi^n\|^2+\|\nabla e_\mu^{n+1}\|^2\right.\\
				&\left.+\|e_\mathbf{u}^{n+1}\|^2+\|e_\mathbf{u}^{n}\|^2\right)+C\tau^2.
			\end{aligned}
		\end{equation}
	Multiplying \eqref{eq_delta_t_r} by $2\tau$ and summing up for $n$ from $0$ to $m$, and recalling the Theorem \ref{theorem_error_estimates_u_phi_mu_r}, we can obtain
	\begin{equation}
		\label{eq_error_estimates_r_delta_tau}
		\begin{aligned}
			\tau \sum_{n=0}^{m}|\delta_\tau e_\rho^{n+1}|^2\leq& C\tau\sum_{n=0}^{m}\left(\|e_\phi^{n+1}\|^2+\|e_\phi^n\|^2+\|\nabla e_\phi^{n+1}\|^2+\|\nabla e_\phi^n\|^2+\|\nabla e_\mu^{n+1}\|^2\right)\\&
			+C\tau\sum_{n=0}^{m}\left(\|e_\mathbf{u}^{n+1}\|^2+\|e_\mathbf{u}^{n}\|^2\right)+C\tau^2\\
			\leq& C\tau^2.
		\end{aligned}
	\end{equation}
	Next we establish the error estimates of the term $\|\delta_\tau e_\phi^{n+1}\|$ and $\|\delta_\tau e_\mu^{n+1}\|$ and denote the
	operator 
	$\delta_{\tau\tau}g^{n+1}:=\frac{\delta_\tau g^{n+1}-\delta_\tau g^n}{\tau}$. Taking the $\delta_{\tau}$ of two consecutive steps in \eqref{eq_error_equations_phi} and \eqref{eq_error_equations_mu}, we get
	\begin{equation}
		\label{eq_error_equations_delta_dtau_phi}
		\begin{aligned}
			\left(\delta_{\tau\tau} e_\phi^{n+1},w_h\right)+M\left(\nabla\delta_\tau e_\mu^{n+1},\nabla w_h\right)+\delta_{\tau}b\left(R_h\phi^{n+1},\mathbf{u}^n,w_h\right)-\delta_{\tau}b\left( \phi_h^{n+1},\mathbf{u}_h^n,w_h\right)&\\
			-\left(\delta_{\tau\tau}\left(R_h\phi^{n+1}-\phi^{n+1}\right),w_h\right)
			-\left(\delta_\tau R_1^{n+1},w_h\right)&=0
		\end{aligned}
	\end{equation}
	 and
	 \begin{equation}
	 	\label{eq_error_equations_delta_dtau_mu}
	 	\begin{aligned}
	 		-\left(\delta_\tau e_\mu^{n+1},\varphi_h\right)+\lambda\left(\nabla \delta_\tau e_\phi^{n+1},\nabla\varphi_h\right)
	 		+\lambda\left(\nabla\left(\delta_\tau \phi^{n+1}-\delta_\tau R_h\phi^{n+1}\right),\nabla\varphi_h\right)&\\
	 		-\left(\delta_\tau \mu^{n+1}-\delta_\tau \Pi_h\mu^{n+1},\varphi_h\right)
	 		-\frac{\lambda}{\epsilon^2}\left(\delta_\tau (\phi_h^{n})^3-\delta_\tau (\phi^{n})^3+\delta_\tau e_\phi^{n}+\delta_\tau \phi^n-\delta_\tau R_h\phi^n,\varphi_h\right)&=0.
	 	\end{aligned}
	 \end{equation}
	Taking the inner product of \eqref{eq_error_equations_delta_dtau_phi} with $\lambda\delta_\tau e_\phi^{n+1}$ and using the define \eqref{eq_trilinear_forms} leads to
	\begin{equation}
		\label{eq_error_equations_delta_dtau_phi_inner_product_delta_tau_e_phi}
		\begin{aligned}
			\lambda\frac{\|\delta_\tau e_\phi^{n+1}\|^2-\|\delta_\tau e_\phi^n\|^2+\|\delta_\tau e_\phi^{n+1}-\delta_\tau e_\phi^n\|^2}{2\tau}+
			M\lambda\left(\nabla\delta_\tau e_\mu^{n+1},\nabla \delta_\tau e_\phi^{n+1}\right)&
			\\+ \lambda\left(\delta_\tau  \nabla R_h\phi^{n+1}\cdot\mathbf{u}^n,\delta_\tau e_\phi^{n+1}\right)-\lambda\left(\delta_\tau \nabla\phi_h^{n+1}\cdot\mathbf{u}_h^n,\delta_\tau e_\phi^{n+1}\right)&\\
			-\lambda\left(\delta_{\tau\tau}\left(R_h\phi^{n+1}-\phi^{n+1}\right),\delta_\tau e_\phi^{n+1}\right)
			-\lambda\left(\delta_\tau R_1^{n+1},\delta_\tau e_\phi^{n+1}\right)&=0
		\end{aligned}
	\end{equation}
    and taking the inner product of \eqref{eq_error_equations_delta_dtau_mu} with $\frac{M}{2}\delta_\tau e_\mu^{n+1}$ and
    $\frac{M\lambda}{2}\delta_\tau\Delta e_\phi^{n+1}$, respectively, we can obtain
    \begin{equation}
    	\label{eq_error_equations_delta_dtau_mu_inner_product_delta_tau_e_phi}
    	\begin{aligned}
    		-\frac{M}{2}\|\delta_\tau e_\mu^{n+1}\|^2+\frac{M\lambda}{2}\left(\nabla \delta_\tau e_\phi^{n+1},\nabla\delta_\tau e_\mu^{n+1}\right)
    		+\frac{M\lambda}{2}\left(\nabla\left(\delta_\tau \phi^{n+1}-\delta_\tau R_h\phi^{n+1}\right),\nabla\delta_\tau e_\mu^{n+1}\right)&\\
    		-\frac{M}{2}\left(\delta_\tau \mu^{n+1}-\delta_\tau \Pi_h\mu^{n+1},\delta_\tau e_\mu^{n+1}\right)&\\
    		-\frac{M\lambda}{2\epsilon^2}\left(\delta_\tau (\phi_h^{n})^3-\delta_\tau (\phi^{n})^3+\delta_\tau e_\phi^{n}+\delta_\tau \phi^n-\delta_\tau R_h\phi^n,\delta_\tau e_\mu^{n+1}\right)&=0,
    	\end{aligned}
    \end{equation}
	and 
	\begin{equation}
		\label{eq_error_equations_delta_dtau_mu_inner_product_delta_tau_e_mu}
		\begin{aligned}
			\frac{M\lambda}{2}\left(\nabla\delta_\tau e_\mu^{n+1},\nabla\delta_\tau e_\phi^{n+1}\right)-\frac{M\lambda^2}{2}\|\delta_\tau\Delta e_\phi^{n+1}\|^2
			-\frac{M\lambda^2}{2}\left(\Delta\left(\delta_\tau \phi^{n+1}-\delta_\tau R_h\phi^{n+1}\right),\delta_\tau\Delta e_\phi^{n+1}\right)&\\
			-\frac{M\lambda}{2}\left(\delta_\tau \mu^{n+1}-\delta_\tau \Pi_h\mu^{n+1},\delta_\tau\Delta e_\phi^{n+1}\right)&\\
			-\frac{M\lambda}{2\epsilon^2}\left(\delta_\tau (\phi_h^{n})^3-\delta_\tau (\phi^{n})^3+\delta_\tau e_\phi^{n}+\delta_\tau \phi^n-\delta_\tau R_h\phi^n,\delta_\tau\Delta e_\phi^{n+1}\right)&=0.
		\end{aligned}
	\end{equation}
	Combining \eqref{eq_error_equations_delta_dtau_phi_inner_product_delta_tau_e_phi}, \eqref{eq_error_equations_delta_dtau_mu_inner_product_delta_tau_e_phi} with \eqref{eq_error_equations_delta_dtau_mu_inner_product_delta_tau_e_mu}, we obtain
	\begin{equation}
		\label{eq_error_equations_delta_dtau_mu_inner_product_delta_tau_e_mu_phi}
		\begin{aligned}
			\frac{\lambda}{2\tau}\left(\|\delta_\tau e_\phi^{n+1}\|^2-\|\delta_\tau e_\phi^n\|^2+\|\delta_\tau e_\phi^{n+1}-\delta_\tau e_\phi^n\|^2\right)+\frac{M}{2}\|\delta_\tau e_\mu^{n+1}\|^2+\frac{M\lambda^2}{2}\|\delta_\tau\Delta e_\phi^{n+1}\|^2&=\\
			- \lambda\left(\delta_{\tau}\nabla R_h\phi^{n+1}\cdot\mathbf{u}^n,\delta_\tau e_\phi^{n+1}\right)+\lambda\left(\delta_\tau\nabla \phi_h^{n+1}\cdot\mathbf{u}_h^n,\delta_\tau e_\phi^{n+1}\right)&\\
			+\lambda\left(\delta_{\tau\tau}\left(R_h\phi^{n+1}-\phi^{n+1}\right),\delta_\tau e_\phi^{n+1}\right)
			+\lambda\left(\delta_\tau R_1^{n+1},\delta_\tau e_\phi^{n+1}\right)&\\
			+\frac{M\lambda}{2}\left(\nabla\left(\delta_\tau \phi^{n+1}
			-\delta_\tau R_h\phi^{n+1}\right),\nabla\delta_\tau e_\mu^{n+1}\right)
			-\frac{M}{2}\left(\delta_\tau \mu^{n+1}-\delta_\tau \Pi_h\mu^{n+1},\delta_\tau e_\mu^{n+1}\right)&\\
			-\frac{M\lambda}{2\epsilon^2}\left(\delta_\tau (\phi_h^{n})^3-\delta_\tau (\phi^{n})^3+\delta_\tau e_\phi^{n}
			+\delta_\tau \phi^n-\delta_\tau R_h\phi^n,\delta_\tau e_\mu^{n+1}\right)&\\
			-\frac{M\lambda^2}{2}\left(\Delta\left(\delta_\tau \phi^{n+1}-\delta_\tau R_h\phi^{n+1}\right),\delta_\tau\Delta e_\phi^{n+1}\right)
			-\frac{M\lambda}{2}\left(\delta_\tau \mu^{n+1}-\delta_\tau \Pi_h\mu^{n+1},\delta_\tau\Delta e_\phi^{n+1}\right)&\\
			-\frac{M\lambda^2}{2\epsilon^2}\left(\delta_\tau (\phi_h^{n})^3-\delta_\tau (\phi^{n})^3+\delta_\tau e_\phi^{n}+\delta_\tau \phi^n-\delta_\tau R_h\phi^n,\delta_\tau\Delta e_\phi^{n+1}\right)
			&:=\sum_{i=1}^{9}F_i.
		\end{aligned}
	\end{equation}
	Then we establish the error estimates in each terms of the above equations \eqref{eq_error_equations_delta_dtau_mu_inner_product_delta_tau_e_mu_phi}. 
	 For the term of $\lambda\left(\delta_\tau\nabla \phi_h^{n+1}\cdot\mathbf{u}_h^n,\delta_\tau e_\phi^{n+1}\right)- \lambda\left(\delta_{\tau}\nabla R_h\phi^{n+1}\cdot\mathbf{u}^n,\delta_\tau e_\phi^{n+1}\right)$, using the equation \eqref{eq_stokes_quasi_proojection_equation0002}, we can obtain
	\begin{equation}
		\label{eq_error_equations_delta_dtau_mu_inner_product_delta_tau_e_mu_phi_drive_b_F1}
		\begin{aligned}
			\big|F_1\big|=&\big|\lambda\left(\left(\delta_\tau\nabla\phi_h^{n+1}\cdot\mathbf{u}_h^n,\delta_\tau e_\phi^{n+1}\right)
				-\left(\delta_\tau\nabla R_h\phi^{n+1}\cdot\mathbf{u}^n,\delta_\tau e_\phi^{n+1}\right)\right)\big|\\
			=&\big|-\lambda\left(\delta_\tau \nabla e_\phi^{n+1}\cdot\mathbf{u}_h^n,\delta_\tau e_\phi^{n+1}\right)
				-\lambda\left(\delta_\tau \nabla R_h\phi^{n+1}\cdot(\mathbf{u}^n-\mathbf{P}_h\mathbf{u}^n),\delta_\tau e_\phi^{n+1}\right)\\
				&-\lambda\left(\delta_\tau \nabla R_h\phi^{n+1}\cdot e_{\mathbf{u}}^n,\delta_\tau e_\phi^{n+1}\right)\big|\\
			\leq&~C\|\delta_\tau\nabla e_\phi^{n+1}\|\|\mathbf{u}_h^n\|_{L^4}\|\delta_\tau e_\phi^{n+1}\|_{L^4}+C\|\delta_\tau \nabla R_h\phi^{n+1}\|\| \mathbf{u}^n-\mathbf{P}_h\mathbf{u}^n\|_{L^4}\|\delta_\tau e_\phi^{n+1}\|_{L^4}\\
				&+C\|\delta_\tau \nabla R_h\phi^{n+1}\|\| e_{\mathbf{u}}^n\|_{L^4}\|\delta_\tau e_\phi^{n+1}\|_{L^4}\\
			\leq &~C\left(\frac{1}{3}\|\delta_\tau \nabla e_\phi^{n+1}\|^2+\|\nabla e_{\mathbf{u}}^n\|^2+\tau^2+h^{2(r+1)}\right).
		\end{aligned}
	\end{equation}
	By \eqref{eq_error_equations_delta_dtau_mu_inner_product_delta_tau_e_mu_phi}, we next estimate the error for each term in turn, as follows, and denote
	\begin{equation}
		\begin{aligned}
			\big|F_2\big|=&\big|\lambda\left(\delta_{\tau\tau}\left(R_h\phi^{n+1}-\phi^{n+1}\right),\delta_\tau e_\phi^{n+1}\right)\big|\\
			=&\frac{\lambda}{\tau}\left(\delta_\tau\left(R_h\phi^{n+1}-\phi^{n+1}\right)-\delta_\tau\left(R_h\phi^n-\phi^n\right),\delta_\tau e_\phi^{n+1}\right)\\
			\leq &~C\|\delta_\tau\left(R_h\phi^{n+1}-\phi^{n+1}\right\|\|\delta_\tau e_\phi^{n+1}\|_{L^4}+C\|\delta_\tau\left(R_h\phi^n-\phi^n\right)\|\|\delta_\tau e_\phi^{n+1}\|_{L^4}\\
			\leq&~C\left(\frac{1}{3}\|\delta_\tau\nabla e_\phi^{n+1}\|^2+\mathcal{E}_h^2+\tau^2\right)
		\end{aligned}
	\end{equation}	
	and using the Lemma \ref{Lemma_0401}, we have
	\begin{equation}
	\begin{aligned}
		\big|F_3\big|=&\big|\lambda\left(\delta_\tau R_1^{n+1},\delta_\tau e_\phi^{n+1}\right)\big|
		\leq C\|\delta_\tau R_1^{n+1}\|\|\delta_\tau e_\phi^{n+1}\|_{L^4}
		\leq C\left(\frac{1}{3}\|\delta_\tau \nabla e_\phi^{n+1}\|+\tau^2\right).
	\end{aligned}
	\end{equation}
	The next term of \eqref{eq_error_equations_delta_dtau_mu_inner_product_delta_tau_e_mu_phi} is estimated as follows,
	\begin{equation}
		\begin{aligned}
			\big|F_4\big|=\big|\frac{M\lambda}{2}\left(\nabla\left(\delta_\tau \phi^{n+1}-\delta_\tau R_h\phi^{n+1}\right),\nabla\delta_\tau e_\mu^{n+1}\right)\big|=0
		\end{aligned}
	\end{equation}
	and using the Lemma \ref{Lemma_Ritz_quasi_projection}, we have 
	\begin{equation}
		\begin{aligned}
			\big|F_5\big|=&\big|-\frac{M}{2}\left(\delta_\tau \mu^{n+1}-\delta_\tau \Pi_h\mu^{n+1},\delta_\tau e_\mu^{n+1}\right)\big|
			\leq C\|\delta_\tau \left(\mu^{n+1}- \Pi_h\mu^{n+1}\right)\|_{H^{-1}}\|\delta_\tau \nabla e_\mu^{n+1}\|\\
			\leq&~C\left(h^{2(r+1)}+\frac{1}{2}\|\delta_\tau\nabla e_\mu^{n+1}\|^2\right).
		\end{aligned}
	\end{equation}
	 We next estimate that
	\begin{equation}
		\begin{aligned}
			\big|F_6\big|=&~\big|-\frac{M\lambda}{2\epsilon^2}\left(\delta_\tau (\phi_h^{n})^3-\delta_\tau (\phi^{n})^3+\delta_\tau e_\phi^{n}+\delta_\tau \phi^n-\delta_\tau R_h\phi^n,\delta_\tau e_\mu^{n+1}\right)\big|\\
			= &~\big|-\frac{M\lambda}{2\epsilon^2}\left(\left(\delta_\tau (\phi_h^{n})^3-\delta_\tau (\phi^{n})^3,\delta_\tau e_\mu^{n+1}\right)
				+\left(\delta_\tau e_\phi^{n},\delta_\tau e_\mu^{n+1}\right)+\left(\delta_\tau \phi^n-\delta_\tau R_h\phi^n,\delta_\tau e_\mu^{n+1}\right)\right)\big|\\
				\leq&~C\|\nabla\left(\delta_\tau (\phi_h^{n})^3-\delta_\tau (\phi^{n})^3\right)\|\|\delta_\tau e_\mu^{n+1}\|_{L^4}
				+C\|\delta_\tau\nabla e_\phi^{n}\|\|\delta_\tau e_\mu^{n+1}\|_{L^4}\\
				&+C\|\delta_\tau \left(\phi^n- R_h\phi^n\right)\|_{H^{-1}}\|\delta_\tau\nabla e_\mu^{n+1}\|\\
				\leq&~C\left(\|\delta_\tau\left(\phi^{n}-R_h\phi^{n}\right)\|_{H^{-1}}^2+\|\delta_\tau \left(\phi^n- R_h\phi^n\right)\|_{H^{-1}}^2+\|\delta_\tau\nabla e_\phi^{n}\|^2+\|\delta_\tau\nabla e_\mu^{n+1}\|^2\right)\\
				\leq&~C\left(h^{2(r+1)}+\frac{1}{2}\|\delta_\tau\nabla e_\phi^{n}\|^2+\frac{1}{2}\|\delta_\tau\nabla e_\mu^{n+1}\|^2\right)
		\end{aligned}
	\end{equation}
	and
	\begin{equation}
		\begin{aligned}
			\big|F_7\big|=&\big|-\frac{M\lambda^2}{2}\left(\Delta\left(\delta_\tau \phi^{n+1}-\delta_\tau R_h\phi^{n+1}\right),\delta_\tau\Delta e_\phi^{n+1}\right)\big|\\
			=&\big|\frac{M\lambda^2}{2}\delta_\tau\left(\nabla\left( \phi^{n+1}- R_h\phi^{n+1}\right),\nabla\delta_\tau\Delta e_\phi^{n+1}\right)\big|=0
		\end{aligned}
	\end{equation}
	and
	\begin{equation}
		\begin{aligned}
			\big|F_8\big|=&\big|-\frac{M\lambda}{2}\left(\delta_\tau \mu^{n+1}-\delta_\tau \Pi_h\mu^{n+1},\delta_\tau\Delta e_\phi^{n+1}\right)\big|\\
			\leq&~C\|\left(\delta_\tau \left(\mu^{n+1}- \Pi_h\mu^{n+1}\right)\right)\|_{L^4}\|\delta_\tau\Delta e_\phi^{n+1}\|\\
			\leq &~C\left(\|\nabla\left(\delta_\tau \left(\mu^{n+1}- \Pi_h\mu^{n+1}\right)\right)\|^2+\|\delta_\tau\Delta e_\phi^{n+1}\|^2\right)\\
			\leq&~C\left(h^{2(r+1)}+\|\delta_\tau\Delta e_\phi^{n+1}\|^2\right)
		\end{aligned}
	\end{equation}
	and
	\begin{equation}
		\begin{aligned}
			\big|F_9\big|=&\big|-\frac{M\lambda^2}{2\epsilon^2}\left(\delta_\tau (\phi_h^{n})^3-\delta_\tau (\phi^{n})^3+\delta_\tau e_\phi^{n}+\delta_\tau \phi^n-\delta_\tau R_h\phi^n,\delta_\tau\Delta e_\phi^{n+1}\right)\big|\\
			= &\big|-\frac{M\lambda^2}{2\epsilon^2}\left(\delta_\tau (\phi_h^{n})^3-\delta_\tau (\phi^{n})^3,\delta_\tau\Delta e_\mu^{n+1}\right)\\
				&-\frac{M\lambda^2}{2\epsilon^2}\left(\delta_\tau e_\phi^{n},\delta_\tau \Delta e_\mu^{n+1}\right)-\frac{M\lambda^2}{2\epsilon^2}\left(\delta_\tau \phi^n-\delta_\tau R_h\phi^n,\delta_\tau\Delta e_\mu^{n+1}\right)\big|\\
			=&\big|\frac{M\lambda^2}{2\epsilon^2}\left(\nabla\left(\delta_\tau (\phi_h^{n})^3-\delta_\tau (\phi^{n})^3\right),\delta_\tau\nabla e_\mu^{n+1}\right)\\
				&+\frac{M\lambda^2}{2\epsilon^2}\left(\delta_\tau\nabla e_\phi^{n},\delta_\tau \nabla e_\mu^{n+1}\right)-\frac{M\lambda^2}{2\epsilon^2}\left(\delta_\tau\left(\nabla\left( \phi^n- R_h\phi^n\right)\right),\delta_\tau\nabla e_\mu^{n+1}\right)\big|\\
			\leq &~C\|\nabla\left(\delta_\tau (\phi_h^{n})^3-\delta_\tau (\phi^{n})^3\right)\|_{H^{-1}}\|\delta_\tau\Delta e_\mu^{n+1}\|+C\|\nabla\delta_\tau e_\phi^{n}\|_{H^{-1}}\|\delta_\tau\Delta e_\mu^{n+1}\|\\
			\leq &~C\left(\|\delta_\tau\Delta e_\mu^{n+1}\|^2+\|(-\Delta_h)^{-\frac{1}{2}}\delta_\tau \left(\phi^{n}-\phi_h^{n}\right)\|^2+\|\delta_\tau\nabla e_\phi^{n}\|^2+\|\delta_\tau\nabla e_\mathbf{u}^{n}\|^2+\tau^2+h^{2(r+1)}\right)\\
			\leq&~C\left(\|\delta_\tau\Delta e_\mu^{n+1}\|^2+\|\nabla e_\mu^{n+1}\|^2+\|\nabla e_\phi^{n}\|^2+\|\nabla e_\mathbf{u}^{n+1}\|^2\right)\\
			&+C\left(\frac{1}{2}\|\delta_\tau\nabla e_\phi^{n}\|^2+\|\delta_\tau\nabla e_\mathbf{u}^{n}\|^2+\tau^2+h^{2(r+1)}\right).
		\end{aligned}
	\end{equation}
	Thus, combining \eqref{eq_error_equations_delta_dtau_mu_inner_product_delta_tau_e_mu_phi} with $F_1,\cdots,F_9$, we can obtain
		\begin{equation}
		\label{eq_error_equations_delta_dtau_mu_inner_product_delta_tau_e_mu_phi_finally_results}
		\begin{aligned}
			&\frac{\lambda}{2\tau}\left(\|\delta_\tau e_\phi^{n+1}\|^2-\|\delta_\tau e_\phi^n\|^2
				+\|\delta_\tau e_\phi^{n+1}-\delta_\tau e_\phi^n\|^2\right)+\frac{M}{2}\|\delta_\tau e_\mu^{n+1}\|^2
				+\frac{M\lambda^2}{2}\|\delta_\tau\Delta e_\phi^{n+1}\|^2\\
				\leq&~C\left(\|\delta_\tau\Delta e_\mu^{n+1}\|^2+\|\delta_\tau\nabla e_\mu^{n+1}\|^2+\|\nabla e_\mu^{n+1}\|^2+\|\delta_\tau\Delta e_\phi^{n+1}\|^2+\|\delta_\tau\nabla e_\phi^{n+1}\|^2+\|\nabla e_\phi^{n}\|^2\right)\\&
					+C\left(\|\delta_\tau\nabla e_\phi^{n}\|^2+\|\delta_\tau e_\phi^n\|^2+\|\delta_\tau\nabla e_\mathbf{u}^{n}\|^2+\|\nabla e_\mathbf{u}^{n+1}\|^2+\|\nabla e_\mathbf{u}^{n}\|^2+h^{2(r+1)}+\tau^2\right).
		\end{aligned}
	\end{equation}
	Mutliplying \eqref{eq_error_equations_delta_dtau_mu_inner_product_delta_tau_e_mu_phi_finally_results} by $2\tau$ and  summing up for $n$ from $1$ to $m$, 
	then using the Lemma \ref{lemma_discrete_Gronwall_inequation} and Theorem \ref{theorem_error_estimates_u_phi_mu_r}, we have 
	\begin{equation}
		\label{eq_error_equations_summing_up_1m}
		\begin{aligned}
			&\|\delta_\tau e_\phi^{m+1}\|^2
			+\sum_{n=1}^{m}\|\delta_\tau e_\phi^{n+1}-\delta_\tau e_\phi^n\|^2+\tau\sum_{n=1}^{m}\|\delta_\tau e_\mu^{n+1}\|^2
			+\tau\sum_{n=0}^{m}\|\delta_\tau\Delta e_\phi^{n+1}\|^2\\
			\leq&\|\delta_\tau e_\phi^1\|^2+\tau\|\delta_\tau\nabla e_\phi^1\|^2+C\tau\sum_{n=1}^{m}\|\delta_\tau e_\mathbf{u}^n\|^2+C\left(h^{2(r+1)}+\tau^2\right),
		\end{aligned}
	\end{equation} 
	where $m=1,2,\cdots,N-1$.
	Thus, we need to estimate $\|\delta_\tau e_\phi^1\|^2$ and $\|\delta_\tau\nabla e_\phi^1\|^2$. Taking the inner product of \eqref{eq_error_equations_phi} with $\delta_\tau e_\phi^1$ when $n=0$ leads to
	\begin{equation}
		\label{eq_error_equations_phi_n0}
		\left(\delta_\tau e_\phi^{1},\delta_\tau e_\phi^1\right)+M\left(\nabla e_\mu^{1},\nabla\delta_\tau e_\phi^1\right)-\left(\delta_\tau\left(R_h\phi^{1}-\phi^{1}\right),\delta_\tau e_\phi^1\right)
		-\left(R_1^{1},\delta_\tau e_\phi^1\right)=0,
	\end{equation}
	we derive that
	\begin{equation}
		\begin{aligned}
			\|\delta_\tau e_\phi^1\|^2=&-M\left(\nabla e_\mu^{1},\nabla\delta_\tau e_\phi^1\right)+\left(\delta_\tau\left(R_h\phi^{1}-\phi^{1}\right),\delta_\tau e_\phi^1\right)
				+\left(R_1^{1},\delta_\tau e_\phi^1\right)\\
			\leq &~\left(\nabla\left(\mu_h^1-\Pi_h\mu^1\right),\nabla\delta_\tau e_\phi^1\right)+\|\delta_\tau\left(R_h\phi^{1}-\phi^{1}\right)\|\|\delta_\tau e_\phi^1\|_{L^4}
				+\|R_1^1\|\|\delta_\tau e_\phi^1\|_{L^4}\\
			\leq &~\|\nabla\left(\mu_h^1-\Pi_h\mu^1\right)\|_{H^{-1}} \|\nabla\delta_\tau e_\phi^1\|+\|\delta_\tau\left(R_h\phi^{1}-\phi^{1}\right)\|\|\delta_\tau e_\phi^1\|_{L^4}
				+\|R_1^1\|\|\delta_\tau e_\phi^1\|_{L^4}\\
			\leq &~C\left(h^{2(r+1)}+\tau^2+\|\delta_\tau\nabla e_\phi^1\|^2\right).
		\end{aligned}
	\end{equation}
	Testing $\varphi_h=\delta_\tau e_\phi^1$ in \eqref{eq_error_equations_delta_dtau_mu} when $n=0$, we can obtain
	\begin{equation}
		\label{eq_delta_t_e_phi1}
		\begin{aligned}
			-\left(\delta_\tau e_\mu^{1},\delta_\tau e_\phi^1\right)+\lambda\left(\nabla \delta_\tau e_\phi^{1},\nabla\delta_\tau e_\phi^1\right)
			+\lambda\left(\nabla\left(\delta_\tau \phi^{1}-\delta_\tau R_h\phi^{1}\right),\nabla\delta_\tau e_\phi^1\right)&\\
			-\left(\delta_\tau \mu^{1}-\delta_\tau \Pi_h\mu^{1},\delta_\tau e_\phi^1\right)
			-\frac{\lambda}{\epsilon^2}\left(\delta_\tau (\phi_h^{0})^3-\delta_\tau (\phi^{0})^3,\delta_\tau e_\phi^1\right)&=0,
		\end{aligned}
	\end{equation}
	 and deduce that
	 \begin{equation}
	 	\label{eq_delta_t_nabal_e_phi1}
	 	\begin{aligned}
	 		\| \delta_\tau\nabla e_\phi^{1}\|^2=&\frac{1}{\lambda}\left(\left(\delta_\tau e_\mu^{1},\delta_\tau e_\phi^1\right)
	 			+\left(\delta_\tau \mu^{1}-\delta_\tau \Pi_h\mu^{1},\delta_\tau e_\phi^1\right)\right)
	 		+\frac{1}{\epsilon^2}\left(\delta_\tau (\phi_h^{0})^3-\delta_\tau (\phi^{0})^3,\delta_\tau e_\phi^1\right)\\
	 		&-\left(\nabla\left(\delta_\tau \phi^{1}-\delta_\tau R_h\phi^{1}\right),\nabla\delta_\tau e_\phi^1\right)\\
	 		\leq &~C\|\mu^1-\mu_h^1\|\|\delta_\tau e_\phi^1\|_{H^{-1}}
	 			+\|(-\Delta_h)^{-\frac{1}{2}}\delta_\tau \left(\phi^{0}-\phi_h^{0}\right)\|\|\delta_\tau e_\phi^1\|_{H^{-1}}\\
	 		\leq &~C\left(\|\mu^1-\mu_h^1\|^2+\|\delta_\tau e_\phi^1\|_{H^{-1}}^2+\|\delta_\tau \left(\phi^{0}-\phi_h^{0}\right)\|^2\right)\\
	 		\leq &~C\left(h^{2r}+\tau^2+\|\nabla e_\mu^1\|^2+\|\nabla e_\phi^1\|^2+\|\nabla e_\mathbf{u}^1\|^2\right).
	 	\end{aligned}
	 \end{equation}
 	Using the equations \eqref{eq_delta_t_e_phi1}, \eqref{eq_delta_t_nabal_e_phi1} and Theorem
 	\ref{theorem_error_estimates_u_phi_mu_r}, we have 
 	\begin{equation}
 		\label{eq_error_estimates_phi_mu_results_delta_tau}
 			\begin{aligned}
 			&\|\delta_\tau e_\phi^{m+1}\|^2
 			+\sum_{n=1}^{m}\|\delta_\tau e_\phi^{n+1}-\delta_\tau e_\phi^n\|^2+M\tau\sum_{n=1}^{m}\|\delta_\tau e_\mu^{n+1}\|^2
 			+M\lambda\tau\sum_{n=0}^{m}\|\delta_\tau\Delta e_\phi^{n+1}\|^2\\
 			\leq&\|\delta_\tau e_\phi^1\|^2+\tau\|\delta_\tau\nabla e_\phi^1\|^2+C\tau\sum_{n=1}^{m}\|\delta_\tau e_\mathbf{u}^n\|^2+C\left(h^{2r}+\tau^2\right)\\
 			\leq &C\tau\sum_{n=1}^{m}\|\delta_\tau e_\mathbf{u}^n\|^2+C\left(h^{2(r+1)}+\tau^2\right),
 		\end{aligned}
 	\end{equation}
 	where $m=1,2,\cdots,N-1$.
	\end{proof}
	Next, we establish an estimate for $\|\delta_\tau e_{\mathbf{u}}^{n+1}\|$. Combining \eqref{eq_error_equations_u_p_cor} with \eqref{eq_error_equations_ns} results in
	\begin{equation}
		\label{eq_combining_error_eqautions_and_projections_term}
		\begin{aligned}
			\left(\delta_\tau e_{{\mathbf{u}}}^{n+1},\mathbf{v}_h\right)+\nu\left(\nabla e_{\tilde{\mathbf{u}}}^{n+1},\nabla\mathbf{v}_h\right)
			+\left(\nabla e_p^{n+1},\mathbf{v}_h\right)\\
			+\left(\frac{\rho^{n+1}}{\sqrt{E_1^{n+1}}}\left(\mathbf{u}^{n+1}\cdot\nabla\mathbf{u}^{n+1},\mathbf{v}_h\right)
			-\frac{\rho_h^n}{\sqrt{E_{1,h}^{n+1}}}\left(\mathbf{u}_h^{n}\cdot\nabla\mathbf{u}_h^{n},\mathbf{v}_h\right)\right)
			+\left(\delta_\tau\left(\mathbf{u}^{n+1}-\mathbf{P}_h\mathbf{u}^{n+1}\right),\mathbf{v}_h\right)\\
			+\left(b\left(\phi_h^{n+1},\mu_h^{n+1},\mathbf{v}_h\right)-b\left(R_h\phi^{n+1},\mu^{n+1},\mathbf{v}_h\right)\right)-
			\left(R_2^{n+1},\mathbf{v}_h\right)-\left(R_3^{n+1},\mathbf{v}_h\right)&=0.\\
		\end{aligned}
	\end{equation}
	The result obtained by subtracting the two consecutive steps 
	in \eqref{eq_combining_error_eqautions_and_projections_term} is shown below that
	\begin{equation}
		\label{eq_combining_error_eqautions_and_projections_term_difference}
		\begin{aligned}
			\left(\delta_{\tau\tau} e_{{\mathbf{u}}}^{n+1},\mathbf{v}_h\right)+\nu\left(\nabla\delta_\tau e_{\tilde{\mathbf{u}}}^{n+1},\nabla\mathbf{v}_h\right)
			+\left(\nabla\delta_\tau e_p^{n+1},\mathbf{v}_h\right)\\
			+\delta_\tau\left(\frac{\rho^{n+1}}{\sqrt{E_1^{n+1}}}\left(\mathbf{u}^{n+1}\cdot\nabla\mathbf{u}^{n+1},\mathbf{v}_h\right)-\frac{\rho_h^n}{\sqrt{E_{1,h}^{n+1}}}\left(\mathbf{u}_h^{n}\cdot\nabla\mathbf{u}_h^{n},\mathbf{v}_h\right)\right)
			+\left(\delta_{\tau\tau}\left(\mathbf{u}^{n+1}-\mathbf{P}_h\mathbf{u}^{n+1}\right),\mathbf{v}_h\right)\\
			+\delta_\tau\left(b\left(\phi_h^{n+1},\mu_h^{n+1},\mathbf{v}_h\right)-b\left(R_h\phi^{n+1},\mu^{n+1},\mathbf{v}_h\right)\right)-
			\left(\delta_\tau R_2^{n+1},\mathbf{v}_h\right)-\left(\delta_\tau R_3^{n+1},\mathbf{v}_h\right)&=0.
		\end{aligned}
	\end{equation}
	We choose the $\mathbf{v}_h=\delta_\tau e_{\tilde{\mathbf{u}}}^{n+1}$ in \eqref{eq_combining_error_eqautions_and_projections_term_difference} to find
	\begin{equation}
		\label{eq_error_estimates_delta_t_e_u}
		\begin{aligned}
				\left(\delta_{\tau\tau} e_{{\mathbf{u}}}^{n+1},\delta_\tau e_{\tilde{\mathbf{u}}}^{n+1}\right)+\nu\left(\nabla\delta_\tau e_{\tilde{\mathbf{u}}}^{n+1},\nabla\delta_\tau e_{\tilde{\mathbf{u}}}^{n+1}\right)
			+\left(\nabla\delta_\tau e_p^{n+1},\delta_\tau e_{\tilde{\mathbf{u}}}^{n+1}\right)-
			\left(\delta_\tau R_2^{n+1},\delta_\tau e_{\tilde{\mathbf{u}}}^{n+1}\right)\\
			+\delta_\tau\left(\frac{\rho^{n+1}}{\sqrt{E_1^{n+1}}}
			\left(\mathbf{u}^{n+1}\cdot\nabla\mathbf{u}^{n+1},\delta_\tau e_{\tilde{\mathbf{u}}}^{n+1}\right)
			-\frac{\rho_h^n}{\sqrt{E_{1,h}^{n+1}}}\left(\mathbf{u}_h^{n}\cdot\nabla\mathbf{u}_h^{n},\delta_\tau e_{\tilde{\mathbf{u}}}^{n+1}\right)\right)\\
			+\left(\delta_{\tau\tau}\left(\mathbf{u}^{n+1}-\mathbf{P}_h\mathbf{u}^{n+1}\right),\delta_\tau e_{\tilde{\mathbf{u}}}^{n+1}\right)\\
			+\delta_\tau\left(b\left(\phi_h^{n+1},\mu_h^{n+1},\delta_\tau e_{\tilde{\mathbf{u}}}^{n+1}\right)-b\left(R_h\phi^{n+1},\mu^{n+1},\delta_\tau e_{\tilde{\mathbf{u}}}^{n+1}\right)\right)-\left(\delta_\tau R_3^{n+1},\delta_\tau e_{\tilde{\mathbf{u}}}^{n+1}\right)&=0.
		\end{aligned}
	\end{equation}
 	Since we can be recast \eqref{eq_error_equations_u_p_cor} into
 	\begin{equation}
 		\delta_\tau e_{\tilde{\mathbf{u}}}^{n+1}=\delta_\tau e_{\mathbf{u}}^{n+1}+\nabla\left(p_h^{n+1}-2p_h^n+p_h^{n-1}\right).
 	\end{equation}
	and obtain to the following result
	\begin{equation}
		\begin{aligned}
			\left(\delta_{\tau\tau}e_{\mathbf{u}}^{n+1},\delta_\tau e_{\tilde{\mathbf{u}}}^{n+1}\right)=\frac{\|\delta_{\tau}e_{\mathbf{u}}^{n+1}\|^2-\|\delta_{\tau}e_{\mathbf{u}}^{n}\|^2+\|\delta_{\tau}e_{\mathbf{u}}^{n+1}-\delta_{\tau}e_{\mathbf{u}}^{n}\|^2}{2\tau}.
		\end{aligned}
	\end{equation}
	Thus, the equation \eqref{eq_error_estimates_delta_t_e_u} can be rewitten as
	\begin{equation}
		\label{eq_error_estimates_delta_t_e_u_rewrite}
		\begin{aligned}
		&\frac{\|\delta_{\tau}e_{\mathbf{u}}^{n+1}\|^2-\|\delta_{\tau}e_{\mathbf{u}}^{n}\|^2+\|\delta_{\tau}e_{\mathbf{u}}^{n+1}-\delta_{\tau}e_{\mathbf{u}}^{n}\|^2}{2\tau}+\nu\|\nabla\delta_\tau e_{\tilde{\mathbf{u}}}^{n+1}\|^2\\
			=&\left(\delta_\tau R_2^{n+1},\delta_\tau e_{\tilde{\mathbf{u}}}^{n+1}\right)
			-\left(\nabla\delta_\tau e_p^{n+1},\delta_\tau e_{\tilde{\mathbf{u}}}^{n+1}\right)\\
			&+\delta_\tau\left(\frac{\rho^{n+1}}{\sqrt{E_1^{n+1}}}
			\left(\mathbf{u}^{n+1}\cdot\nabla\mathbf{u}^{n+1},\delta_\tau e_{\tilde{\mathbf{u}}}^{n+1}\right)
			-\frac{\rho_h^n}{\sqrt{E_{1,h}^{n+1}}}\left(\mathbf{u}_h^{n}\cdot\nabla\mathbf{u}_h^{n},\delta_\tau e_{\tilde{\mathbf{u}}}^{n+1}\right)\right)\\
			&-\left(\delta_{\tau\tau}\left(\mathbf{u}^{n+1}-\mathbf{P}_h\mathbf{u}^{n+1}\right),\delta_\tau e_{\tilde{\mathbf{u}}}^{n+1}\right)\\
			&+\left(\delta_\tau(\nabla R_h\phi^{n+1}\cdot\mu^{n+1}),\delta_\tau e_{\tilde{\mathbf{u}}}^{n+1}\right)-\left(\delta_\tau\left(\nabla\phi_h^{n+1}\cdot\mu_h^{n+1}\right),\delta_\tau e_{\tilde{\mathbf{u}}}^{n+1}\right)
			+\left(\delta_\tau R_3^{n+1},\delta_\tau e_{\tilde{\mathbf{u}}}^{n+1}\right).
		\end{aligned}
	\end{equation}
	We now estimate terms on the right-hand side of \eqref{eq_error_estimates_delta_t_e_u_rewrite} as follows:
	\begin{equation}
		\begin{aligned}
			\left(\delta_\tau R_2^{n+1},\delta_\tau e_{\tilde{\mathbf{u}}}^{n+1}\right)\leq \frac{\nu}{10}\|\nabla\delta_\tau e_{\tilde{\mathbf{u}}}^{n+1}\|^2+C\tau^2
		\end{aligned}
	\end{equation}
	and 
	\begin{equation}
		\begin{aligned}
			\big|-\left(\nabla\delta_\tau e_p^{n+1},\delta_\tau e_{\tilde{\mathbf{u}}}^{n+1}\right)\big|=\big|-\left(\nabla\delta_\tau e_p^{n},\delta_\tau e_{\tilde{\mathbf{u}}}^{n+1}\right)-\left(\nabla\left(\delta_\tau e_p^{n+1}-\delta_\tau e_p^{n}\right),\delta_\tau e_{\tilde{\mathbf{u}}}^{n+1}\right)\big|,
		\end{aligned}
	\end{equation}
	and the term $-\left(\nabla\delta_\tau e_p^{n},\delta_\tau e_{\tilde{\mathbf{u}}}^{n+1}\right)$ is estimated by
	\begin{equation}
		\label{eq_error_estimate_delta_t_p_u_one}
		\begin{aligned}
			\big|-\left(\nabla\delta_\tau e_p^{n},\delta_\tau e_{\tilde{\mathbf{u}}}^{n+1}\right)\big|
			=&\big|-\left(\nabla\delta_\tau e_p^{n},\nabla\left(p_h^{n+1}-2p_h^n+p_h^{n-1}\right)\right)\big|\\
			=&\big|-\tau\left(\nabla\delta_\tau e_p^{n},\nabla\left(\delta_\tau e_p^{n+1}-\delta_\tau e_p^n\right)\right)-\left(\nabla\delta_\tau e_p^{n+1},\nabla\delta_\tau R_4^{n+1} \right)\big|\\
			=&\big|-\frac{\tau}{2}\left(\|\nabla\delta_\tau e_p^{n+1}\|^2-\|\nabla\delta_\tau e_p^{n}\|^2-\|\nabla\delta_\tau e_p^{n+1}-\nabla\delta_\tau e_p^{n}\|^2\right)\\
				&-\left(\nabla\delta_\tau e_p^{n+1},\nabla\delta_\tau R_4^{n+1} \right)\big|\\
			\leq&-\frac{\tau}{2}\left(\|\nabla\delta_\tau e_p^{n+1}\|^2-\|\nabla\delta_\tau e_p^{n}\|^2-\|\nabla\delta_\tau e_p^{n+1}-\nabla\delta_\tau e_p^{n}\|^2\right)\\
				&+C\left(\tau^2\|\nabla\delta_\tau e_p^n\|^2+\tau^2\right)
		\end{aligned}
	\end{equation}
	and the term $-\left(\nabla\left(\delta_\tau e_p^{n+1}-\delta_\tau e_p^{n}\right),\delta_\tau e_{\tilde{\mathbf{u}}}^{n+1}\right)$ is 
	bounded by
	\begin{equation}
		\label{eq_error_estimate_delta_t_p_u_two}
		\begin{aligned}
			&\big|-\left(\nabla\left(\delta_\tau e_p^{n+1}-\delta_\tau e_p^{n}\right),\delta_\tau e_{\tilde{\mathbf{u}}}^{n+1}\right)\big|
				=\big|-\left(\nabla\left(\delta_\tau e_p^{n+1}-\delta_\tau e_p^{n}\right),\nabla\left(p_h^{n+1}-2p_h^n+p_h^{n-1}\right)\right)\big|\\
			=&\big|-\tau\left(\nabla\left(\delta_\tau e_p^{n+1}-\delta_\tau e_p^{n}\right),\nabla\left(\delta_\tau e_p^{n+1}-\delta_\tau e_p^{n}\right)\right)-\left(\nabla\left(\delta_\tau e_p^{n+1}-\delta_\tau e_p^{n}\right),\nabla\delta_\tau R_4^{n+1}\right)\big|\\
			\leq&-\frac{\tau}{2}\|\nabla\delta_\tau e_p^{n+1}-\nabla\delta_\tau e_p^{n}\|^2+C\tau^2.
		\end{aligned}
	\end{equation}
	Combining \eqref{eq_error_estimate_delta_t_p_u_one} with \eqref{eq_error_estimate_delta_t_p_u_two}, we have
	\begin{equation}
		-\left(\nabla\delta_\tau e_p^{n+1},\delta_\tau e_{\tilde{\mathbf{u}}}^{n+1}\right)\leq -\frac{\tau}{2}\left(\|\nabla\delta_\tau e_p^{n+1}\|^2-\|\nabla\delta_\tau e_p^{n}\|^2\right)+C\tau^2.
	\end{equation}
 	The term $\delta_\tau\left(\frac{\rho^{n+1}}{\sqrt{E_1^{n+1}}}
 	\left(\mathbf{u}^{n+1}\cdot\nabla\mathbf{u}^{n+1},\delta_\tau e_{\tilde{\mathbf{u}}}^{n+1}\right)
 	-\frac{\rho_h^n}{\sqrt{E_{1,h}^{n+1}}}
 	\left(\mathbf{u}_h^{n}\cdot\nabla\mathbf{u}_h^{n},\delta_\tau e_{\tilde{\mathbf{u}}}^{n+1}\right)\right)$ 
 	can be bounded by
 	\begin{equation}
 		\begin{aligned}
 			&\bigg|\delta_\tau\left(\frac{\rho^{n+1}}{\sqrt{E_1^{n+1}}}
 			\left(\mathbf{u}^{n+1}\cdot\nabla\mathbf{u}^{n+1},\delta_\tau e_{\tilde{\mathbf{u}}}^{n+1}\right)
 			-\frac{\rho_h^n}{\sqrt{E_{1,h}^{n+1}}}\left(\mathbf{u}_h^{n}\cdot\nabla\mathbf{u}_h^{n},\delta_\tau e_{\tilde{\mathbf{u}}}^{n+1}\right)\right)\bigg|\\
 			=~\bigg|&\frac{1}{\tau}\left(\left(\frac{\rho^{n+1}}{\sqrt{E_1^{n+1}}}
 			\left(\mathbf{u}^{n+1}\cdot\nabla\mathbf{u}^{n+1},\delta_\tau e_{\tilde{\mathbf{u}}}^{n+1}\right)
 			-\frac{\rho^{n}}{\sqrt{E_1^{n}}}
 			\left(\mathbf{u}^{n}\cdot\nabla\mathbf{u}^{n},\delta_\tau e_{\tilde{\mathbf{u}}}^{n+1}\right)\right)\right.
 			\\
 			&\left. -\left(\frac{\rho_h^n}{\sqrt{E_{1,h}^{n+1}}}
 			\left(\mathbf{u}_h^{n}\cdot\nabla\mathbf{u}_h^{n},\delta_\tau e_{\tilde{\mathbf{u}}}^{n+1}\right)
 			-\frac{\rho_h^n}{\sqrt{E_{1,h}^{n+1}}}
 			\left(\mathbf{u}_h^{n-1}\cdot\nabla\mathbf{u}_h^{n-1},\delta_\tau e_{\tilde{\mathbf{u}}}^{n+1}\right)\right)	\right)\bigg|\\
 			\leq~&\frac{\nu}{10}\|\nabla\delta_\tau e_{\tilde{\mathbf{u}}}^{n+1}\|^2+C\left(\mathcal{E}_h^2
 					+\tau^2+\|\nabla e_{\tilde{\mathbf{u}}}^{n+1}\|^2+\|\nabla e_{\tilde{\mathbf{u}}}^{n}\|^2\right).
 		\end{aligned}
 	\end{equation}
 	The error estimate of term $-\left(\delta_{\tau\tau}\left(\mathbf{u}^{n+1}-\mathbf{P}_h\mathbf{u}^{n+1}\right),\delta_\tau e_{\tilde{\mathbf{u}}}^{n+1}\right)$ can be bounded by
 	\begin{equation}
 		\begin{aligned}
 			&\big|-\left(\delta_{\tau\tau}\left(\mathbf{u}^{n+1}-\mathbf{P}_h\mathbf{u}^{n+1}\right),\delta_\tau e_{\tilde{\mathbf{u}}}^{n+1}\right)\big|
 			=\big|-\frac{1}{\tau} \left(\delta_\tau\left(\mathbf{u}^{n+1}-\mathbf{P}_h\mathbf{u}^{n+1}\right)-\delta_\tau\left(\mathbf{u}^{n}-\mathbf{P}_h\mathbf{u}^{n}\right),\delta_\tau e_{\tilde{\mathbf{u}}}^{n+1}\right)\big|\\
 			\leq&~\frac{1}{\tau}\left(\|\delta_\tau\left(\mathbf{u}^{n+1}-\mathbf{P}_h\mathbf{u}^{n+1}\right)\|\|\delta_\tau e_{\tilde{\mathbf{u}}}^{n+1}\|_{L^4}+\|\delta_\tau\left(\mathbf{u}^{n}-\mathbf{P}_h\mathbf{u}^{n}\right)\|\|\delta_\tau e_{\tilde{\mathbf{u}}}^{n+1}\|_{L^4}\right)\\
 			\leq&~\frac{\nu}{10}\|\nabla\delta_\tau e_{\tilde{\mathbf{u}}}^{n+1}\|^2+Ch^{2(r+1)}.
 		\end{aligned}
 	\end{equation}
 	We next estimate the term $\left(\delta_\tau(\nabla R_h\phi^{n+1}\cdot\mu^{n+1}),\delta_\tau e_{\tilde{\mathbf{u}}}^{n+1}\right)-\left(\delta_\tau\left(\nabla\phi_h^{n+1}\cdot\mu_h^{n+1}\right),\delta_\tau e_{\tilde{\mathbf{u}}}^{n+1}\right)$ that
 	\begin{equation}
 		\begin{aligned}
 			&\left(\delta_\tau(\nabla R_h\phi^{n+1}\cdot\mu^{n+1}),\delta_\tau e_{\tilde{\mathbf{u}}}^{n+1}\right)-\left(\delta_\tau\left(\nabla\phi_h^{n+1}\cdot\mu_h^{n+1}\right),\delta_\tau e_{\tilde{\mathbf{u}}}^{n+1}\right)\\
 			=&\left(\delta_\tau(\nabla \left(R_h\phi^{n+1}-\phi_h^{n+1}\right)\cdot\mu^{n+1}),\delta_\tau e_{\tilde{\mathbf{u}}}^{n+1}\right)+\left(\delta_\tau(\nabla \phi_h^{n+1}\cdot\mu^{n+1}),\delta_\tau e_{\tilde{\mathbf{u}}}^{n+1}\right)\\
 			&-\left(\delta_\tau\left(\nabla\phi_h^{n+1}\cdot\mu_h^{n+1}\right),\delta_\tau e_{\tilde{\mathbf{u}}}^{n+1}\right)\\
 			=~&\left(\delta_\tau(\nabla \left(R_h\phi^{n+1}-\phi_h^{n+1}\right)\cdot\mu^{n+1}),\delta_\tau e_{\tilde{\mathbf{u}}}^{n+1}\right)+\left(\delta_\tau(\nabla \phi_h^{n+1}\cdot\left(\mu^{n+1}-\Pi_h\mu^{n+1}\right)),\delta_\tau e_{\tilde{\mathbf{u}}}^{n+1}\right)\\
 			&+\left(\delta_\tau(\nabla \phi_h^{n+1}\cdot \left(\Pi_h\mu^{n+1}-\mu_h^{n+1}\right)),\delta_\tau e_{\tilde{\mathbf{u}}}^{n+1}\right)\\
 			=~&\left(\delta_\tau(\nabla e_\phi^{n+1}\cdot\mu^{n+1}),\delta_\tau e_{\tilde{\mathbf{u}}}^{n+1}\right)
 				+\left(\delta_\tau(\nabla \phi_h^{n+1}\cdot\left(\mu^{n+1}-\Pi_h\mu^{n+1}\right)),\delta_\tau e_{\tilde{\mathbf{u}}}^{n+1}\right)\\
 					&+\left(\delta_\tau(\nabla \phi_h^{n+1}\cdot e_\mu^{n+1}),\delta_\tau e_{\tilde{\mathbf{u}}}^{n+1}\right)\\
 			\leq&\|\delta_\tau\nabla e_\phi^{n+1}\|\|\mu^{n+1}\|_{L^4}\|\delta_\tau e_{\tilde{\mathbf{u}}}^{n+1}\|_{L^4}
 					+\|\nabla\phi_h^{n+1}\|_{L^4}\|\nabla(\delta_\tau\left(\mu^{n+1}-\Pi_h\mu^{n+1}\right))\|\|\delta_\tau e_{\tilde{\mathbf{u}}}^{n+1}\|_{L^4}\\
 				&+\|\phi_h^{n+1}\|_{L^4}\|\delta_\tau\nabla e_{\mu}^{n+1}\|\|\delta_\tau e_{\tilde{\mathbf{u}}}^{n+1}\|_{L^4}\\
 			\leq ~&\frac{\nu}{10}\|\nabla\delta_\tau e_{\tilde{\mathbf{u}}}^{n+1}\|^2+C\left(\|\delta_\tau\nabla e_\phi^{n+1}\|^2+\|\delta_\tau\nabla e_{\mu}^{n+1}\|^2+h^{2(r+1)}+\tau^2\right)
 		\end{aligned}
 	\end{equation}
 	The last term $\left(\delta_\tau R_3^{n+1},\delta_\tau e_{\tilde{\mathbf{u}}}^{n+1}\right)$ can be estimated by
 	\begin{equation}
 		\begin{aligned}
 			\left(\delta_\tau R_3^{n+1},\delta_\tau e_{\tilde{\mathbf{u}}}^{n+1}\right)\leq \frac{\nu}{10}\|\nabla\delta_\tau e_{\tilde{\mathbf{u}}}^{n+1}\|^2+C\tau^2.
 		\end{aligned}
 	\end{equation}
 	Thus, 
 	 using the above error estimates, we can obtain
 	\begin{equation}
 		\label{eq_using_the_above_error_estimates_reslut}
 		\begin{aligned}
 			&\frac{\|\delta_{\tau}e_{\mathbf{u}}^{n+1}\|^2-\|\delta_{\tau}e_{\mathbf{u}}^{n}\|^2+\|\delta_{\tau}e_{\mathbf{u}}^{n+1}-\delta_{\tau}e_{\mathbf{u}}^{n}\|^2}{2\tau}+\frac{\nu}{2}\|\nabla\delta_\tau e_{\tilde{\mathbf{u}}}^{n+1}\|^2
 				+\frac{\tau}{2}\left(\|\nabla\delta_\tau e_p^{n+1}\|^2-\|\nabla\delta_\tau e_p^{n+1}\|^2\right)\\
 			\leq&C\left(\|\nabla e_{\tilde{\mathbf{u}}}^{n+1}\|^2+\|\nabla e_{\tilde{\mathbf{u}}}^n\|^2+\|\delta_\tau\nabla e_{\phi}^{n+1}\|^2+\|\delta_\tau\nabla e_\mu^{n+1}\|^2+\mathcal{E}_h^2+\tau^2\right).
 		\end{aligned}
 	\end{equation}
	Multiplying \eqref{eq_using_the_above_error_estimates_reslut} by  $2\tau$ and summing up for $n$ from $1$ to $m$, we get
	\begin{equation}
		\label{eq_summing_up_the_above_error_estimates_reslut}
		\begin{aligned}
			\|\delta_{\tau}e_{\mathbf{u}}^{m+1}\|^2+&\nu\tau\sum_{n=1}^{m}\|\nabla\delta_\tau e_{\tilde{\mathbf{u}}}^{n+1}\|^2+\tau^2\|\nabla\delta_\tau e_p^{m+1}\|^2\\
			\leq&~\|\delta_{\tau}e_{\mathbf{u}}^{1}\|^2+\tau^2\|\nabla\delta_\tau e_p^{1}\|^2\\
				&+C\tau\sum_{n=1}^{m}\left(\|\nabla e_{\tilde{\mathbf{u}}}^{n+1}\|^2+\|\nabla e_{\tilde{\mathbf{u}}}^n\|^2+\|\delta_\tau\nabla e_{\phi}^{n+1}\|^2+\|\delta_\tau\nabla e_\mu^{n+1}\|^2\right)
				+C\left(h^{2(r+1)}+\tau^2\right).
		\end{aligned}
	\end{equation}
	Recalling the error equation \eqref{eq_combining_error_eqautions_and_projections_term}, on the right-hand side of \eqref{eq_summing_up_the_above_error_estimates_reslut}, the first two terms in time $n=0$ can be estimated by 
	\begin{equation}
		\label{eq_recast_error_equations_n0}
		\begin{aligned}
				\left(\delta_\tau e_{{\mathbf{u}}}^{1},\mathbf{v}_h\right)+\nu\left(\nabla e_{\tilde{\mathbf{u}}}^{1},\nabla\mathbf{v}_h\right)
			+\left(\nabla e_p^{1},\mathbf{v}_h\right)\\
			+\left(\frac{\rho^1}{\sqrt{E_1(\phi^1)}}\left(\mathbf{u}^{1}\cdot\nabla\mathbf{u}^{1},\mathbf{v}_h\right)
			-\frac{\rho_h^0}{\sqrt{E_{1,h}^1}}\left(\mathbf{u}_h^{0}\cdot\nabla\mathbf{u}_h^{0},\mathbf{v}_h\right)\right)
			+\left(\delta_\tau\left(\mathbf{u}^{1}-\mathbf{P}_h\mathbf{u}^{1}\right),\mathbf{v}_h\right)\\
			+\left(b\left(\phi_h^{1},\mu_h^{1},\mathbf{v}_h\right)-b\left(R_h\phi^{1},\mu^{1},\mathbf{v}_h\right)\right)-
			\left(R_2^{1},\mathbf{v}_h\right)-\left(R_3^{1},\mathbf{v}_h\right)&=0.\\
		\end{aligned}
	\end{equation}
	We next choose $\mathbf{v}_h=\delta_\tau e_{\mathbf{u}}^1$ in \eqref{eq_recast_error_equations_n0} and use the $\nabla\cdot e_{\mathbf{u}}^{n+1}=0$ to obtain
	\begin{equation}
		\begin{aligned}
			\|\delta_\tau e_{{\mathbf{u}}}^{1}\|^2=&-\left(\delta_\tau\left(\mathbf{u}^{1}-\mathbf{P}_h\mathbf{u}^{1}\right),\delta_\tau e_{\mathbf{u}}^1\right)-
				\left(b\left(\phi_h^{1},\mu_h^{1},\delta_\tau e_{\mathbf{u}}^1\right)-b\left(R_h\phi^{1},\mu^{1},\delta_\tau e_{\mathbf{u}}^1\right)\right)+
			\left(R_2^{1},\delta_\tau e_{\mathbf{u}}^1\right).
		\end{aligned}
	\end{equation}
	It is easy to show that
	\begin{equation}
		\|\delta_\tau e_{{\mathbf{u}}}^{1}\|^2\leq C\left(\tau^2+h^{2(r+1)}\right).
	\end{equation}
 	 We estimate the term $\tau^2\|\nabla\delta_\tau e_p^{1}\|^2$ and then recast \eqref{eq_error_equations_u_p_cor} into 
 	\begin{equation}
 		\label{eq_error_equations_u_p_cor_recast}
 		\tau^2\nabla\delta_\tau e_p^{n+1}=\tau R_3^{n+1}+e_{\tilde{\mathbf{u}}}^{n+1}-e_\mathbf{u}^{n+1}.
 	\end{equation}
 	Taking the inner product of \eqref{eq_error_equations_u_p_cor_recast} with $\nabla\delta_\tau e_p^{n+1}$
 	when $n=0$  and using the $\nabla\cdot e_{\mathbf{u}}^{n+1}=0$,
 	we get
 	\begin{equation}
 		\begin{aligned}
 			\tau^2\|\nabla\delta_\tau e_p^{1}\|^2=&\left(\tau R_3^{1},\nabla\delta_\tau e_p^{1}\right)+\left(e_{\tilde{\mathbf{u}}}^{1}-e_{\mathbf{u}}^1,\nabla\delta_\tau e_p^{1}\right)\\
 			\leq &~C\left(\|e_{\tilde{\mathbf{u}}}^{1}\|^2+\|e_{\mathbf{u}}^1\|^2+h^{2(r+1)}+\tau^2\right)\leq~C\left(h^{2(r+1)}+\tau^2\right).
 		\end{aligned}
 	\end{equation}
 	Using the Lemma \ref{lemma_discrete_Gronwall_inequation} and substituting the above error estimates into \eqref{eq_summing_up_the_above_error_estimates_reslut}, we can obtain
 	\begin{equation}
 		\begin{aligned}
 			\|\delta_{\tau}e_{\mathbf{u}}^{m+1}\|^2+&\nu\tau\sum_{n=1}^{m}\|\nabla\delta_\tau e_{\tilde{\mathbf{u}}}^{n+1}\|^2+\tau^2\|\nabla\delta_\tau e_p^{m+1}\|^2
 					\leq~C\left(h^{2(r+1)}+\tau^2\right).
 		\end{aligned}
 	\end{equation}
 	Next, the pressure estimate is proved in the following. Recalling the above equation \eqref{eq_combining_error_eqautions_and_projections_term}, we have
 	\begin{equation}
 			\label{eq_combining_error_eqautions_and_projections_term_recast}
 		\begin{aligned}
 			\left(\nabla e_p^{n+1},\mathbf{v}_h\right)=&-\left(\delta_\tau e_{{\mathbf{u}}}^{n+1},\mathbf{v}_h\right)-\nu\left(\nabla e_{\tilde{\mathbf{u}}}^{n+1},\nabla\mathbf{v}_h\right)
 			-\left(\delta_\tau\left(\mathbf{u}^{n+1}-\mathbf{P}_h\mathbf{u}^{n+1}\right),\mathbf{v}_h\right)\\
 			&-\left(\frac{\rho^{n+1}}{\sqrt{E_{1}^{n+1}}}\left(\mathbf{u}^{n+1}\cdot\nabla\mathbf{u}^{n+1},\mathbf{v}_h\right)
 			-\frac{\rho^n}{\sqrt{E_{1,h}^{n+1}}}\left(\mathbf{u}_h^{n}\cdot\nabla\mathbf{u}_h^{n},\mathbf{v}_h\right)\right)\\
 			&-\left(b\left(\phi_h^{n+1},\mu_h^{n+1},\mathbf{v}_h\right)-b\left(R_h\phi^{n+1},\mu^{n+1},\mathbf{v}_h\right)\right)+
 			\left(R_2^{n+1},\mathbf{v}_h\right)+\left(R_3^{n+1},\mathbf{v}_h\right).\\
 		\end{aligned}
 	\end{equation}
 	According to the above error estimates, we get the fact that
 	\begin{equation}
 		\|e_p^{n+1}\|\leq C\sup_{\mathbf{v}_h \in \mathbf{X}_h^{r+1} }\frac{\left(\nabla e_p^{n+1},\mathbf{v}_h\right)}{\|\nabla\mathbf{v}_h\|}.
 	\end{equation} 
 	Thus we derive that
 	\begin{equation}
 		\begin{aligned}
 			\tau\sum_{n=0}^{m}\|e_p^{n+1}\|^2\leq &~C\tau\sum_{n=0}^{m}\left(\|\delta_\tau e_\mathbf{u}^{n+1}\|^2+\|\nabla e_{\tilde{\mathbf{u}}}^{n+1}\|^2
 				+\|e_\mathbf{u}^n\|^2+\|e_\mu^n\|^2+\|\nabla e_\mu^n\|^2\right)\\
 				&+C\tau\sum_{n=0}^{m}\left(\|e_\phi^n\|^2+\|\nabla e_\phi^n\|^2+|e_r^{n+1}|^2\right)+C\left(h^{2(r+1)}+\tau^2\right).
 		\end{aligned}
 	\end{equation}
 	Indrawing to a close, we have completed the proof of the pressure error estimate. Thus, the error estimates of Theorem \ref{theorem_error_estimates_u_phi_mu_r} holds.
	\begin{figure}[htbp]
		\centering
		\includegraphics[width=1\linewidth]{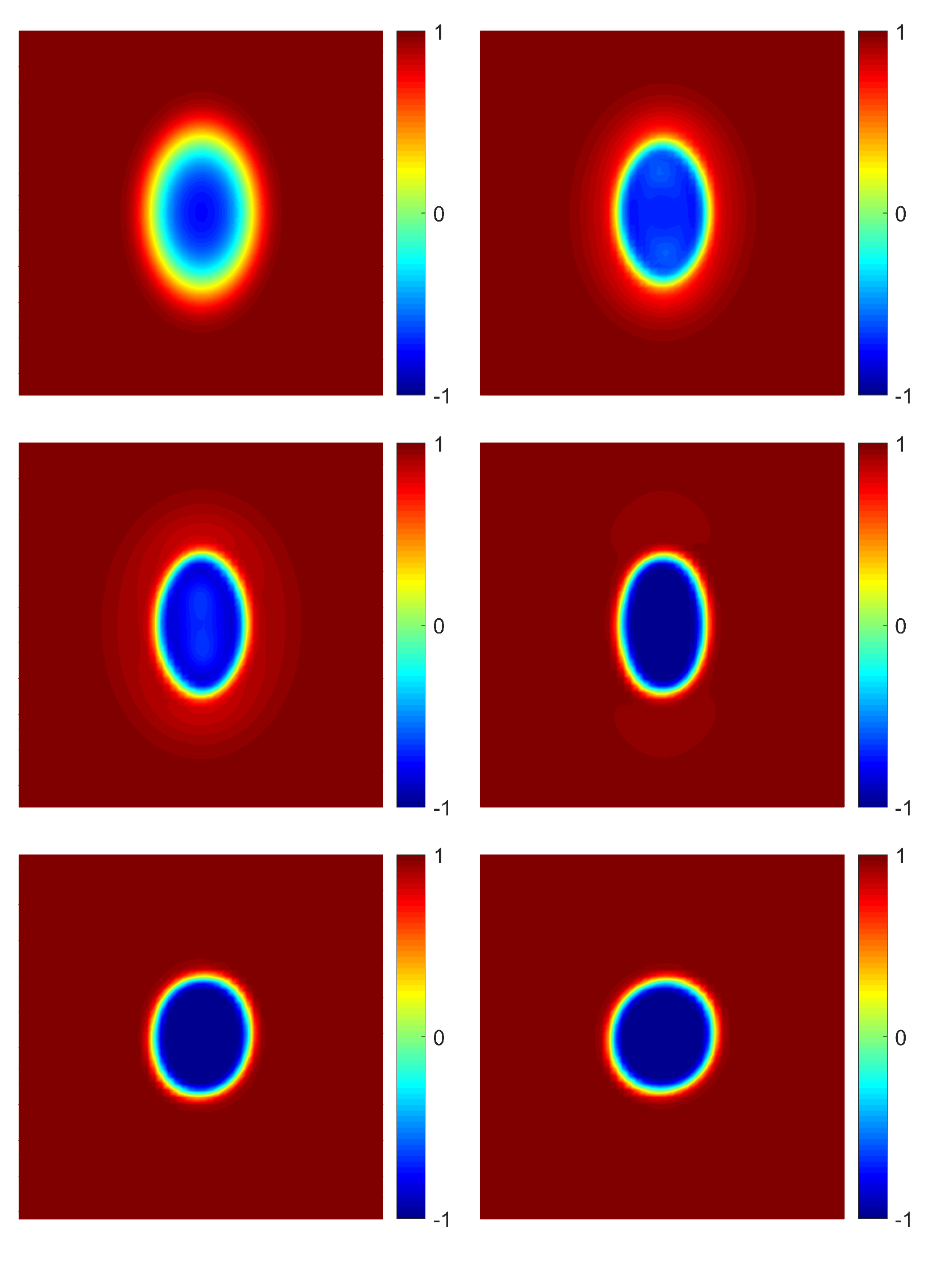}
		\caption{Snapshots of phase function at difference times from left to right row by row with $t=0, 5\!\times\!10^{-6}, 10^{-5}, 5\!\times\!10^{-5}, 5\!\times\!10^{-4}, 10^{-3}$, respectively.}\label{Figure-DegenerateDyna-Elliptical}
	\end{figure}

	\section{Numerical experiments}\label{section_numerical_experiments}
	In this section, three numerical experiments are performed for the corresponding three purposes: 
1)- validating convergence rates in Theorem \ref{theorem_error_estimates_u_phi_mu_r}; 
2)- simulating the dynamics of a elliptical shape fluid, and 
3)- showing numerically the energy stability of our scheme. 
In all tests, we adopt $S^1_h \times S^1_h \times \mathbf{X}^2_h \times S^1_h$ finite element spaces for phase function, chemical potential, velocity, and pressure, respectively; in other words, $P_1$ function space for both $\phi^{n+1}_h$ and $\mu^{n+1}_h$, and $\mathbf{P}_2\times P_1$ for $(\mathbf{u}^{n+1}_h, p^{n+1}_h)$. 
\subsection{Convergence tests}\label{example_CT}
The first test considers the following Cahn–Hilliard–Navier–Stokes equations
\begin{equation}
	\label{eq_chns_recast_ca}
	\left\{\begin{aligned}
		\partial_t \phi+\boldsymbol{u} \cdot \nabla \phi-M \lambda \Delta \mu=0, & & \text { in } \Omega \times[0, T], \\
		\mu+\lambda \Delta \phi-F^{\prime}(\phi)=0, & & \text { in } \Omega \times[0, T], \\
		\partial_t \boldsymbol{u} - \nu \Delta \boldsymbol{u} + \boldsymbol{u} \cdot \nabla \boldsymbol{u} + \nabla p - \mu \nabla \phi=\mathbf{0}, & & \text { in } \Omega \times[0, T], \\
		\nabla \cdot \boldsymbol{u}=0, & & \text { in } \Omega \times[0, T],
	\end{aligned}\right.
\end{equation} 
and the corresponding parameters will be given in the each examples. The modified discrete energy is
\begin{equation}
	\label{eq_modified_discrete_energy}
	\tilde{E}_h^n(\phi,\mathbf{u},p,\rho)=\frac{1}{2}\int_{\Omega}\left(\lambda|\nabla\phi_h^n|^2+|\mathbf{u}_h^n|^2+\tau^2|\nabla p_h^n|^2\right)dx+\lambda|\rho_h^n|^2.
\end{equation}
where $\Omega = [0,1]\times[0,1]$ and time $T = 0.01$, and the physical parameters are selected as $M = 0.1$, $\lambda = 0.04$, $\epsilon = 0.2$ and $\nu = 0.01$. The time-step $\tau$ is set up as $\tau = h^3$ to balance the convergence rates between time and space.
The manufactured solutions are chosen from \cite{2022_ChenYaoyao_CHNS_2022_AMC} and takes the form
$$\left\{
\begin{aligned}
	\phi(t,x,y) &= 2 + \sin(t)\cos(\pi x)\cos(\pi y),  \\
	\mathbf{u}(t,x,y) &= \left[\pi\sin(\pi x)^2\sin(2\pi y)\sin(t), -\pi\sin(\pi y)^2\sin(2\pi x)\sin(t)\right]^{\text{T}},  \\
	p(t,x,y) &= \cos(\pi x)\sin(\pi y)\sin(t),
\end{aligned}\right.
$$
and the exact chemical potential $\mu(t,x,y)$ is obtained by its definition, i.e., 
$$
\begin{aligned}
	\mu(t,x,y) &:= -\lambda \Delta \phi(t,x,y) + F^{\prime}\big(\phi(t,x,y)\big) \\
	&=  -\lambda \Delta \phi(t,x,y) + \frac{1}{\epsilon^2} \left(\phi(t,x,y)^3 - \phi(t,x,y)\right) \\
	&= 2\lambda \pi^2 \cos(\pi x)\cos(\pi y) \sin(t) - \frac{1}{\epsilon^2}\left[\cos(\pi x)\cos(\pi y) \sin(t) - \big(\cos(\pi x)\cos(\pi y) \sin(t) + 2\big)^3 + 2 \right].
\end{aligned}
$$
For simplicity of notations, we adopt the following notations:
$$
\begin{aligned}
	\|\cdot\|_{\ell^{\infty}(L^2)} := \max\limits_{0\leq n \leq N}\|\cdot\|_{L^2}, \quad 
	|\cdot|_{\ell^{\infty}} := \max\limits_{0\leq n \leq N}|\cdot|, \quad \|\cdot\|_{\ell^{2}(L^2)} := \sqrt{\tau \sum_{n=0}^{N}\|\cdot\|_{L^2}}.
\end{aligned}
$$
Therefore, according to Theorem 4.1, the following theoretically optimal orders should be observed:
$$
\begin{aligned}
	\left\|\phi - \phi_h\right\|_{\ell^{\infty}(L^2)} & \approx  \mathcal{O}(h^2), \qquad \left\|\mu - \mu_h\right\|_{\ell^2(L^2)}  \approx  \mathcal{O}(h^2), \\
	\left\|\boldsymbol{u} - \boldsymbol{u}_h\right\|_{\ell^{\infty}(L^2)} & \approx   \mathcal{O}(h^3), \qquad \left\|p - p_h\right\|_{\ell^2(L^2)}  \approx  \mathcal{O}(h^2), \\
	\left\|\nabla(\boldsymbol{u} - \boldsymbol{u}_h)\right\|_{\ell^{\infty}(L^2)} & \approx  \mathcal{O}(h^2), \qquad \left|\rho - \rho_h\right|_{\ell^{\infty}} \approx  \mathcal{O}(\tau).
\end{aligned}
$$
\autoref{L2-error-convRates-FirstOrder} and \autoref{H1-error-convRates-FirstOrder} present the errors and convergence rates, where the temporal orders of SAV are presented by spatial ones after setting $\tau = \mathcal{O}(h^3)$. From those tables, we can conclude that the numerical results are consistent with the above \emph{a prior} rates, and hence the scheme proposed here can indeed obtain the optimal convergence rates.   
\begin{table}[h]
	\centering
	\fontsize{10}{10}
	\begin{threeparttable}
		\caption{$L^2$ errors and convergence orders.}\label{L2-error-convRates-FirstOrder}
		\begin{tabular}{c|c|c|c|c|c|c|c|c}
			\toprule
			\multirow{2.5}{*}{$h$}  & \multicolumn{2}{c}{$\left\|\phi - \phi_h\right\|_{\ell^{\infty}(L^2)}$} & \multicolumn{2}{c}{$\left\|\mu - \mu_h\right\|_{\ell^2(L^2)}$} & \multicolumn{2}{c}{$\left\|\boldsymbol{u} - \boldsymbol{u}_h\right\|_{\ell^{\infty}(L^2)}$} & \multicolumn{2}{c}{$\left\|p - p_h\right\|_{\ell^2(L^2)}$} \cr
			\cmidrule(lr){2-3} \cmidrule(lr){4-5} \cmidrule(lr){6-7}  \cmidrule(lr){8-9}
			& error & rate & error & rate & error & rate & error & rate \cr
			\midrule
			1/4   & 1.7852$-$04 &  -       & 1.2149$-$05  & -       & 5.6137$-$04  & -      & 1.8790$-$03 & -      \cr
			1/8   & 2.1368$-$05 &  3.0626  & 1.3138$-$06  & 3.2091  & 9.8455$-$05  & 2.5114 & 1.9000$-$04 & 3.3059 \cr
			1/16  & 3.1612$-$06 &  2.7569  & 1.8802$-$07  & 2.8048  & 1.1181$-$07  & 3.1384 & 2.3500$-$05 & 3.0153 \cr
			1/32  & 6.6269$-$07 &  2.2541  & 3.6619$-$08  & 2.3602  & 8.1548$-$07  & 3.7773 & 3.8771$-$06 & 2.5996 \cr
			1/64  & 1.6764$-$07 &  1.9829  & 8.9149$-$09  & 2.0383  & 5.4351$-$08  & 3.9073 & 9.1833$-$07 & 2.0779 \cr
			\bottomrule
		\end{tabular}
	\end{threeparttable}
\end{table}
\begin{table}[h]
	\centering
	\fontsize{10}{10}
	\begin{threeparttable}
		\caption{$H^1$ errors and convergence orders for velocity and temporal convergence orders for SAV (presented by spatial orders).}\label{H1-error-convRates-FirstOrder}
		\begin{tabular}{c|c|c|c|c}
			\toprule
			\multirow{2.5}{*}{$h$} & \multicolumn{2}{c}{$\left\|\nabla(\boldsymbol{u} - \boldsymbol{u}_h)\right\|_{\ell^{\infty}(L^2)}$} & \multicolumn{2}{c}{$\left|\rho - \rho_h\right|_{\ell^{\infty}}$} \cr
			\cmidrule(lr){2-3} \cmidrule(lr){4-5}   
			& error & rate & error & rate \cr
			\midrule
			1/4   \qquad&\quad 3.3592$-$02 $\quad$&\quad  -       $\quad$&\quad 3.0022$-$06  $\quad$&\quad -        \cr
			1/8   \qquad&\quad 1.2206$-$02 $\quad$&\quad  1.4604  $\quad$&\quad 4.2887$-$07  $\quad$&\quad 2.8074   \cr
			1/16  \qquad&\quad 2.7855$-$03 $\quad$&\quad  2.1316  $\quad$&\quad 5.5682$-$08  $\quad$&\quad 2.9452   \cr
			1/32  \qquad&\quad 4.0910$-$04 $\quad$&\quad  2.7674  $\quad$&\quad 6.9891$-$09  $\quad$&\quad 2.9940   \cr
			1/64  \qquad&\quad 5.5197$-$05 $\quad$&\quad  2.8898  $\quad$&\quad 8.7406$-$10  $\quad$&\quad 2.9993   \cr
			\bottomrule
		\end{tabular}
	\end{threeparttable}
\end{table}
\begin{figure}[htbp]
	\centering
	\includegraphics[width=1\linewidth]{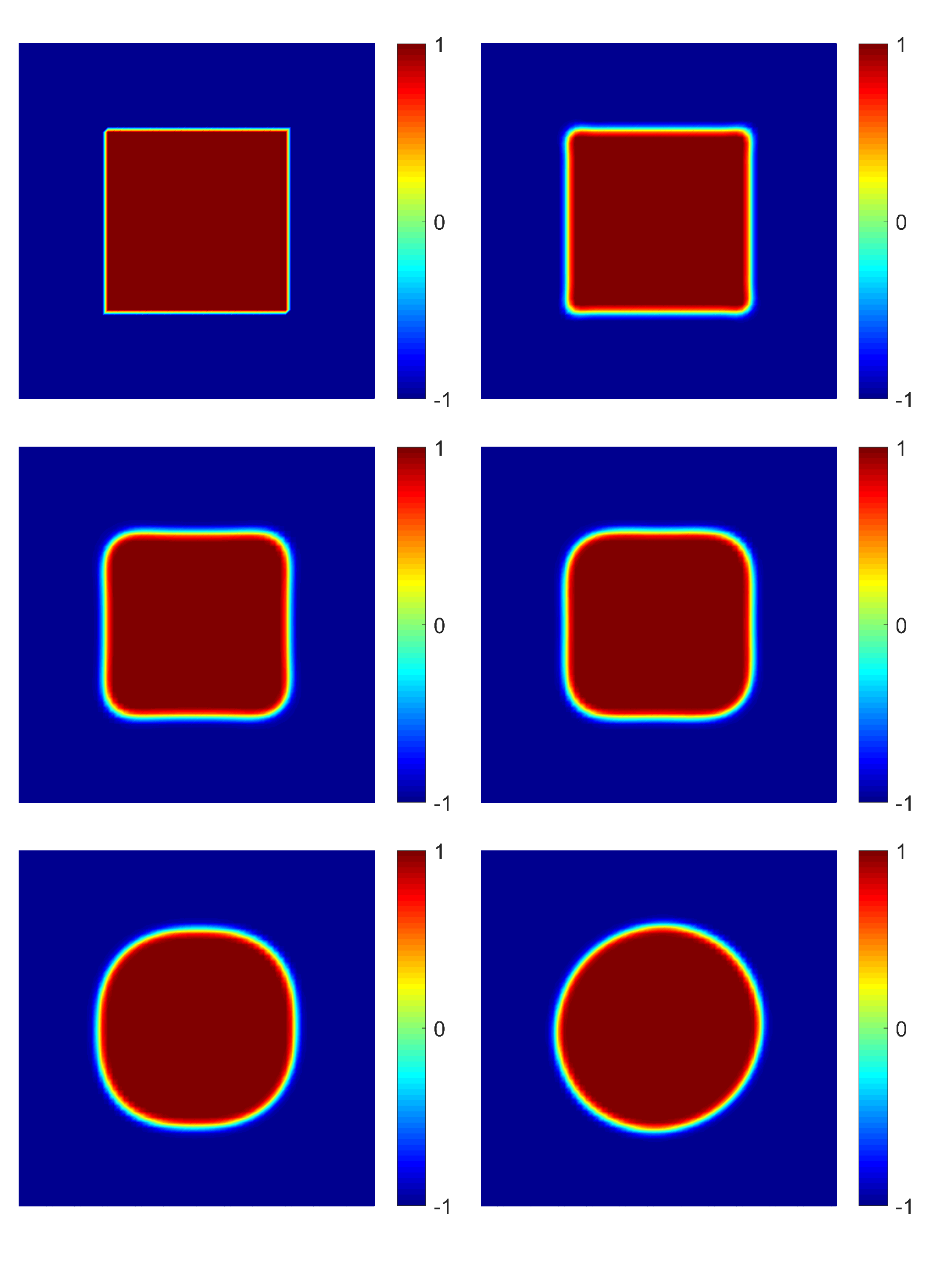}
	\caption{Snapshots of phase function at difference times from left to right row by row with $t=0, 0.001, 0.03, 0.08, 0.3, 1$, respectively.}\label{Figure-DegenerateDyna-Square}
\end{figure}
\subsection{Dynamics of a Elliptical Shape Fluid}\label{example_DESF}
In this example, we consider the Cahn–Hilliard–Navier–Stokes \eqref{eq_chns_recast_ca} to simulate the dynamics of a elliptical shape fluid \cite{2006_FengXiaobing_FullydiscretefiniteelementapproximationsoftheNavierStokesCahnHilliarddiffuseinterfacemodelfortwophasefluidflows}. We choose $\Omega = [-0.4,0.4]^2$, and $M = 0.1$, $\lambda = 0.1$, $\epsilon = 0.01$, and $\nu = 1$, with the initial phase function as
$$
\phi_0 = \tanh\left(\frac{x^2}{0.01} + \frac{y^2}{0.0225} - 1\right), 
$$ 
and the right term $\boldsymbol{f} = [1,0]^{\text{T}}$ is applied to the RHS of the momentum equation. The spatial length $h=1/64$, and $\tau = 10^{-7}$. We run this test up to final time $T=10^{-3}$, and record the snapshots of phase function at $t=0, 5\!\times\!10^{-6}, 10^{-5}, 5\!\times\!10^{-5}, 5\!\times\!10^{-4}, 10^{-3}$, respectively, in \autoref{Figure-DegenerateDyna-Elliptical}. We notice, from this figure, that the elliptical bubble quickly deforms into a circular bubble, which, as expected, can be thought of as approaching the equilibrium state.

\subsection{Dynamics of a Square Shape Fluid}\label{example_DSSF}
In this example, we consider the Cahn–Hilliard–Navier–Stokes \eqref{eq_chns_recast_ca} to simulate the dynamics of a elliptical shape fluid \cite{2020_LiXiaoli_On_a_SAV_MAC_scheme_for_the_Cahn_Hilliard_Navier_Stokes_phase_field_model_and_its_error_analysis_for_the_corresponding_Cahn_Hilliard_Stokes_case_}. We choose $\Omega = [0,1]^2$, and $M = 0.002$, $\lambda = 0.1$, $\epsilon = 0.01$, and $\nu = 1$, with the initial phase function as
$$
\phi_0(x, y)=\left\{\begin{aligned}
	1, \quad  & \text { if } 0.25 \leq x \leq 0.75 ~\& ~0.25 \leq y \leq 0.75, \\
	-1, \quad & \text { if otherwise. }
\end{aligned}\right.
$$
The spatial length $h=1/64$, and $\tau = 10^{-5}$. We run this test up to final time $T=1$, and record the snapshots of phase function at $t=0, 0.001, 0.03, 0.08, 0.3, 1$, respectively, in \autoref{Figure-DegenerateDyna-Square}. We notice, from this figure, that the elliptical bubble quickly deforms into a circular bubble, which, as expected, can be thought of as approaching the equilibrium state.

\indent Moreover, the history curve of the modified energy defined in \eqref{eq_modified_discrete_energy} for example \ref{example_DESF} and in \eqref{eq_modified_discrete_energy} for example \ref{example_DSSF} are showed in \autoref{EnergyStab_Degenerate_Ellip_Squa}. It demonstrates that the modified energy decreases.
\begin{figure}[htbp]
	\centering
	\includegraphics[width=1\linewidth]{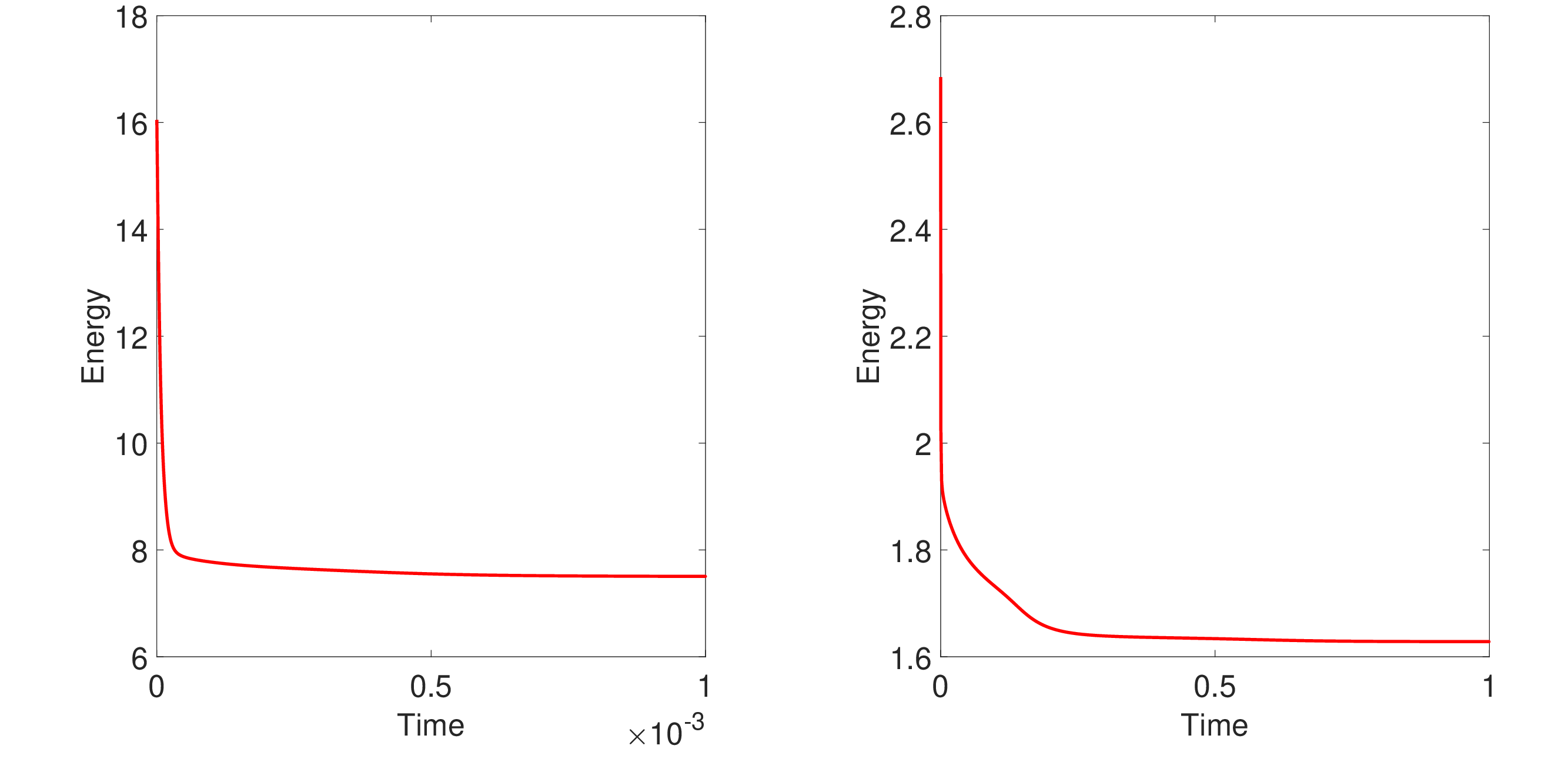}
	\caption{Evolution of the modified energy, left: example \ref{example_DESF}; right: example \ref{example_DSSF} .}\label{EnergyStab_Degenerate_Ellip_Squa}
\end{figure}

	\section{Conclusions}\label{section_conclusion}
	In this work, we  proposed a decoupling method for the CHNS model which is the highly coupled nonlinear system. 
	The method is first-order, fully discrete and unconditional energy stabilization, and is very the efficient and  
	easy to implement.
	Moreover, we carried out a rigorous error analysis for the fully dicrete scheme  and derived optimal error estimates for all relevant functions in $L^2$ norm.
	The above numerical experiments are presented to verify the theoretical results of the our method.
	
	\section*{Acknowledgments}
	This research is supported by the National Natural Science Foundation of China (No.11971337).

\bibliographystyle{unsrt}
\bibliography{../document/bibfile}
\end{document}